\documentclass[11pt, english, reqno]{amsart}

\usepackage[T1]{fontenc}
\usepackage{babel}
\usepackage{mathrsfs}
\usepackage{amsthm}
\usepackage{graphs}

\usepackage[dvips,top=3.8cm,left=4cm,right=3.8cm,
foot=3.8cm,bottom=4.0cm]{geometry}
\usepackage{amsfonts,amssymb,amsmath,amsthm, booktabs,
latexsym}

\usepackage{color}

\theoremstyle{definition}
\newtheorem{defin}{Definition}[section]
\theoremstyle{remark}
\newtheorem{example}[defin]{Example}

\newtheorem{remar}[defin]{Remark}
\newtheorem{remark}[defin]{Remark}
\theoremstyle{plain}
\newtheorem{thm}[defin]{Theorem}
\newtheorem{prop}[defin]{Proposition}

\newtheorem{lemma}[defin]{Lemma}
\newtheorem{corol}[defin]{Corollary}

\def\wt{\widetilde}
\def\ol{\overline}
\def\ov{\overline}
\def\tl{\widetilde}
\def\wh{\widehat}

\numberwithin{equation}{section}

\begin{document}

\title{Exact number of ergodic invariant measures for Bratteli diagrams}

\date{}

\author{Sergey Bezuglyi}
\address{Department of Mathematics, University of Iowa, Iowa City,
52242 IA, USA}
\email{sergii-bezuglyi@uiowa.edu}

\author{Olena Karpel}
\address{B. Verkin Institute for Low Temperature Physics and Engineering,
Kharkiv, Ukraine; Department of Dynamical Systems, Institute of Mathematics of
Polish Academy of Sciences, Wroclaw, Poland}

\curraddr{Faculty of Applied Mathematics, AGH University of Science and Technology
 Krakow, Poland}

\email{helen.karpel@gmail.com}

\author{Jan Kwiatkowski}

\address{Kotarbinski University of Information Technology and Management,
Olsztyn, Poland}

\email{jkwiat@mat.umk.pl}

\begin{abstract}
For a Bratteli diagram $B$, we study the simplex $\mathcal{M}_1(B)$ of
 probability measures on the path space of $B$ 
which are invariant with respect to the tail equivalence relation. Equivalently, 
$\mathcal{M}_1(B)$ is formed by  probability measures  invariant with
 respect to a homeomorphism of a Cantor set.  
We study relations between the number 
of ergodic measures from  $\mathcal{M}_1(B)$  and the structure and 
properties  of the diagram $B$. 
We prove a criterion and find sufficient conditions  of unique ergodicity of a
 Bratteli diagram, in which case the simplex 
$\mathcal{M}_1(B)$ is a singleton. For   a 
finite rank $k$ Bratteli diagram $B$ having exactly $l \leq k$ 
ergodic invariant measures, we explicitly describe the structure of the
 diagram and  find the subdiagrams  which support these measures. We find
sufficient  conditions under  which: (i) a Bratteli diagram has a prescribed
 number (finite or infinite) of ergodic invariant measures, and 
  (ii) the extension of a measure from a uniquely ergodic subdiagram gives a finite ergodic invariant measure.  Several
 examples, including stationary Bratteli diagrams, Pascal-Bratteli diagrams, and
 Toeplitz flows, are  considered.
\end{abstract}
\maketitle
\tableofcontents

\section{Introduction}
The concept of a Bratteli diagram appeared first in the famous paper by
O. Bratteli  \cite{Bratteli1972} devoted to the classification of approximately 
finite $C^*$-algebras (named AF-algebras).
Bratteli proved that AF-algebras can be completely described
combinatorially by introducing a certain type of graphs, now called
Bratteli diagrams.

Though Bratteli diagrams were introduced to answer a challenging problem
in  the  $C^*$-algebra theory, these diagrams have  found a surprisingly
 large number of diverse applications in related fields.  Among them are:
 ergodic  theory and symbolic dynamics, representation theory, classification
problems,  analysis and random walk in network models, etc.

In the present article, we use Bratteli diagrams to study ergodic invariant
measures of Cantor dynamical systems $(X, T)$ generated by a single
homeomorphism of a Cantor set $X$. Our approach  is based on the
existence of Bratteli-Vershik models  for aperiodic
(minimal) homeomorphisms of Cantor sets.
Loosely speaking, every aperiodic homeomorphism $T$ of a Cantor set $X$
is conjugate to a homeomorphism $\varphi_B$ (called the Vershik map)
of a path space $X_B$ of a Bratteli diagram $B$. This means that 
 all properties of a dynamical system can be seen in terms of the corresponding  Bratteli diagram.

This concept was successfully realized in the series of remarkable  papers
\cite{Vershik1981, Vershik1982, HermanPutnamSkau1992,
 GiordanoPutnamSkau1995} in the context of  ergodic theory and minimal
 Cantor dynamics.  Later on,
this approach was extended to Borel dynamical systems and aperiodic
homeomorphisms of a Cantor set in
\cite{BezuglyiDooleyKwiatkowski2006, Medynets2007}.

The idea to study  the Vershik map $\varphi_B$ defined on the path
space $X_B$ of the corresponding Bratteli diagram  proved to be very useful
and productive. We recall that, first of all, this approach allowed to classify
 minimal  homeomorphisms of a Cantor set up to orbit equivalence
  \cite{GiordanoPutnamSkau1995, GiordanoMatuiPutnamSkau2010}.
 Furthermore, it turns out that the structure  of a Bratteli diagram makes it
 possible to see  distinctly on the diagram several important invariants of
 a homeomorphism. They are,  for example, the set of minimal components,
  the support of any ergodic measure
 $\mu$,  and the values of  $\mu$  on clopen sets. We discussed various
 aspects of this  method in a series of papers
 \cite{BezuglyiKwiatkowskiMedynetsSolomyak2010, BezuglyiKarpel2011,
 BezuglyiKwiatkowskiMedynetsSolomyak2013, BezuglyiHandelman2014,
 BezuglyiKwiatkowskiYassawi2014, BezuglyiJorgensen2015,
   BezuglyiKarpel2016}. Other important references on Bratteli diagrams in
   Cantor   dynamics are \cite{DurandHostSkau1999, Durand2010,
   GjerdeJohansen2000, Cortez2006, HamachiKeaneYuasa2011,
   KwiatkowskiWata2004,  Putnam2010, Skau2000}.

Let $(X, T)$ be a Cantor dynamical system. The question about a complete
 (or even partial) description of the simplex
$\mathcal{M}_1(T)$ of invariant probability measures for  $(X, T)$ is
one of the most important in ergodic theory. It has
a long history and many remarkable results. By the Kakutani-Markov
 theorem,
this simplex is always non-empty. Ergodic invariant measures form extreme
points of the simplex. The cardinality of the set of ergodic measures is an
important invariant of dynamical systems. The study of relations
 between the properties  of the simplex $\mathcal{M}_1(T)$ and those of
  the  dynamical system $(X,T)$ is a hard and intriguing problem. There is
an extensive list of references regarding this problem, we mention here only
 the books  \cite{Phelps2001, Glasner2003} and the papers
  \cite{Downarowicz1991,  Downarowicz2006, Downarowicz2008} for
  further  citations.

It is worth  pointing  out two important facts that are constantly used in
the paper. Firstly, we consider only  Bratteli-Vershik models to study  the
simplex $\mathcal{M}_1(T)$ of probability invariant measures. More
precisely, we work  in a more general setting considering
probability measures invariant with respect to the
\textit{tail equivalence relation} $\mathcal E$ on the path space $X_B$ of
a Bratteli diagram. We denote this set by $\mathcal{M}_1(B)$. It is not
hard to see that such measures are also invariant with
respect to the Vershik map acting on $X_B$ 
\cite{BezuglyiKwiatkowskiMedynetsSolomyak2010} if the diagram $B$ 
admits a Vershik map (see \cite{Medynets2007}, 
\cite{BezuglyiKwiatkowskiYassawi2014}, \cite{BezuglyiYassawi2017} for 
more details). Hence,  this approach
includes the case of dynamical systems generated by a transformation.
Moreover, it can be used  not only for homeomorphisms of  a Cantor set
but for measure preserving automorphisms of a standard measure space.
Secondly, the structure of a Bratteli diagram is completely described by the 
 sequence of matrices $(\tl F_n)$
 with non-negative integer entries. Such matrices are called \textit{incidence
 matrices}. Knowing entries of the incidence matrices, one can determine
 the number $h^{(n)}_v$ of all finite paths between the top of the diagram
 and any vertex $v$ of the level $n$: they are obtained as entries of
  the products of  incidence matrices. This information can be combined
together in the sequence $(F_n)$  of \textit{stochastic incidence matrices}
whose entries are
  $$
  f_{vw}^{(n)} = \frac{\tl f_{vw}^{(n)} h^{(n)}_w} {h^{(n+1)}_v},
  $$
  where $\tl f_{vw}^{(n)}$ are entries of the incidence matrix
  $\tl F_n$.
It turns out that namely these matrices $(F_n)$  carry the information
needed to describe the structure of the  simplex $\mathcal{M}_1(B)$.

In this paper, we focus mostly on the following problem: given a Bratteli
diagram $B$, determine the exact number of ergodic probability measures 
on $B$ which are invariant with respect to the tail equivalence relation
 $\mathcal E$. This problem
includes also the case of uniquely ergodic homeomorphisms for which
there exists just one ergodic probability measure. In our earlier papers, we
proved a number of results for stationary (i.e., all incidence matrices are
equal) and finite rank (i.e., the size of all incidence matrices is bounded)
Bratteli diagrams which  give  partial answers about the structure of the 
set $\mathcal{M}_1(B)$, see
\cite{BezuglyiKwiatkowskiMedynetsSolomyak2010,
BezuglyiKwiatkowskiMedynetsSolomyak2013, 
AdamskaBezuglyiKarpelKwiatkowski2016}. Especially important for us 
is Theorem \ref{Theorem_measures_general_case} (see below) that was
 proved in  \cite{BezuglyiKwiatkowskiMedynetsSolomyak2010}.  
 We recall that in the case of
stationary and finite rank Bratteli diagrams the simplex $\mathcal{M}_1(B)$
is finite-dimensional, and the number of ergodic measures cannot exceed the
rank of $B$.

Our main results of this paper are given in Theorems  \ref{uniq_erg},
 \ref{main1}, and
 \ref{main_gen_case}.  Theorem  \ref{uniq_erg}  contains a  criterion for
 unique ergodicity of a
  Bratteli diagram. This condition is  formulated in terms of the stochastic
 matrices $F_n$ associated with a Bratteli diagram.  The theorem 
 asserts that $B$ is uniquely ergodic if and only if
$$
\lim_{n\to \infty}\ \max_{v, v' \in V_{n+1}} \left(\sum_{w \in V_n}
 \left|f_{vw}^{(n)} - f_{v'w}^{(n)}\right| \right) =0.
$$
Another sufficient condition for unique ergodicity of a diagram $B$
 is given in Theorem \ref{1.6a}. 
In Theorem \ref{main1}, we consider the case when a finite rank $k$ Bratteli
 diagram $B$
has exactly $l, 1\leq l \leq k,$ ergodic measures. It turns out that,
in this case,  one can clarify the structure of the diagram $B$ and describe
subdiagrams that support ergodic measures. Theorem \ref{main_gen_case}
focuses on the question about finiteness of measure extension from a
subdiagram. More precisely,
we find conditions that guarantee the finiteness of measure extension
from a ``chain subdiagram''. It gives a description  of ergodic finite
measures
on $B$ obtained by extension from  uniquely ergodic subdiagrams.

The outline of  the paper is as follows. In Section \ref{sect 2} we give 
necessary definitions and notation. We also prove a few preliminary 
statements which are based on Theorem 
\ref{Theorem_measures_general_case} and the construction of 
Kakutani-Rokhlin towers related to a Bratteli diagram. 
Section \ref{Section 3} contains the proof of 
Theorem \ref{uniq_erg}. We also included in this section  Example  
\ref{ex 3.2} which shows the importance of the telescoping procedure 
for the study of invariant measures. Sections \ref{sect 4} and \ref{sect 5}
 form the core of our work. The proofs of our main results, Theorems
\ref{main1} and \ref{main_gen_case}, are given in these two sections. 
We should remark that the proofs are rather long and  difficult. They are
based on the technique developed in our previous papers, see e.g. 
\cite{BezuglyiKwiatkowskiMedynetsSolomyak2013} and 
\cite{AdamskaBezuglyiKarpelKwiatkowski2016}.

The last two sections are devoted to some applications. In Section 
\ref{sect 6} we consider several particular cases. Firstly, we show that, 
in case of stationary Bratteli diagrams, the description of invariant 
measures given in 
\cite{BezuglyiKwiatkowskiMedynetsSolomyak2010}
can be included in the scheme elaborated in Sections \ref{sect 5} and   
\ref{sect 6}. In two other subsections of Section \ref{sect 6}, we
 discuss the Pascal-Bratteli diagram and the class of simple Bratteli 
 diagrams with countably many 
 ergodic invariant measures. We remark that it is not hard to see that
  every  Bratteli-Vershik
diagram,  considered in Subsection~\ref{subsection_cntbl_erg_meas},
 has zero topological entropy. It follows from the N.~Frantzikinakis
and B.~Host result \cite{FrantzikinakisHost2017} that a class of topological
dynamical systems constructed in Subsection~\ref{subsection_cntbl_erg_meas} satisfy the logarithmic Sarnac conjecture.
The last  section  contains an interpretation of our results 
 in terms of symbolic dynamics.

\section{Basics on Bratteli diagrams and invariant measures}\label{sect 2}

In this section, we give necessary definitions, notation, and prove some
results about the structure of the set of invariant measures which are used
 in the paper below.

\subsection{Main definitions and notation}

\begin{defin}
 A {\it Bratteli diagram} is an infinite graph $B=(V,E)$ such that the
  vertex set  $V =\bigcup_{i\geq 0}V_i$ and the edge set
  $E = \bigcup_{i\geq 0}E_i$ are   partitioned into disjoint subsets
  $V_i$ and $E_i$ where

(i) $V_0=\{v_0\}$ is a single point;

(ii) $V_i$ and $E_i$ are finite sets, $\forall i \geq 0$;

(iii) there exist $r : E \to V$ (range map $r$) and
$s : E \to V$ (source map  $s$), both from $E$ to $V$, such that
 $r(E_i)= V_{i+1}$, $s(E_i)= V_{i}$,
  and $s^{-1}(v)\neq\emptyset$, $r^{-1}(v')\neq\emptyset$ for all
   $v\in V$  and $v'\in V\setminus V_0$.

\end{defin}

The set of vertices $V_i$ is called the \textit{$i$-th level} of the
diagram $B$. A finite or infinite sequence of edges $(e_i : e_i\in E_i)$
such that $r(e_{i})=s(e_{i+1})$ is called a {\it finite} or {\it infinite
path}, respectively. For $m<n$, $v\, \in V_{m}$ and $w\,\in V_{n}$,
let $E(v,w)$ denote the set of all paths $e(v, w) = (e_{1},\ldots,
e_{p})$ with $s(e) = s(e_{1})=v$ and $r(e) = r(e_{p})=w$.

For a Bratteli diagram $B$,
let $X_B$ be the set of all infinite paths beginning at the top vertex
 $v_0$.
 We endow $X_B$ with the topology generated by cylinder sets
$[e]$ where $e = (e_0, ... , e_n)$,  $n \in \mathbb N$, and
$[e]:=\{x\in X_B : x_i=e_i,\; i = 0, \ldots, n\}$. With this topology,
$X_B$ is a  0-dimensional compact metric space.
\textit{By assumption}, we will consider only such Bratteli diagrams $B$ for
which $X_B$ is a {\em Cantor set}, that is $X_B$ has no isolated
 points. 

\begin{remark}\label{defin_and_notat} We will use the following definitions and notation
 throughout  the  paper.
\begin{itemize}

\item By $\ov x$, we denote a (column) vector in
$\mathbb{R}^{k}$. We will always denote the coordinates of $\ov x$ using
 the subscript.
For instance, the notation $\ov x(w) = (x_v(w)) $ means that we have a
set of vectors
 enumerated by $w$, and the $v$-coordinate of  $\ov x(w) $ is $x_v(w)$.
By $||\ov x||$ we  denote the Euclidean norm of $\ov x$.

\item The notation $A^T$ is used for the transpose of $A$, and
$(x_1, ... , x_k)^T$ means the column vector $\ov x$ with entries
 $x_i$ in the standard basis in $\mathbb R^k$.

\item For $\ov x = (x_1, ... , x_k)^T, \ov y = (y_1, ... , y_k)^T \in
\mathbb{R}^k$, define the metric
$$
d^*(\ov x, \ov y) = \sum_{i = 1}^k |x_i - y_i|.
$$
Let $d$ denote the metric on $\mathbb R^k$ generated by the Euclidean
norm $||\cdot||$. Clearly, the two metrics, $d^*$ and $d$, are equivalent.

\item Let $x =(x_n)$ and $y =(y_n)$  be two paths from $X_B$. It is
 said that $x$ and $y$ are {\em tail equivalent} (in symbols,  $(x,y)
  \in   \mathcal E$)  if there exists some $n$ such that $x_i = y_i$ for
   all  $i\geq n$.  An $\mathcal E$-orbit of a path $x$ is the set of all infinite paths $y$ which are tail equivalent to $x$.
   The diagrams with infinite $\mathcal E$-orbits are
    called {\em aperiodic}.  { \it Throughout the paper we will consider Bratteli diagrams whose tail equivalence relation is aperiodic.}
    Note that a Bratteli diagram is \textit{simple}
    if the  tail equivalence relation   $\mathcal E$ is minimal, see Definition
    \ref{rank_d_definition} below.

\item For a Bratteli diagram  $B = (V,E)$,  the incidence matrices $\tl
 F_n   = (\tl  f_{v,w}^{(n)})$, where  $v \in V_{n+1}, w \in V_n$, are
  formed by
  entries that indicate the number of edges between vertices $v \in
   V_{n+1}$ and $ w \in V_n$. We reserve the notation $F_n$ and
   $ f_{v,w}^{(n)}$ for the    corresponding stochastic matrices (see
    below).

\item
Let $B$ be a Bratteli diagram, and let $n_0 = 0 < n_1< n_2 < \ldots$
be a strictly increasing sequence of integers. The {\em telescoping of
 $B$ to $ (n_k)$} is the Bratteli diagram $B' = (V', E')$ whose $k$-level
vertex  set $V_k'$ is  $V_{n_k}$ and whose incidence matrices $(\tl F_k')$
 are defined by
   \[
   \tl F_k'= \tl F_{n_{k+1}-1} \cdots \tl F_{n_k},
   \]
The operation of telescoping preserves the space $X_B$,
the tail equivalence relation, and the set of invariant measures.

\item Denote by $\ov h^{(n)} = (h_v^{(n)} : v \in V_n)$ the  vector of
 heights of  Kakutani-Rokhlin towers \cite{GiordanoPutnamSkau1995,
  HermanPutnamSkau1992}, i.e.,   $h_v^{(n)} = |E(v_0, v)|$ where $|A|$
  denotes the  cardinality of a set $A$. Here $h^{(1)}_{v}$ is
   the number of    edges   between $v_0$ and $v \in V_1$.   It
   follows from the definition of a  Bratteli diagram that
\begin{equation}\label{formula for heights}
\tl F_n \ov h^{(n)} = \ov  h^{(n+1)}, \ \ n \geq 1,
\end{equation}

\item We will constantly use the sequence of row stochastic  matrices
$(F_n)$  whose entries are defined by the formula
\begin{equation}\label{stoch_inc_matr}
f_{vw}^{(n)} = \frac{\tl f_{vw}^{(n)} h_w^{(n)}}{h_v^{(n+1)}}.
\end{equation}
In general, it is difficult to compute the elements of the matrix $F_n$ explicitly because the terms $h_w^{(n)}$ used in the formula (\ref{stoch_inc_matr}) are the entries of the product of matrices. In Section~\ref{Section 3}, the reader can find Examples~\ref{ex 3.2}, \ref{non-simple-stat-ex} of Bratteli diagrams, for which stochastic incidence matrices are computed explicitly. 
  \end{itemize}
\end{remark}

\begin{defin}\label{rank_d_definition} Let $B = (V, E)$ be a
Bratteli diagram.

\begin{enumerate}
\item
We say that $B$ has \textit{finite rank} if there exists some
 $k\in \mathbb N$ such that $|V_n| \leq k$  for all $n\geq 1$.

\item
For a finite rank diagram $B$, we  say that $B$ has
\textit{rank $d$} if  $d$ is the smallest integer such that $|V_n|=d$
 for infinitely many $n$.

\item
We say that $B$ is {\em simple} if for any level
$n$ there is $m>n$ such that $E(v,w) \neq \emptyset$ for all $v\in
V_n$ and $w\in V_m$. Otherwise, $B$ is called {\em non-simple}.

\item
 We say that $B$ is  \textit{stationary} if $\tl F_n = \tl F_1$  for all $n\geq
  2$.
\end{enumerate}
\end{defin}

Note that a Cantor dynamical system $(X, \varphi)$ can be represented, in general, by Bratteli diagrams of different ranks, see Section \ref{Section7}. Throughout this paper we fix a Bratteli diagram and study the number of ergodic invariant measures for the system in terms of this diagram. Since the number of ergodic measures for $(X, \varphi)$ does not depend on the representation, our results will hold if we start with another Bratteli diagram which represents the system.

\subsection{Invariant measures on Bratteli diagrams}

In this paper, we study only positive Borel measures on $X_B$ which
are invariant  with respect to the tail equivalence relation $\mathcal{E}$.
In this subsection, we show that the set of all probability invariant measures
on a Bratteli diagram
corresponds to the inverse limit of a decreasing  sequence of convex sets.
Let $\mu$ be a Borel probability non-atomic $\mathcal E$-invariant
  measure on $X_B$ (for brevity, we will use the term
  ``measure on $B$'' below). We denote the set of all such measures
   by $\mathcal{M}_1(B)$.
   The   fact that $\mu$ is an
   $\mathcal E$-invariant measure
 means that  $\mu([e]) = \mu([e'])$ for any two cylinder sets
 $e,e' \in  E(v_0, w)$, where $w \in  V_n$ is an arbitrary vertex, and
 $n \geq1$.   Since any measure on $X_B$ is completely determined
  by its  values  on  clopen (even cylinder) sets, we conclude that in order to
  define an  $\mathcal E$-invariant measure $\mu$, one needs to
   know the  sequence   of vectors  $\ov p^{(n)} = (p_w^{(n)} :
   w \in V_n), n  \geq 1$,  such that $p_w^{(n)} =  \mu([e(v_0,w)])$
    where  $e(v_0,w)$ is a finite path
 from $E(v_0,w)$. It is clear that, for $w \in V_n$,
\begin{equation}\label{eq path extension}
[e(v_0, w)] = \bigcup_{e(w, v), v \in V_{n+1}} [e(v_0, w), e(w, v)]
\end{equation}
so that $[e(v_0, v)] \subset [e(v_0, w)]$.

In other words, we see that the entries of  vectors $\ov p^{(n)}$ can
 be  found by the formula
$$
p_w^{(n)} = \frac{\mu(X_w^{(n)})}{h_w^{(n)}}
$$
where
\begin{equation}\label{eq tower X}
X_w^{(n)} = \bigcup_{e \in E(v_0,w)} [e], \ \ w \in V_n.
\end{equation}
This means that $X_w^{(n)}$ is a tower in the Kakutani-Rokhlin
  partition   that   corresponds to the vertex $w$
   \cite{HermanPutnamSkau1992}.  The
  measure of this tower is
 \begin{equation}\label{eq tower measure}
  \mu(X_w^{(n)}) = h_w^{(n)}p_w^{(n)} =: q_w^{(n)}.
 \end{equation}
 Then relation (\ref{eq path extension}) implies that
\begin{equation}\label{formula for measures}
\tl F_n^T \ov p^{(n+1)} = \ov p^{(n)}, \ n\geq 1,
\end{equation}
where  $\tl F_n^T$ denotes the transpose of the matrix $\tl F_n$.

Because $\mu(X_B) =1$, we see that,  for any $n >1$,
$$
\sum_{w \in V_n} h_w^{(n)}p_w^{(n)} = \sum_{w \in V_n} q_w^{(n)} = 1.
$$

The following lemma is formulated in terms of the stochastic incidence matrices $F_n$
 (see~(\ref{stoch_inc_matr})).
\begin{lemma}\label{formula_for_qn}
Let $\mu$ be a probability measure on the path
 space $X_B$ of a Bratteli diagram $B$. Let $(F_n)$ be a sequence
 of corresponding stochastic incidence matrices. Then, for every $n
  \geq 1$, the vector $\ov q^{(n)} = ( q^{(n)}_v : v \in
   V_n)$ (see (\ref{eq tower measure})) is a  probability vector  such that
 \begin{equation}\label{eq stoch matrix and measures}
 F_n^T \ov q^{(n+1)} = \ov q^{(n)}, \ \ \ \ n  \geq 1.
 \end{equation}
\end{lemma}

\begin{proof}
This fact   follows from the definition of the stochastic matrix $F_n$
 and relations (\ref{formula for heights}) - (\ref{formula for
 measures})
\end{proof}

We see that the formula in (\ref{formula for measures}) is a necessary
condition for a sequence of vectors $(\ov p^{(n)})$ to be defined by an
  invariant  probability measure. It turns out that the  converse
   statement is  true, in general. We formulate below Theorem
    \ref{Theorem_measures_general_case}  (proved in
\cite{BezuglyiKwiatkowskiMedynetsSolomyak2010}), where all
 $\mathcal E$-invariant measures are explicitly
described.

Using Lemma~\ref{formula_for_qn},  we
 define a decreasing sequence of convex polytopes $\Delta_m^{(n)}$,
 $n,m \geq 1$, and the limiting convex sets $\Delta_{\infty}^{(n)}$. They
 are used  to describe the set  $\mathcal{M}_1(B)$ of all probability
 $\mathcal{E}$-invariant measures on $B$. Namely, denote
$$
\Delta^{(n)} := \{(z^{(n)}_w)_{w \in V_n}^T : \sum_{w \in V_n}
z^{(n)}_w = 1\ \mbox{ and } \ z^{(n)}_w   \geq 0, \  w \in V_n \}.
$$
The sets $\Delta^{(n)}$ are standard simplices in the space
$\mathbb{R}^{|V_n|}$ with $|V_n|$
 extreme points $\{\ov e^{(n)}(w) : w \in V_n\}$, where
$\ov e^{(n)}(w) = (0, ... , 0, 1, 0,...0)^T$ is the standard basis vector, i.e.
$e^{(n)}_u(w) = 1$ if and only if $u = w$.
Since $F_n$ is a stochastic matrix, we have the obvious property that
 $$
 F_{n}^T(\Delta^{(n+1)}) \subset \Delta^{(n)}, \quad n \in \mathbb{N}.
 $$

Let $\mu$ be  a probability $\mathcal E$-invariant measure $\mu$
on $X_B$ with values $q_w^{(n)}$ on the towers $X_w^{(n)}$.
Then $(q_w^{(n)} : w \in V_n)^T$ lies in  the simplex
 $\Delta^{(n)}$.
Set
\begin{equation}\label{eq def Delta^n_m}
\Delta_m^{(n)} = F_n^T \cdot \ldots \cdot
F_{n+m-1}^T(\Delta^{(n+m)})
\end{equation}
for $m = 1,2,\ldots$
Then we see that
$$
\Delta^{(n)} \supset \Delta^{(n)}_1 \supset \Delta^{(n)}_2 \supset \ldots
$$
Denote
\begin{equation}\label{eq convex set Delta}
\Delta_{\infty}^{(n)} = \bigcap_{m=1}^{\infty}\Delta_m^{(n)}.
\end{equation}
It follows from  (\ref{eq def Delta^n_m}) and  (\ref{eq convex set Delta})
 that
\begin{equation}\label{eq F_n relates Delta}
F_n^T(\Delta_{\infty}^{(n +1)}) = \Delta_{\infty}^{(n)}, \quad n \geq 1.
\end{equation}

The next theorem, that was proved in
\cite{BezuglyiKwiatkowskiMedynetsSolomyak2010}, describes
all $\mathcal E$-invariant probability measures. We formulate here a
slightly stronger version of this statement.

\begin{thm}
\label{BKMS_measures=invlimits}
 \label{Theorem_measures_general_case}  Let $B = (V,E)$ be a
 Bratteli diagram with the sequence of stochastic incidence matrices
 $(F_n)$, and let $\mathcal M_1(B)$ denote the set of
 $\mathcal E$-invariant probability measures on the path space
 $X_B$.

(1) If  $\mu \in \mathcal M_1(B)$, then the probability vector
$$
\ov q^{(n)}= (\mu(X_w^{(n)}))_{w\in V_n}
$$
satisfies the following conditions for $n\geq 1$:

(i) $$\ov q^{(n)} \in  \Delta_{\infty}^{(n)},$$

(ii) $$F^{T}_n\ov q^{(n+1)} = \ov q^{(n)},$$
where $X_w^{(n)}$ is defined in (\ref{eq tower X}).

Conversely, suppose that $\{\ov q^{(n)}\}$ is a sequence of
non-negative probability vectors such that, for every
 $\ov q^{(n)}= (q^{(n)}_w)_{w\in V_n} \in  \Delta^{(n)}$ ($n\geq 1$),
 the condition  (ii) holds. Then the vectors  $\ov q^{(n)}$ belong to
 $\Delta_{\infty}^{(n)}$, $n \in \mathbb N$, and   there  exists a uniquely
  determined $\mathcal E$-invariant  probability measure $\mu$ such that
   $\mu(X_w^{(n)})= q_w^{(n)}$   for $w\in V_n, n \in \mathbb N$.

(2) Let $\Omega$ be the subset of the infinite product
$\prod_{n\geq 1} \Delta^{(n)}_\infty$ consisting of sequences
$(\ov q^{(n)})$ such that $F^{T}_n\ov q^{(n+1)} = \ov q^{(n)}$. Then
 the map
$$
\Phi : \mu \ \mapsto \ (\ov q^{(n)}) : \mathcal M_1(B)  \mapsto
\Omega,
$$
is an affine isomorphism. Moreover, $\Phi(\mu)$ is an extreme point of
$\Omega$ if and only if $\mu$ is ergodic.

(3) Let $B$ be a Bratteli diagram of rank $K$. Then the number of ergodic invariant
 measures on $B$ is
 bounded  above by $K$ and bounded below by the dimension of  the finite-dimensional simplex
 $\Delta^{(1)}_\infty$.
\end{thm}

\begin{remar}
(a) Recall that the set $\Delta_{\infty}^{(n)} = \bigcap_{m=1}^{\infty}
 \Delta_m^{(n)}$
is a convex subset of $\Delta^{(n)}$ for all $n \geq 1$.
It follows from (\ref{eq F_n relates Delta}) that  the following
sequence of maps is defined:
\begin{equation}\label{2.3a}
\Delta^{(1)}_{\infty} \stackrel{F_1^T}{\longleftarrow}
\Delta^{(2)}_{\infty} \stackrel{F^T_2}{\longleftarrow}
\Delta^{(3)}_{\infty} \stackrel{F^T_3}{\longleftarrow} \ldots
\end{equation}

By Theorem \ref{Theorem_measures_general_case}, the set
$\mathcal{M}_1(B)$ can be identified with the inverse limit of the
sequence $(F_n^T, \Delta^{(n)}_{\infty})$.
There is a one-to-one correspondence between  measures $\mu \in
\mathcal{M}_1(B)$ and  sequences of vectors $\ov q^{(n)} \in
\Delta_{\infty}^{(n)}$ such that $\ov q^{(n)} = F_n^T(\ov q^{(n+1)}),
n=1,2,\ldots$

(b) Sometimes, to shorten our notation, we will call the extreme points of a convex subset of
$\mathbb{R}^n$, $n \in \mathbb{N}$ the {\it vertices}. When it does not cause a confusion, these two notions will be used as synonyms.

(c) Theorem~\ref{BKMS_measures=invlimits} states that  in order to find
all ergodic  invariant measures on a diagram it suffices to know the
number of  extreme points of $\Delta^{(n)}_{\infty}$ for
 every $n$. Thus, in the statements which are  proved in this paper,
 we will fix any natural number  $n$ and investigate the corresponding
  convex set  $\Delta^{(n)}_{\infty}$.

(d) In general, the set $\Delta_\infty^{(n)}$ is a convex subset of the
$(|V_n| - 1)$-dimensional simplex $\Delta^{(n)}$. In some cases,
 which will be considered in Section~\ref{sect 4}, the set
 $\Delta_\infty^{(n)}$ is a finite-dimensional simplex itself.

(e)  In Theorem~\ref{Theorem_measures_general_case} part (2), the set
$\mathcal{M}_1(B)$ can be affinely isomorphic to the set
$\Delta_{\infty}^{(1)}$. For instance, it happens when all stochastic
incidence matrices are square non-singular matrices of the same dimension
$K \times K$ for some $K \in
\mathbb{N}$. This case will be considered in Section~\ref{sect 4}.

(f) The procedure of telescoping defined in Remark~\ref{defin_and_notat}
preserves the set of invariant measures; hence we can apply it when
necessary without loss of generality.
\end{remar}

\begin{example}[Equal row sums (ERS) Bratteli diagrams]
\label{ExampleERS1}We illustrate  Theorem~\ref{Theorem_measures_general_case}
 by a particular class
of Bratteli diagrams that have  the equal row sum (ERS) property.
This  means that the  incidence matrices $(\wt F_n)$ of a Bratteli diagram
$B$ satisfy the condition
  $$
  \sum_{w \in V_n} \tl f_{v,w}^{(n)} = r_n
  $$
  for every $v \in V_{n+1}$. In other words, we have $|r^{-1}(v)| = r_n$ for every $v \in V_{n+1}$. It is known that Bratteli-Vershik
   systems with the $ERS$ property can serve as models for
   Toeplitz subshifts   (see \cite{GjerdeJohansen2000}).
In particular, we have $\wt F_0 = \ov h^{(1)} = (r_0, \ldots, r_0)^T$.
 It follows from (\ref{formula for heights}) that, for ERS Bratteli diagrams,
  $h^{(n)}_w = r_0 \cdots r_{n-1}$ for every $w \in V_n$.
Thus, for every probability $\mathcal{E}$-invariant measure on $B$,
 we obtain
$$
\sum_{w \in V_n} p_w^{(n)} = \frac{1}{r_0 \cdots r_{n-1}}
$$
for $n = 1, 2, \ldots$ The proof of this fact follows easily from
 (\ref{formula for measures}) by induction.  Furthermore, in the case
 of ERS diagrams, we have
$$
f_{vw}^{(n)} = \frac{\tl f_{vw}^{(n)}}{r_n}.
$$
\end{example}

\subsection{Invariant measures in terms of decreasing sequences of polytopes}

In this subsection we  study the decreasing sequence of convex sets which
correspond  to probability invariant measures.
We  focus on  extreme points of these convex sets and on the description of
 these sets  as decreasing sequences of convex polytopes.

Let
$$
G_{(n+m,n)} = F_{n+m} \cdots F_{n}
$$
for $m \geq 0$ and $n \geq 1$.
Denote the elements of $G_{(n+m,n)}$ by $(g_{uw}^{(n+m,n)})$, where
$u \in V_{n+m+1}$ and $w\in V_{n}$; then
$$
g_{uw}^{(n+m,n)} = \sum_{(u_m, \ldots, u_1) \in V_{n+m}\times \ldots
 \times V_{n+1}} f^{(n+m)}_{u,u_m}\cdot f^{(n+m-1)}_{u_m,u_{m-1}}
 \cdots f^{(n)}_{u_1,w}.
$$
The sets $\Delta_m^{(n)}, m \geq 0$, defined in
(\ref{eq def Delta^n_m}),
form a decreasing sequence of convex polytopes in $\Delta^{(n)}$.
Recall that the vectors $\{\ov e^{(n)}(w) : w \in V_n\}$
form the standard basis in $\mathbb{R}^{|V_n|}$.
The extreme points of $\Delta_m^{(n)}$
are some (or all) vectors from the set $\{\ov g^{(n+m,n)}(v) :
v \in V_{n+m+1}\}$, where we denote
$$
\ov g^{(n+m,n)}(v) = (g_{w}^{(n+m,n)}(v))_{w \in V_n} =
G_{(m+n,n)}^T(\ov e^{(n+m+1)}(v)).
$$
Obviously, we have the relation
\begin{equation}\label{2.3}
\ov g^{(n+m,n)}(v) = \sum_{w \in V_n}g_{vw}^{(n+m,n)}\ov e^{(n)}(w).
\end{equation}
Let $\{\ov y^{(n,m)}(v)\}$ be the set of all extreme points of
$\Delta_m^{(n)}$.
Then $\ov y^{(n,m)}(v) = \ov g^{(n+m,n)}(v)$ for $v$ belonging
 to some subset $V_{n+m+1}^{(n)}$ of $V_{n+m+1}$.

We observe the following fact. Every vector $\ov q^{(n)} $ from
 the set $\Delta_{\infty}^{(n)}$ can be  written in the standard basis as
$$
\ov q^{(n)} = \sum_{w\in V_n}q^{(n)}_w \ov e^{(n)}(w).
$$
 It turns out that the numbers $q_v^{(n+m+1)}$, $v \in V_{n+m+1}$,
  are the coefficients in the convex decomposition of $\ov q^{(n)}$
  with respect to vectors $\ov g^{(n+m,n)}(v)$.

\begin{prop}\label{decomposition_q_n_by_q_n+m}
Let $\mu \in \mathcal{M}_1(B)$, and $q^{(n)}_w = \mu(X^{(n)}_w)$
($w \in V_n$) for all $n \in \mathbb{N}$. Then
\begin{equation}\label{2.3ab}
\ov q^{(n)} = \sum_{v \in V_{n+m+1}} q_v^{(n+m+1)} \ov
g^{(n+m,n)}(v).
\end{equation}
In particular,
$$
\ov q^{(1)} = \sum_{v \in V_{m+1}} q_v^{(m+1)} \ov g^{(m+1,1)}(v).
$$
\end{prop}

\begin{proof}
The result follows from the following computation:
\begin{align*}
\ov q^{(n)} &= F_n^{T} \cdots F^T_{n+m} (\ov q^{(n+m+1)})\nonumber\\
 &= F_n^{T} \cdots F^T_{n+m} \left(\sum_{v \in V_{n+m+1}}
 q_v^{(n+m+1)} \ov e^{(n+m+1)}(v)\right)\nonumber\\
 &= \sum_{v \in V_{n+m+1}} q_v^{(n+m+1)} F_n^{T} \cdots F^T_{n+m}
 \ov e^{(n+m+1)}(v)\nonumber\\
 &= \sum_{v \in V_{n+m+1}} q_v^{(n+m+1)} \ov g^{(n+m,n)}(v).
\end{align*}
\end{proof}

For every $n \geq 1$,  define
$$
\Delta^{(n), \varepsilon}_{\infty} := \bigcup_{\ov q \in
 \Delta_{\infty}^{(n)}} B(\ov q, \varepsilon),
$$
where $B(\ov q, \varepsilon)$ is the ball of radius $\varepsilon > 0$
centered at  $\ov q \in \mathbb{R}^{|V_n|}$. Here the metric is
defined by the Eucleadian norm $||\cdot ||$ on $\mathbb R^{|V_n|}$.

The following lemma can be proved straightforward, so we omit the proof.

\begin{lemma}\label{remar1}
Fix any natural numbers  $n$ and $m$.   Let $\Delta_{m}^{(n)}$ be defined
as above.
If $\ov q^{(n,m)} \in \Delta_{m}^{(n)}$ for infinitely many $m$
and $\ov q^{(n,m)} \rightarrow \ov q^{(n)}$ as $m \to \infty$,   then
$\ov q^{(n)} \in
\Delta_{\infty}^{(n)}$. Moreover, for every $\varepsilon > 0$ there exists $m_0 = m_0(n,\varepsilon)$
 such that $\Delta_{m}^{(n)} \subset \Delta^{(n),
 \varepsilon}_{\infty}$ for all $m \geq m_0$.
\end{lemma}

In the next statement we prove that  extreme points of the limiting convex set
$\Delta^{(n)}_{\infty}$ can be obtained as limits of sequences of extreme points
 of convex polytopes.

\begin{lemma}\label{remar3}
Fix $n \in \mathbb{N}$ and $\varepsilon >0$.  Let
$\Delta^{(n)}_{\infty}$,
$\Delta_{n+m+1}^{(n)}$, $V^{(n)}_{n+m+1}$ be defined as above for
any $m \in \mathbb{N}$. Then, for every extreme point $\ov y$ of
$\Delta^{(n)}_{\infty}$ there exists $m_0 = m_0(n,\varepsilon)$ such that,
 for all $m \geq m_0$, one can find an extreme point $\ov y^{(n,m)}(v)$ of
  $\Delta_{n+m+1}^{(n)}$,  $v \in V^{(n)}_{n+m+1}$,
satisfying the property
$$
 ||\ov y - \ov y^{(n,m)}(v)|| < \varepsilon.
 $$
\end{lemma}

\begin{proof}
We first assume that $\mathrm{diam}(\Delta_{\infty}^{(n)}) = 0$. Then
$\mathrm{diam}(\Delta_m^{(n)})\to 0$ as $m \to \infty$.
Hence $\Delta_{\infty}^{(n)}$ is a single point, and Lemma~\ref{remar3}
holds.

Let $\mathrm{diam}(\Delta_{\infty}^{(n)}) > 0$. Then  we can find
$m_0 \geq 1$ and $1 \leq n_0 \leq |V_n| - 1$ such that $\mathrm{dim}
 (\Delta_m^{(n)}) = n_0$ for all $m \geq m_0$. Without loss of generality
 we can assume that $m_0 = 1$.

Assume that Lemma~\ref{remar3} is not true, that is the converse
 statement holds. Then, by Lemma
\ref{remar1},  we can find $\varepsilon_0 > 0$ and an infinite
 subset $M \subset \{1,2,\ldots\}$ such that $B(\ov y, \varepsilon_0)$
 contains no extreme points of  $\Delta_{m}^{(n)}$ for $m \in M$. By
 Caratheodory's Theorem on convex hulls, there is an $n_0$-dimensional
 simplex with extreme points $\{\ov y^{(n,m)}(v), v \in V'_{n+m+1} \subset
 V^{(n)}_{n+m+1}\}$ such that $|V'_{n+m+1}| = n_0 + 1$ and
\begin{equation}\label{2.8a}
\ov y = \sum_{v \in V'_{n+m+1}} a_v \ov y^{(n,m)}(v),\ \ m \in M,
\end{equation}
where each $a_v \geq 0$ and $\sum_{v \in V'_{n+m+1}} a_v = 1$. Since
 $B(\ov y, \varepsilon_0)$ contains no vectors $\ov y^{(n,m)}(v), v \in
 V'_{n+m+1}$, there is $\varepsilon'_0 > 0$ such that $0 \leq a_v \leq 1
 - \varepsilon'_0$ for all $v \in V'_{n+m+1}$. As far as $\sum_{v \in
 V'_{n+m+1}} a_v = 1$, we can find two different vertices $v, v'
 \in V'_{n+m+1}$ (both $v$ and $v'$ depend on $m$) such that
\begin{equation*}\label{2.8b}
a_v, a_{v'} \in [\varepsilon''_0, \ 1 - \varepsilon'_0],
\end{equation*}
where $\varepsilon''_0 = \dfrac{\varepsilon'_0}{|V_n|}$.

Relation ~(\ref{2.8a}) can be rewritten in the form
\begin{multline}\label{2.8c}
\ov y = (a_v + a_{v'})\left(\frac{a_v}{a_v + a_{v'}} \ov y^{(n,m)}(v) +
 \frac{a_{v'}}{a_v + a_{v'}} \ov y^{(n,m)}(v')\right)  \\
 +(1 - a_v - a_{v'}) \sum_{u \in  V'_{n+m+1} \setminus \{v,v'\}}
  \frac{a_u}{1 - a_v - a_{v'}} \ov y^{(n,m)}(u).
\end{multline}

First assume that $(a_v + a_{v'}) \rightarrow 1$ as $m \rightarrow
\infty$. We can find an infinite subset $M' \subset M$ such that when $m
\in M'$ and $m \rightarrow \infty$ we have
\begin{equation}\label{2.8c1}
\frac{a_v}{a_v + a_{v'}} \rightarrow \lambda, \qquad  \frac{a_{v'}}{a_v +
 a_{v'}} \rightarrow 1 - \lambda,
\end{equation}
where $\frac{1}{2}\varepsilon''_0 \leq \lambda \leq 1 - \frac{1}{2}
\varepsilon''_0$, and
$$
\ov y^{(n,m)}(v) \rightarrow \ov z(v) \in \Delta_{\infty}^{(n)}, \qquad
\ov y^{(n,m)}(v') \rightarrow \ov z(v') \in \Delta_{\infty}^{(n)}.
$$
It follows from~(\ref{2.8c}) that
\begin{equation}\label{2.8d}
\ov y = \lambda \ov z(v) + (1-\lambda)\ov z(v').
\end{equation}
Clearly, $\ov z(v) \neq \ov z(v')$ because otherwise we would get that
$B(\ov y, \varepsilon_0)$ contains an extreme vector $\ov y^{(n,m)}(v)$ for some
$m \in M'$. It follows from~(\ref{2.8d}) that $\ov y$ is not an extreme vector of
$\Delta_{\infty}^{(n)}$ and we get a contradiction.

If $a_v + a_{v'} \nrightarrow 1$ as $m\to \infty$, then we can find again
 an infinite subset $M'' \subset M$ such that conditions~(\ref{2.8c1}) are
 satisfied. Moreover,  for some $\tau\in (0, 1)$, we obtain
$$
a_v + a_{v'} \rightarrow \tau,
$$
and
$$
\sum_{u \in  V'_{n+m+1} \setminus \{v,v'\}} \frac{a_u}{1 - a_v - a_{v'}}
\ov y^{(n,m)}(u) \rightarrow \ov z' \in \Delta^{(n)}_{\infty}
$$
as $m \in M''$ and $m \rightarrow \infty$. Then~(\ref{2.8c}) implies that
$$
\ov y = \tau(\lambda \ov z(v) + (1-\lambda)\ov z(v')) + (1 - \tau)\ov z'.
$$
Therefore, we showed that $\ov y$ is not an extreme vector
of $\Delta_{\infty}^{(n)}$. This contradiction proves the lemma.
\end{proof}

\section{A criterion of unique ergodicity of Bratteli diagrams} \label{Section 3}

In this section, we deal with arbitrary Bratteli diagram $B$ and prove  a
 criterion of the unique ergodicity of $B$, i.e., we discuss the case when
 the space $\mathcal M_1(B)$ is a singleton.

\begin{thm}\label{uniq_erg} 
A Bratteli diagram $B = (V,E)$ is uniquely ergodic if and only if there exists a telescoping $B'$ of $B$ such that
\begin{equation}\label{2.3b}
\lim_{n\to \infty}\ \max_{v, v' \in V_{n+1}} \left(\sum_{w \in V_n}
\left|f_{vw}^{(n)} - f_{v'w}^{(n)}\right| \right) = 0 
\end{equation}
Here $f_{vw}^{(n)}$ are entries of
 the stochastic matrix $F_n$ defined by the diagram $B'$.

\end{thm}

\begin{proof}
We observe that $B$ is uniquely ergodic if and only if the simplex
$\Delta_{\infty}^{(n)}$ is a singleton for all $n = 1,2,\ldots$

To prove the ``if'' part, it suffices to show that the diameter
$\mathsf{diam}(\Delta_m^{(n)}) \to 0$ 
 as $m \to \infty$.  The polytope $\Delta_m^{(n)}$ is the convex
hull of the vectors $\{\ov g^{(n+m,n)}(v)\}_{v \in V_{n+m+1}}$.
According to~(\ref{2.3}), we have
$$
d^*(\ov g^{(n+m,n)}(v), \ov g^{(n+m,n)}(v')) = \sum_{w \in V_n}
\left|g_{vw}^{(n+m,n)} - g_{v'w}^{(n+m,n)}\right|.
$$
Then we obtain
\begin{eqnarray*}
\sum_{w \in V_n}\left|g_{vw}^{(n+m,n)} - g_{v'w}^{(n+m,n)}\right| &=&
\sum_{w \in V_n}\left|\sum_{u \in V_{n+m}} (f_{vu}^{(n+m)} - f_{v'u}^{(n
+m)})g_{uw}^{(n+m-1,n)}\right|\\
&\leq&  \sum_{u \in V_{n+m}} \left|f_{vu}^{(n+m)} - f_{v'u}^{(n+m)}\right|
\sum_{w \in V_n} g_{uw}^{(n+m-1,n)}\\
&=& d^*(\ov f^{(n+m)}(v),\ov f^{(n+m)}(v')).
\end{eqnarray*}

The last equality is due to the fact that $\sum_{w \in V_n}
g_{uv}^{(n+m-1,n)} = 1$.
Thus, we have
$$
d^*(\ov g^{(n+m,n)}(v), \ov g^{(n+m,n)}(v')) \leq d^*(\ov f^{(n+m)}(v),
\ov f^{(n+m)}(v')) \rightarrow 0
$$
as $m \rightarrow \infty$. Hence $\mathsf{diam}(\Delta_m^{(n)})
\rightarrow 0$ which implies that $\Delta^{(n)}_{\infty}$ is a singleton.

Next,  we prove the ``only if'' part. Since $\Delta^{(n)}_{\infty}$ is a single
point, we have
 $\mathsf{diam}(\Delta_m^{(n)}) \rightarrow 0$ as $m \rightarrow \infty$
 for each $n = 1,2,\ldots$  Then
$$
\max_{v,v' \in V_{n+m+1}} d^{*}(\ov g^{(n+m,n)}(v), \ov g^{(n+m,n)}(v')) \leq \mathsf{diam}
(\Delta_m^{(n)}) \rightarrow 0 \mbox{ as } m \rightarrow \infty.
$$
Let $(\varepsilon_n)$ be a sequence converging to $0$. For every $n$,
we can take sufficiently large $m = m(n)$ such that
\begin{equation*}
\max_{v,v' \in V_{n+m+1}} \sum_{w \in V_n}\left|g_{vw}^{(n+m,n)} -
g_{v'w}^{(n+m,n)}\right| \leq \varepsilon_{n}.
\end{equation*}
Hence, we can find sequences $\{n_i\}_{i \geq 1}$,
$\{m_i\}_{i \geq 1}$ of positive integers such that
\begin{equation}\label{2.5}
\max_{v, v' \in V_{n_i+1}} \sum_{w \in V_{n_i}}\left|g_{vw}^{({n_i}+m_i,
{n_i})} - g_{v'w}^{({n_i}+m_i,{n_i})}\right| \leq \varepsilon_{n_i},
\end{equation}
and such that
$$
1 \leq n_1 < n_1 + m_1 = n_2 < n_2 + m_2 = n_3 < \ldots
$$
Telescoping the Bratteli diagram $B = (V,E)$ with respect to the levels $n_1
< n_2 < n_3 < \ldots$,
 we obtain a new Bratteli diagram $B' = (V',E')$ for which stochastic incidence matrices
 are  $G^{(n_i + m_i, n_i)}$, $i = 1,2,\ldots$ Then (\ref{2.5}) implies
 (\ref{2.3b}), and the proof is complete.
\end{proof}

We remark that the operation of telescoping and using the stochastic 
incidence matrix instead of the usual integer-valued incidence matrix are
 crucial steps for Theorem  
\ref{uniq_erg}. In the following example we show that without telescoping
the statement of the theorem is not true.

\begin{example} \label{ex 3.2} In this example we illustrate the criterion 
given in
 Theorem \ref{uniq_erg}. Let $B_1$ be a Bratteli diagram with incidence
  matrices
$$
\tl F_n^{(1)} =
\begin{pmatrix}
n & 1\\
1 & n
\end{pmatrix}, \qquad n \in \mathbb N,
$$
and let $B_2$ be a Bratteli diagram with incidence matrices
$$
\tl F_n^{(2)} =
\begin{pmatrix}
n^2 & 1\\
1 & n^2
\end{pmatrix}, \qquad n \in \mathbb N.
$$
It is known that the diagram $B_1$ is uniquely ergodic, and $B_2$ has
exactly two finite invariant  ergodic measures, see details in
 \cite{BezuglyiKwiatkowskiMedynetsSolomyak2013},
 \cite[Example 3.6]{AdamskaBezuglyiKarpelKwiatkowski2016},  and
 \cite{FerencziFisherTalet2009}.

We show how Theorem \ref{uniq_erg} works in this case and
emphasize  the importance of telescoping in its proof.

First notice that the both diagrams  have the ERS property, hence the
 corresponding stochastic matrices are easy to compute
 (see Example~\ref{ExampleERS1}):
$$
F^{(1)}_n =
\begin{pmatrix}
1 - \dfrac{1}{n+1} & \dfrac{1}{n+1}\\
\\
\dfrac{1}{n+1} & 1 - \dfrac{1}{n+1}
\end{pmatrix}
$$
and
$$
F^{(2)}_n =
\begin{pmatrix}
1 - \dfrac{1}{n^2+1} & \dfrac{1}{n^2+1}\\
\\
\dfrac{1}{n^2+1} & 1 - \dfrac{1}{n^2+1}
\end{pmatrix}.
$$
Obviously, without telescoping, for the both diagrams $B_1$ and $B_2$
the limit in~(\ref{2.3b})
 equals $2$. However, the telescoping procedure reveals the crucial
  difference between the diagrams $B_1$ and $B_2$.

Suppose we have an ERS diagram with $2 \times 2$ stochastic incidence
 matrices
$$
F_n =
\begin{pmatrix}
a_n & b_n\\
b_n & a_n
\end{pmatrix}.
$$
As before, let $G_{(n,n+m)} = (g_{vw}^{(n+m,n)})$ be the corresponding
 product matrix.
It can be easily proved  by induction that,  for arbitrary
$n, m  \in \mathbb{N}$, the following formula holds:
$$
S(n,m) = \sum_{w \in V_n} \left|g_{vw}^{(n+m,n)} - g_{v'w}^{(n+m,n)}\right| =
 2\prod_{i = 0}^{m}\left|(a_{n+i} - b_{n+i})\right|.
$$

In the case of the diagram $B_1$, we obtain
$$
S(n,m) = 2\prod_{i = 0}^{m}\left(1 - \frac{2}{n+i+1}\right), \quad
n, m \in \mathbb N.
$$
Since the series $\sum_{n=1}^{\infty} n^{-1}$ diverges, we see
that
$$
S(n,m) \to 0, \qquad m \to \infty.
$$
 Choose a decreasing sequence $(\varepsilon_k)$ such that
 $\varepsilon_k \to 0$ as $k \to \infty$. For $n = n_1$ and
 $\varepsilon_1$, find $m_1$ such that $S(n_1, m_1) < \varepsilon_1$.
 Set $n_2 = n_1 + m_1$. For  $\varepsilon_2$, find $m_2$ such that
 $S(n_2, m_2) < \varepsilon_2$. Set $n_3 = n_2 + m_2$. Continuing
 in the same manner, we construct a sequence $(n_k)$ such
 that $S(n_k, n_{k+1} - n_k) < \varepsilon_k$. Telescope the diagram
 with respect to the levels $(n_k)$. By Theorem
 \ref{uniq_erg}, we conclude that the diagram $B_1$ is uniquely ergodic.

In the case of diagram $B_2$,  the convergence
of the series $\sum_{n=1}^{\infty} n^{-2}$, does not allow us
to conclude that $S(n,m) \to 0$ as $m \to \infty$. Hence, the criterion
is not applicable for the diagram $B_2$.

\end{example}

The following example shows that using stochastic incidence matrices is an
 important part of Theorem~\ref{uniq_erg}. If we use usual 
integer-valued incidence matrices instead of stochastic ones, then Theorem 
\ref{uniq_erg} will be not true.

\begin{example}\label{non-simple-stat-ex}
Let $B$ be the stationary Bratteli diagram defined by the incidence
  matrices
$$
\tl F_n = \tl F =
\begin{pmatrix}
3 & 0\\
1 & 2
\end{pmatrix}
$$
for every $n \in \mathbb N$. It is well known that $B$ has a unique finite
 ergodic  measure supported by the 3-odometer (see e.g.
 \cite{BezuglyiKwiatkowskiMedynetsSolomyak2010}). We show that $B$ 
 satisfies the condition of unique ergodicity formulated in 
 Theorem \ref{uniq_erg}. It is easy to check that the $n$-th power of 
 $\tl F$ is
$$
\tl F^{n} =  
\begin{pmatrix}
3^n & 0\\
3^n - 2^n & 2^n
\end{pmatrix}.
$$
Hence the entries of the matrix $\tl F$ do not satisfy (\ref{2.3b}) even after 
taking products of these matrices (which corresponds to telescoping of $B$).
 Notice that $B$ has the 
ERS property. For any $n \in \mathbb{N}$ and a vertex $w \in V_n$, we 
have $h_w^{(n)} = 3^n$. Therefore, the corresponding stochastic incidence
 matrix and its $n$-th power are
$$
F =
\begin{pmatrix}
1 & 0\\
\frac{1}{3} & \frac{2}{3}
\end{pmatrix}
$$
and
$$
F^{n} =  
\begin{pmatrix}
1 & 0\\
1 - \dfrac{2^n}{3^n} & \dfrac{2^n}{3^n}
\end{pmatrix}.
$$
Hence, we see that $B$ satisfies (\ref{2.3b}).

\end{example}

\section{Ergodic invariant measures for finite rank Bratteli diagrams }
\label{sect 4}

In this section, we  study ergodic invariant measures on finite rank
 Bratteli diagrams. Let $B=(V,E)$ be a Bratteli diagram of rank $k\geq2$,
  i.e. after telescoping one can assume that $|V_n| = k$ for $n = 1,2,\ldots$
Then, for any $n$, the incidence matrices $\tl F_n =
(\tl f_{v,w}^{(n)})_{v \in V_{n+1},
 w \in V_n}$ and corresponding stochastic matrices $F_n$ are $k \times k$
  matrices. In the case of a Bratteli diagram of finite rank,
the representation of the set $\mathcal{M}_1(B)$ as the inverse limit of
$(F_n^T, \Delta_{\infty}^{(n)})$  becomes more transparent.

Throughout this section, we use the following assumption.

\medskip
\noindent
\textbf{Assumption}: \textit{We  assume that}:

(i) $B = (V,E)$ \textit{is a Bratteli diagram of finite rank} $k$;

(ii) \textit{all incidence matrices $\tl F_n = (\tl f^{(n)}_{vw})$ of
$B$ are  nonsingular,}
  $$
  \det \tl F_n \neq 0, \ \ \ n\in \mathbb N.
  $$
\medskip

It follows immediately that  the stochastic matrices $F_n =
(f_{vw}^{(n)})_{v\in V_{n+1}, w \in V_n}$ are also nonsingular.

\subsection{Invariant measures in terms of decreasing sequences of
 simplices}\label{subsec4.1}

Under the made assumption, the number of probability ergodic
invariant measures on a Bratteli diagram equals the number of extreme points of a
simplex $\Delta_{\infty}^{(1)}$, which, in turn, is the limit of a decreasing
sequence of simplices $\{\Delta_{m}^{(1)}\}_{m=1}^{\infty}$. Indeed,
since the square matrices $F_n^T$  are  invertible, one can
restore the whole sequence of vectors
  $(\ov q^{(n)})_{n=1}^{\infty}$ from the vector $\ov q^{(1)}$,
see  Theorem \ref{Theorem_measures_general_case}.
Below we explicitly define these simplices, using the incidence matrices
 of the Bratteli
diagram, and  study the way these simplices are embedded one into
another. This study will allow us to derive the properties of the incidence
matrices in the case when the number of probability ergodic invariant
 measures is known.

\begin{remark}
(1) Because of the  assumption formulated above, it is easily seen that, for
any $m, n \in \mathbb{N}$, the sets $\Delta_m^{(n)}$ are simplices with
 exactly $k$ extreme points.
 Then the simplices $\Delta^{(n)} (=\Delta^{(1)})$  do not depend
 on the level $V_n$. The index $n$ points out only that $\Delta^{(n)}$ has
  the  basis enumerated by vertices of the level $V_n$.  

(2) Since all incidence matrices are non-singular, the intersection
$\Delta^{(1)}_{\infty} =
 \bigcap_{m=1}^{\infty} \Delta_m^{(1)}$ is a simplex which we can
  identify with the set $\mathcal{M}_1(B)$ of all probability
  $\mathcal{E}$-invariant Borel measures on $X_B$. The simplex
  $\Delta^{(1)}_{\infty}$ has \textit{at most} $k$ extreme points (see the
  proof in \cite{BezuglyiKwiatkowskiMedynetsSolomyak2013}) which is
  based on  \cite{Phelps2001, Pullman1971}, and
   \cite{BezuglyiKwiatkowskiMedynetsSolomyak2010}).
It is not hard to see that the dimension of $\Delta^{(1)}_{\infty}$ can be
 any natural number $ l$, $1 \leq l \leq k$.
In \cite{BezuglyiKwiatkowskiMedynetsSolomyak2013}, the reader can find
various examples of Bratteli  diagrams where the number of ergodic
$\mathcal E$-invariant  measures varies from $1$ (this is the case of
uniquely ergodic   diagrams) to $k$.

 (3) If $\Delta_{\infty}^{(1)}$ is a simplex with $l$ extreme points then there are
 exactly $l$ ergodic $\mathcal E$-invariant Borel probability
 measures on $B$. In the case of finite rank Bratteli diagrams with
  nonsingular incidence matrices, the simplex $\Delta^{(n)}_\infty$
   also has exactly $l$ extreme points for every $n$.

(4) Under the made assumption, we know precisely the number $k$ of extreme points of the prelimit simplices $\Delta_m^{(n)}$, and the number of extreme points of the limiting simplex $\Delta^{(n)}_\infty$ may be less than $k$ only due to the limit procedure. In the case when matrices $\tl F_n$ are singular, the decrease of the number of extreme points may happen already in the prelimit simplices $\Delta_m^{(n)}$ because of the rank of incidence matrices. Moreover, some of the limiting simplices $\Delta^{(n)}_\infty$ may have different number of extreme points depending on $n$. Nevertheless, it is not difficult to see that there always exists $1 \leq l \leq k$ such that the simplices $\Delta^{(n)}_\infty$ have exactly $l$ extreme points for all $n$ large enough. Then if we restrict the linear operators corresponding to $\tl F_n$ to the appropriate $l$-dimensional subspaces, we will obtain non-singular $l \times l$ matrices which satisfy the assumption.

(5)  We recall that $\Delta^{(m)}$ denotes the standard simplex in
$\mathbb R^{|V_m|}$, and the entries of a  vector from this
simplex are enumerated by vertices from $V_m$. In the case
when $|V_m| = k $, these simplices $\Delta^{(m)}$ are essentially
the same.
Since the incidence matrices of
$B$ are nonsingular, we obtain that the vectors
\begin{equation}\label{eq vectors z_m}
\ov y^{(m)}(w)  =   F_1^T  \cdots
 F_{m-1}^T(\ov e^{(m)}(w))
\end{equation}
are extreme in $\Delta^{(1)}_{m}$ and form a basis in the simplex
$\Delta^{(1)}_{m}, m \geq 1$. Hence, there is a
one-to-one correspondence between vertices $w$ of $V_m$ and the
 extreme vectors $\ov y^{(m)}(w)$ from $\Delta^{(1)}_{m}$.
\end{remark}

Lemma \ref{lem (n+1)-th simplex in n-th} shows explicitly how the
 simplex $\Delta^{(1)}_{n+1}$ is embedded into $\Delta^{(1)}_{n}$.

\begin{lemma}\label{lem (n+1)-th simplex in n-th}
Let $\ov y^{(n)}(w) = (y_v^{(n)}(w) : v \in V_1)^T$
 be the extreme vector from the simplex  $\Delta^{(1)}_{n}$ which
  corresponds  to a vertex $w \in V_n$, see (\ref{eq vectors z_m}).
   Then, for any $v\in V_{n+1}$,
\begin{equation}\label{1.2a}
\ov y^{(n+1)}(v) = \sum_{w \in V_n}f_{vw}^{(n)} \ov y^{(n)}(w).
\end{equation}
\end{lemma}

\begin{proof}
We observe that the vector $F_n^T (\ov e^{(n+1)}(v))$ coincides
 with the $v$-th row of $F_n$, and it can be written as a linear
  combination of vectors from the standard basis $\{\ov e^{(n)}(w) :
   w \in V_n\}$. Hence, similarly to
   Proposition~\ref{decomposition_q_n_by_q_n+m}, we have
\begin{align*}
\ov y^{(n+1)}(v) & = F_1^T  \cdots  F_{n-1}^T
F_n^T (\ov e^{(n+1)}(v))\\
& = (F_1^T \cdots  F_{n-1}^T)
\left(\sum_{w \in V_n} f^{(n)}_{vw}\ov e^{(n)}(w)\right)\\
 &= \sum_{w \in V_n} f^{(n)}_{vw} \ov y^{(n)}(w).
\end{align*}
This proves the lemma.
\end{proof}

Theorem \ref{BKMS_measures=invlimits} states that, under the made
 assumption, the set $\mathcal{M}_1(B)$ can be identified with the
 simplex $\Delta_{\infty}^{(1)}$.
In fact, since $\det(F_n) \neq 0$ for $n\geq 1$, we have
$$
\Delta_{\infty}^{(1)} = F_1^T \cdots F_{n-1}^T (\Delta_{\infty}^{(n)}).
$$
Therefore, each $\Delta_{\infty}^{(n)}$ is a simplex with the same
number of extreme points. The inverse
limit $\varprojlim(F_n^T,\Delta_{\infty}^{(n)})$ can be identified with
$\Delta_{\infty}^{(1)}$.

We recall that, in this case,  every vector $\ov q^{(1)}$  from the simplex
$\Delta_{\infty}^{(1)}$ corresponds to an
 $\mathcal E$-invariant measure $\mu$. The measure $\mu$
 is defined by a sequence of vectors $\{\ov q^{(n)}\}$ such that
 $\ov q^{(n)}  = F_n^T \ov q^{(n+1)}$.  The  entries of $\ov q^{(n)}$ in
  the standard basis  $\{\ov e^{(n)}(w)\}$ are the numbers
  $(q_w^{(n)})$   which give
the values of measure $\mu$ on the towers  $\{X^{(n)}_w : w \in V_n\}$.
 On the other hand, since $\ov q^{(1)}$ belongs to every simplex
 $\Delta^{(1)}_n$ (recall that $F_n$ is nonsingular for every $n$),
this vector can be represented in the basis formed by the extreme vectors
  $\{\ov y^{(n)}(w) : w \in V_n\}$.

In the next result,  we establish,  using the same argument as in
Proposition~\ref{decomposition_q_n_by_q_n+m}, the relation between the
vector $\ov q^{(1)}$ and the vectors $\{\ov y^{(n)}(w) : w \in V_n\}$.

\begin{lemma}\label{lem barycentric coordinates}
Let $\mu$, $\ov q^{(n)}$ and $\{\ov y^{(n)}(w) : w \in V_n\}$, $n \in
\mathbb{N}$ be as above. Then
for every $n >1$, the entries  $(q_w^{(n)} : w \in V_n)$ of
$\ov q^{(n)}$ are the barycentric coordinates of the vector
$\ov q^{(1)}$ in $\Delta_n^{(1)}$.
\end{lemma}

\begin{proof}
Indeed, we have the following chain of equalities that proves the lemma.
\begin{align}\label{1.3a}
\ov q^{(1)} &=(F_1^T \cdots F_{n-1}^T)
(\ov q^{(n)})\nonumber \\
 &= (F_1^T \cdots F_{n-1}^T)
\left(\sum_{w \in V_n}q_w^{(n)}\ov e^{(n)}(w)\right)\nonumber\\
 &= \sum_{w \in V_n}q_w^{(n)} (F_1^T \cdots
 F_{n-1}^T)(\ov e^{(n)}(w))\\
 \nonumber
 &= \sum_{w \in V_n}q_w^{(n)} \ov y^{(n)}(w).
\end{align}
\end{proof}

Observe that relation (\ref{1.3a}) uniquely defines
 the numbers $q^{(n)}_w$ by the vector $q^{(1)} \in
 \Delta_{\infty}^{(1)}$.

{\subsection{Subdiagrams of Bratteli diagrams and ergodic invariant
 measures}
In this subsection, we study  connections between vertex subdiagrams
of a Bratteli diagram $B$ and  ergodic invariant measures on $B$.

By a \textit{Bratteli subdiagram,} we mean a Bratteli diagram $B'$ that can
be obtained from
$B = (V, E)$ by removing some vertices and edges from each non-zero level of $B$.
 Then  $X_{B'} \subset X_B$ (for more details see e.g.
  \cite{AdamskaBezuglyiKarpelKwiatkowski2016}). In this paper, we will
  consider only the case of \emph{vertex subdiagrams}. To define such a
  subdiagram,  we start with a sequence $\overline W = \{W_n\}_{n>0}$
 of proper subsets $W_n$ of $V_n$ for all $n > 0$. The
 $n$-th level of the vertex subdiagram $B' =  (\ol W, \ol E)$ is formed  by
the vertices from $W_n$. To define the edge set  $\ov E =\{\ov E_n\}$
we take the restriction of $E_n$ on the set of edges connecting vertices
 from  $W_{n}$ and   $W_{n+1}$, $n \in \mathbb N$. This means that
 $e \in \ov E_n$ if the source and range of $e$ are
 in $W_{n}$ and $W_{n+1}$, respectively. Thus, the incidence  matrix
 $F'_n$ of $B'$ has the size $|W_{n+1}| \times |W_n|$, and  it can be seen
 as a submatrix of  $F_n$  corresponding to the
 vertices from $W_{n}$ and $W_{n+1}$. We say, in this case, that $\ol W
  = (W_n)$   is the \textit{support} of $B'$.

%
%

The following theorem will play an important role in our study of ergodic
measures and their supports. For a finite rank Bratteli diagram $B$, it
 describes how extreme points of $\Delta_{\infty}^{(1)}$ determine 
 subdiagrams of $B$.

\begin{thm}\label{prop subdiagrams}
Let $B$ be a Bratteli diagram of rank $k$, and let $B$ have $l$ probability
ergodic invariant measures, $1 \leq l \leq k$. Let $\{\ov y_1, ... , \ov y_l\}$ denote the extreme vectors in $\Delta_{\infty}^{(1)}$.
Then, after telescoping and
renumbering  vertices, there exist exactly $l$ disjoint subdiagrams $B_i$
(they share no vertices other than the root) with the corresponding sets of vertices $\{V_{n,i}\}_{n = 0}^{\infty}$
such that

(a) for every $i = 1,...,l$ and any $n,m >0$, $|V_{n,i}| = |V_{m,i}| > 0$,
while the set $V_{n,0} = V_n \setminus \bigsqcup\limits_{i = 1}^l V_{n,i}$ may be, in particular, empty;

(b) for any $i = 1, \ldots, l$ and any choice of $v_n \in V_{n,i}$, the extreme vectors $\ov y^{(n)}
(v_n) \in \Delta_n^{(1)}$ converge to the extreme vector $\ov y_i \in
\Delta_{\infty}^{(1)}$.

In general, the diagram $B$ can have up to $k - l$ disjoint subdiagrams
$B'_j$ with vertices $\{V'_{n,j}\}_{n = 0}^{\infty}$ such that they are
 also disjoint with subdiagrams $B_i$ and for any $w_n \in V'_{n,j}$, the
  extreme vectors $\ov y^{(n)}(w_n) \in \Delta_n^{(1)}$ converge to a
  non-extreme vector $\ov z \in \Delta_{\infty}^{(1)}$.
\end{thm}

\begin{proof} In the proof, we describe the process of forming the 
vertex sets $\{V_{n,i}\}_{n = 1}^{\infty}$  and the 
corresponding subdiagrams $B_i$ for $i = 0, \ldots, l$. We observe that the
 way of assigning a vertex to a subdiagram is not necessary unique and
 depends on the telescoping of the diagram. 
 
 For every $i = 1,\ldots,l$, we
  can take, by Lemma~\ref{remar3},  a sequence of vectors 
  $\ov y^{(n)}(w_i) \in \Delta_n^{(1)}$ (and
 corresponding vertices $w_i = w_i(n) \in V_n$)
such that $||\ov y^{(n)}(w_i) - \ov y_i|| \rightarrow 0$ as $n \rightarrow
\infty$. If $n$ is sufficiently large, then the vectors 
$\{\ov y^{(n)}(w_i)\}_{i = 1}^{l}$ (and hence the vertices $\{w_i(n)\}_{i = 
1}^{l}$) are disjoint. Therefore, we can choose the beginnings of the 
sequences $\{w_i(n)\}_{n = 1}^{\infty}$ in such a way that all vertices 
$\{w_i(n)\}_{i = 1}^{l}$ are disjoint for all $n$ and the limits of the 
sequences $\{\ov y^{(n)}(w_i)\}_{n = 1}^{\infty}$ stay unchanged. We set 
$w_i(n) \in V_{n,i}$. Thus, at this stage of construction, each set $V_{n,i}$, 
$i = 1, \ldots, l$ consists of a single point.

Now for every $n$, pick up any vertex $u = u(n)$ in $V_n \setminus 
\{w_i(n)\}_{i=1}^l$.
Since the decreasing sequence of simplices $\Delta_{n}^{(1)}$ lies in the
compact set $\Delta^{(1)}$, the sequence $\{\ov y^{(n)}(u)\}_{n = 1}
^{\infty}$ has a convergent subsequence $\{\ov y^{(n_m)}(u)\}_{m = 1}
^{\infty}$. We recall that simplices $\Delta_n^{(1)}$ are decreasing,
so that we may work only with the simplices
$\{\Delta_{n_m}^{(1)}\}_{m = 1}^\infty$ and the limiting simplex
$\Delta_\infty^{(1)}$. For the corresponding Bratteli diagram, this means
that we apply telescoping with respect to the levels $\{n_m\}_{m=1}^\infty$. 
We note that the operation of telescoping does not change the limits of the
 converging sequences $\ov y^{(n)}(w_i(n))$, since the telescoping 
 corresponds to the picking a subsequence $\ov y^{(n_m)}(w_i(n_m))$.
Hence, after telescoping, we have the sequence of vertices $u(n) \in V_n
 \setminus \{w_i(n)\}_{i=1}^l$ such that there exists 
$$
\lim_{n \rightarrow \infty} \ov y^{(n)}(u(n)) = \ov z,
$$
where $\ov z \in \Delta_{\infty}^{(1)}$ (by Lemma~\ref{remar1}). If $\ov z 
= \ov y_i$ for some $i = 1, \ldots, l$, then we set $u(n) \in V_{n,i}$. 
Otherwise, we set $u(n) \in V_{n,0}$.
 Now, for every $n$, pick a vertex $u' \in V_n \setminus (\{w_i\}_{i=1}^l
  \cup \{u\})$ and apply the same procedure again. Since $B$ is of 
 finite rank $k$,
  we will use at most $k - l$ telescopings to finally form the sets $\{V_{n,i}\}_{i = 0}^{l}$. 
In general, the set $V_{n,0}$ may be
 empty. 
If $V_{n,0}$ is not empty, we may also divide it into disjoint subsets 
$\{V'_{n,j}\}$ depending on the limits of 
$\{\ov y^{(n)}(u(n))\}_{n=1}^{\infty}$, though we will not use such a 
partition later. The theorem is proved.
 \end{proof}

Renumber the vertices
   of the diagram such that in every $V_n$ first come the vertices of
$V_{n,1}$, then the vertices of $V_{n,2}, \ldots, V_{n,l}$ and, at last, the
vertices of $V_{n,0}$. Then $F_n$ have blocks $V_{n+1,j} \times
 V_{n,j}$ on the diagonal for $j = 0, 1, \ldots, n$.
Thus, after telescoping and renumbering  vertices, the stochastic incidence
matrices of the Bratteli diagram $B$ from Theorem \ref{prop subdiagrams} will look as follows
\begin{equation}\label{IncidenceMatrixForm}
F_n =\left(
  \begin{array}{ccccccc}
    F^{(n)}_{B_1} & * & \cdots & * & * \\
    * & F^{(n)}_{B_2} & \cdots & * & * \\
    \vdots & \vdots & \ddots & \vdots & \vdots\\
    * & * & \cdots & F^{(n)}_{B_l} & *\\
    * & * & \cdots & * & F^{(n)}_{B_0} \\
  \end{array}
\right),
\end{equation}
where $F^{(n)}_{B_i}$ are submatrices of the matrix $F_n$ of the size
$|V_{n+1,i}| \times |V_{n,i}|$ for $i = 0,\ldots,l$ and symbol $*$
represents  other elements of the matrix.

The following theorem describes the properties of the submatrices
 $F^{(n)}_{B_i}$ for $i = 1,\ldots,l$.
We use here notation of Theorem \ref{prop subdiagrams}.

\begin{thm}\label{estimate1}
Suppose that the Bratteli diagram $B$ is as in Theorem
\ref{prop subdiagrams}. Then, for every $i \in \{1, \ldots, l\}$ and $v \in V_{n+1,i}$, the
 following property holds:
$$
\sum_{w \in V_{n,i}} f^{(n)}_{vw} \rightarrow 1.
$$
\end{thm}

\begin{proof}
We recall what we know about the vectors
$\{\ov y^{(n)}(w)\}_{w \in V_n}$ (extreme points of the
simplex $\Delta^{(1)}_n$):

(a) they form a basis of $\mathbb{R}^k$ for any
 $n \in \mathbb{N}$,

 (b) $\ov y^{(n)}(w) \rightarrow \ov y_i$ as $n \rightarrow \infty$,
 where $w \in V_{n,i}$, $i = 1, \ldots,l$,

  (c)   $\lim_{n   \rightarrow \infty} \ov y^{(n)}(w)$ does not belong to the
  set  $\{\ov y_i\}_{i = 1}^l$ as $w \in V_{n,0}$.

For every $j = 1, ..., l,$ we represent $\ov y_j \in   \Delta_n^{(1)}$ as a
 barycentric combination of $\ov y^{(n)}(w)$:
\begin{equation}\label{eq baryc comb}
\ov y_j = \sum_{w \in V_n}\alpha_w^{(n)} \ov y^{(n)}(w),
\end{equation}
where $\sum_{w \in V_n} \alpha_w^{(n)} = 1$ and $\alpha_w^{(n)} \geq 0$ for all
$w$.

We will show that $\alpha_w^{(n)} \rightarrow 0$ as
$n \rightarrow \infty$ for $w \in V_n \setminus V_{n,j}$.
Indeed, relation (\ref{eq baryc comb}) can be written in the form
$$
\ov y_j - \sum_{w \in V_{n,j}}\alpha_w^{(n)} \ov y^{(n)}(w) = \sum_{u
\notin V_{n,j}}\alpha_u^{(n)} \ov y^{(n)}(u).
$$
Since $\sum_{w \in V_n} \alpha_w^{(n)} = 1$, we get
\begin{equation}\label{vspom}
\sum_{w \in V_{n,j}}\alpha_w^{(n)} (\ov y_j - \ov y^{(n)}(w)) =
\sum_{u \notin V_{n,j}}\alpha_u^{(n)} (\ov y^{(n)}(u) - \ov
y_j).
\end{equation}
As far as $\lim_{n \rightarrow \infty}\ov y^{(n)}(w) = \ov y_j$ for $w \in
V_{n,j}$, the left-hand side of relation (\ref{vspom}) tends to zero as $n\to
\infty$. In the right-hand side of  (\ref{vspom}), we obtain that,
for $u \in V_{n,i}$ and $i = \{1,\ldots l\} \setminus \{j\}$,
 $$
 \ov y^{(n)}(u)-  \ov y_j \rightarrow \ov y_i - \ov y_j, \qquad
 n \rightarrow \infty.
 $$
For $u \in V_{n,0}$, we have $\ov y^{(n)}(u) \rightarrow \ov z(u)$, where
$\ov z(u) \in  \Delta_\infty^{(1)}$, and $\ov z(u)$ does not coincide with
 any of the extreme points $\{\ov y_i\}_{i = 1}^l$. Hence $\ov z(u) =
 \sum_{i = 1}^l
\beta_i \ov y_i$, where all $\beta_i$ are non-negative, $\sum_{i=1}^l
 \beta_i = 1$, and at least two of coefficients from the set
  $\{\beta_i\}_{i = 1}^l$ are strictly positive. Thus,
$$
\ov z(u) - \ov y_j = \sum_{i = 1}^l \beta_i \ov y_i - \sum_{i = 1}^l
\beta_i \ov y_j = \sum_{i \neq j} \beta_i (\ov y_i - \ov y_j) \neq 0.
$$
Therefore the right-hand side of (\ref{vspom}) is a linear combination of
vectors which converge to linearly independent non-zero vectors
$\{\ov y_i - \ov y_j\}_{i \neq j}$. It follows that all the coefficients
$\{\alpha_u^{(n)}\}_{u \notin V_{n,j}}$ converge to zero as $n \to \infty$.
This proves that
$\sum_{w \in  V_{n,j}} \alpha_w^{(n)} \rightarrow 1$ as $n \to \infty$.

Since $\ov y^{(n+1)}(v)$
 approaches arbitrary close to $\ov y_j$ for $v\in
 V_{n+1,j}$ and  sufficiently  large $n$, the same result holds for
 the coefficients of barycentric  combination for
 $\ov y^{(n+1)}(v)\in \Delta_{n+1}^{(1)}$.  By~(\ref{1.2a}), the
   corresponding combination for $\ov
 y^{(n+1)}(v)$ has coefficients $f_{vw}^{(n)}$. Hence, for every
 $v \in V_{n +1,j}$ we have
$$
\sum_{w \in V_{n,j}} f_{vw}^{(n)} \rightarrow 1
$$
and the proof is complete.
\end{proof}

As a consequence of Theorem~\ref{estimate1}, we get the following corollary:

\begin{corol}
In notation of Theorem \ref{prop subdiagrams}, the following property holds for every $i,j \in \{1, \ldots, l\}$:
$$
\frac{1}{|V_{n+1,j}|}\sum_{v \in V_{n+1,j}}\sum_{w \in V_{n,i}}
f^{(n)}_{vw} \rightarrow \delta_{ij},
$$
where $\delta_{ij} = 1$ if $i = j$; $\delta_{ij} = 0$ if $i \neq j$ and
$\dfrac{1}{|V_{n+1,j}|}$ does not depend on $n$ and plays the role of a normalizing coefficient.

\end{corol}

For the matrices $G_{(n+m,n)} = F_{n+m} \cdots F_{n}$, the equality
\begin{equation}\label{meaning_of_product_matrix}
\ov y^{(n+m+1)}(u) = \sum_{w \in V_n} g_{uw}^{(n+m,n)}\ov y^{(n)}(w)
\end{equation}
can be established similarly to~(\ref{1.2a}).
Moreover, by the same argument as in Theorem~\ref{estimate1},
we can prove
 parts (1) and (2) of the following proposition (we leave the details to the
 reader). We note that statement (3) of this proposition  is a straightforward
  corollary of Theorem \ref{prop subdiagrams}.

\begin{corol}\label{estimate2}
In notation of Theorem \ref{prop subdiagrams}, the following
properties hold: for every $\varepsilon > 0$ there exists $N \in \mathbb{N}$
 such that for every $n > N$,  any $j = 1,\ldots,l$, and any $m = 0,1,
 \ldots$

(1)
$$
\min_{u \in V_{n+m+1,j}}\sum_{w \in V_{n,j}} g_{uw}^{(n+m,n)} \geq 1 -
\varepsilon;
$$

(2)
$$
\max_{u \in V_{n+m+1,j}}\sum_{w \in V_n \setminus V_{n,j}} g_{uw}^{(n
+m,n)} \leq \varepsilon,
$$
where $g_{uw}^{(n+m,n)} $ are entries of  $G_{(n+m,n)}$;

(3) there exist $C > 0$ and $N \in \mathbb{N}$ such that for every $n
> N$ and every $j = 1,\ldots,l$ we have
$$
d(\ov y^{(n)}(w), \ov y_j) \geq C, \qquad  w \in V_{n,0},
$$
and
$$
\max_{w\in V_{n,0}}d(\ov y^{(n)}(w), \Delta^{(1)}_\infty) \rightarrow 0
$$
as $n\rightarrow \infty$.
\end{corol}

\subsection{The main theorem I} One of our main results, Theorem
 \ref{main1}, is
proved in this subsection. This result holds for a class of Bratteli diagrams
which  satisfy the following condition.

\begin{defin}\label{def vanishing blocks}
For a Bratteli diagram $B$,  we say that a sequence of proper subsets
$U_n \subset V_n$ defines \textit{blocks of vanishing weights
(or vanishing blocks)} in the stochastic incidence matrices $F_n$ if
$$
 \sum_{w \in U^c_n, v \in U_{n+1}} f_{vw}^{(n)} \rightarrow 0, \qquad
 n \rightarrow \infty
 $$
 where $U_n^c = V_n \setminus U_n$.

If additionally, for every sequence of vanishing blocks $(U_n)$, there exists
a constant  $0 < C_1 < 1$ such that, for sufficiently large $n$,
\begin{equation}\label{eq vanish blocks}
\min_{v\in U_{n+1}} \sum_{u \in U_{n}^c} f_{vu}^{(n)}
\geq C_1 \max_{v\in U_{n+1}} \sum_{u \in U_{n}^c}  f_{vu}^{(n)},
\end{equation}
then we say that the stochastic incidence matrices $F_n$ of $B$ have
 \textit{regularly vanishing blocks}.
\end{defin}

Throughout the paper, we develop the approach used in \cite{BezuglyiKwiatkowskiMedynetsSolomyak2010, BezuglyiKwiatkowskiMedynetsSolomyak2013, AdamskaBezuglyiKarpelKwiatkowski2016} to study an ergodic invariant measure on a Bratteli diagram as the extension of a finite ergodic invariant measure defined on a uniquely ergodic subdiagram $B'$. The key property of this method uses the following fact: if the number of finite paths of length $n$ which lie inside the subdiagram $B'$ grows when $n$ tends to infinity much faster then the number of finite paths that end in the vertices of the same subdiagram but do not lie completely in the subdiagram, then the measure extension is finite, see e.g. Proposition~\ref{1.6b}. 
Below we assume that stochastic incidence matrices of a Bratteli diagram have the property of {\it regularly vanishing blocks} and apply it to the case when $U_n = V_{n,i}$ for some $i = 1, \ldots, l$. 
Then the blocks of the matrices corresponding to the edges that connect vertices from outside of the supporting subdiagram $B_i$ to the vertices of $B_i$ are the blocks of vanishing weights. 
The property of regularly vanishing blocks means that for all vertices of the subdiagram, 
the contribution of the corresponding rows of the stochastic incidence matrix is in some sense similar. For instance, the diagrams from Examples~\ref{ex 3.2}, \ref{non-simple-stat-ex} trivially satisfy this condition.

For the next theorem,  we will use  the following notation. Set
$$
\ov a_j^{(n)} = \frac{1}{|V_{n,j}|} \sum_{w \in V_{n,j}} \ov y^{(n)}(w),
\quad j = 0, 1, \ldots, l,
$$
where the subsets $V_{n,j}$ are defined as in Theorem
\ref{prop subdiagrams}.
Then $\ov a_j^{(n)} \in \Delta_{n,j}^{(1)} :=
\textrm{Conv}\{\ov y^{(n)}(w),
w \in V_{n,j}\}$, the convex hull of the set $\{\ov y^{(n)}(w),
w \in V_{n,j}\}$.   The sets $\Delta_{n,j}^{(1)}$ are subsimplices of
$\Delta_n^{(1)}$, $j = 0, 1, \ldots, l$.
We observe that, by Lemma~\ref{remar1},
 $$
 \max_{\ov a \in \Delta_{n,0}^{(1)}}
\textrm{dist}(\ov a, \Delta_{\infty}^{(1)}) \rightarrow 0
$$
as $n \rightarrow \infty$, where $\Delta_{\infty}^{(1)} =
\bigcap_{n=1}^{\infty}\Delta_n^{(1)}$.
\medskip

\begin{thm}\label{main1} Let $B$ be a Bratteli diagram of rank $k$ such
that the incidence matrices $F_n$ have the property of regularly vanishing
blocks, see Definition \ref{def vanishing blocks}. If $B$ has exactly
$l\ (1 \leq l \leq k)$ ergodic invariant probability measures, then, after
 telescoping, the set $V_n$ can be partitioned into subsets $\{V_{n,1},
 \ldots,V_{n,l},V_{n,0}\}$  such that

\smallskip
(a) $V_{n,i} \neq \emptyset$ for $i = 1,\ldots,l$;

\smallskip
(b) $|V_{n,i}|$ does not depend on $n$, i.e.,
$|V_{n,i}| = k_i$ for $i = 0,1, \ldots,l$ and $n \geq 1$;

\smallskip
(c) for $j = 1,\ldots,l$,
$$
\sum_{n = 1}^\infty \left(1 - \min_{v \in V_{n+1,j}}\sum_{w \in V_{n,j}}
f_{vw}^{(n)}\right)< \infty;
$$

\smallskip
(d)  for $j = 1,\ldots,l$,
$$
\max_{v,v' \in V_{n+1,j}} \sum_{w \in V_{n}} \left|f_{vw}^{(n)} -
f_{v'w}^{(n)}\right| \rightarrow 0
$$
as $n \rightarrow \infty$;

\smallskip

\ \ (e1)
for every $w \in V_{n,0}$
$$
\mathrm{vol}_{l} S(\ov a_1^{(n)}, \ldots, \ov a_l^{(n)}, \ov y^{(n)}(w))
 \rightarrow 0
$$
as $n \rightarrow \infty$, where $S$ is a simplex with extreme points
$\ov a_1^{(n)}, \ldots, \ov a_l^{(n)}, \ov y^{(n)}(w)$, and
$\mathrm{vol}_l(S)$ stands for the volume of $S$;

 (e2) for every $v \in V_{n+1,0}$ and  for sufficiently large $n$,
there exists some $C>0$ such that, for every $j = 1, \ldots,l$,
$$F_v^{(n,j)} =  \sum_{w  \in V_{n,j}}  f_{vw}^{(n)} < 1 -C.
$$

\end{thm}

\begin{proof} First, assume that $l = 1$, i.e., the diagram $B$ is uniquely
 ergodic. In this case we can set $V_{n,1} = V_n$ for every $n \geq 1$.
  Then  conditions (a) - (c) are obviously true. Condition (d) coincides
 with~(\ref{2.3b}) (see Proposition~\ref{uniq_erg}), and we have $V_{n,0} =
 \emptyset$ for  every $n$.

\textit{Proof of $(a)$ and $(b)$}.
Consider now the case when  $l > 1$. We construct partitions $\{V_{n,1},
\ldots,V_{n,l},V_{n,0}\}$ as in Theorem~\ref{prop subdiagrams}. It
follows that conditions $(a)$ and $(b)$ are  satisfied because they are
proved in Theorem \ref{prop subdiagrams}.

\textit{Proof of $(c)$}. Fix any $\varepsilon > 0$.
Then, by property~$(2)$ of Corollary~\ref{estimate2}, we obtain that
 for every sufficiently large $n$, any $m = 0,1,\ldots$, any $j \in
 \{1,\ldots, l\}$,  and any $v \in V_{n+m+2,j}$, the following estimate
 holds:
\begin{eqnarray}\label{4}\nonumber
\varepsilon  & \geq & \sum_{w \in V_n \setminus
V_{n,j}} g_{v w}^{(n+m,n)}\\ \nonumber
& =&  \sum_{u \in V_{n+m+1}}f_{vu}^{(n+m+1)} \sum_{w \notin V_{n,j}}
 g_{uw}^{(n+m,n)} \\
&=  &\sum_{u \notin V_{n+m+1,j}}f_{vu}^{(n+m+1)} \sum_{w \notin
 V_{n,j}}g_{uw}^{(n+m,n)}\\
&   \ \ \ \ & + \sum_{u \in V_{n+m+1,j}}f_{vu}^{(n+m+1)} \sum_{w
\notin V_{n,j}}g_{uw}^{(n+m,n)}. \nonumber
\end{eqnarray}

\medskip
\textit{\textbf{Claim 1}.} For sufficiently large $n$  and for every $m =
0,1, \ldots$, there exists a constant $C > 0$ such that
\begin{equation}\label{5}
\sum_{w \notin V_{n,j}}g_{uw}^{(n+m,n)} \geq C
\end{equation}
whenever $u \notin V_{n+m+1,j}$.

\textit{Proof of Claim 1.} If $u \in V_{n+m+1,i}$ for some $i = 1,\ldots,l$
and $i \neq j$, then Claim 1 is satisfied in virtue of  Property (1) of
 Corollary~\ref{estimate2}.

Suppose that $u \in V_{n+m+1,0}$. Assume for contrary that
$\sum_{w \notin V_{n,j}}g_{uw}^{(n+m,n)}\to 0 $ for an infinite
subsequence $(n_k)$.
We telescope the diagram with respect to these levels.
 Then
 $$\sum_{w \in V_{n,j}}g_{uw}^{(n+m,n)} \to 1
 $$ as $n \rightarrow   \infty$,  and by
 Theorem~\ref{estimate1}, we conclude that
 $d(\ov y^{(n+m+1)}(u), \ov y_j)\rightarrow
 0$  as  $n\to \infty$. This contradiction proves   Claim 1.
 \medskip

We obtain from (\ref{4}),~(\ref{5})
$$
\varepsilon \geq C \sum_{u \notin V_{n+m+1,j}}f_{vu}^{(n+m+1)} +
\sum_{u \in V_{n+m+1,j}}f_{vu}^{(n+m+1)} \sum_{w \notin V_{n,j}}
g_{uw}^{(n+m,n)}.
$$
For $u \in V_{n+m+1,j}$, we have
$$
\begin{aligned}
\sum_{w \notin V_{n,j}} g_{uw}^{(n+m,n)} &= \sum_{v' \in V_{n+m}}
f_{uv'}^{(n+m)} \sum_{w \notin V_{n,j}}g_{v'w}^{(n+m-1,n)} \\
& =  \sum_{v' \notin V_{n+m,j}}f_{uv'}^{(n+m)} \sum_{w \notin V_{n,j}}
 g_{v'w}^{(n+m-1,n)}\\
 & \ \ \ \  + \sum_{v' \in V_{n+m,j}}f_{uv'}^{(n+m)} \sum_{w
 \notin V_{n,j}}g_{v'w}^{(n+m-1,n)}.
\end{aligned}
$$
Using the same arguments as above, we get
$$
\begin{aligned}
\varepsilon  & \geq C \sum_{u \notin V_{n+m+1,j}}f_{vu}^{(n+m+1)} + C
\sum_{u \in V_{n+m+1,j}}f_{vu}^{(n+m+1)} \sum_{v' \notin V_{n+m,j}}
f_{uv'}^{(n+m)} \\
& \ \ \ \ + \sum_{u \in V_{n+m+1,j}}f_{vu}^{(n+m+1)}\sum_{v' \in
V_{n+m,j}} f_{uv'}^{(n+m)} \sum_{w \notin V_{n,j}}g_{v'w}^{(n+m-1,n)}.
\end{aligned}
$$
A similar computation allows us to deduce
$$
\begin{aligned}
\varepsilon  & \geq C \; \left[\sum_{u \notin V_{n+m+1,j}} f_{vu}^{(n+m+1)} +
\sum_{u_1 \in V_{n+m+1,j}, u_2 \notin V_{n+m,j}} f^{(n+m+1)}_{vu_1}
f^{(n+m)}_{u_1 u_2} + \ldots +\right.\\
& \ \ \ \ \ \ \ \ + \left. \sum_{(u_1,\ldots, u_m) \in V_{n+m+1,j}\times \ldots \times V_{n
+1,j}, u_{m+1} \notin V_{n,j}}f^{(n+m+1)}_{vu_1} f^{(n+m)}_{u_1 u_2}
\cdots f^{(n)}_{u_m u_{m+1}}\right] \\
& = C \cdot M_v^{(n+m+1,n)},
\end{aligned}
$$
where by $M_v^{(n+m+1,n)}$ we denote the expression in quadratic
brackets from the inequality above.
Thus, we proved that for sufficiently large $n$, every $m = 0, 1, \ldots$,
and every $v \in V_{n+m+2,j}$,
\begin{equation}\label{M-estimate}
M_v^{(n+m+1, n)} \leq \frac{\varepsilon}{C}.
\end{equation}
Denote
$$
s_v^{(n+m+1)} = \sum_{u \notin V_{n+m+1,j}} f_{vu}^{(n+m+1)}.
$$
Then, by definition of $M_v^{(n+m+1, n)}$, we have
\begin{eqnarray*}
M_v^{(n+m+1, n)} &=& s_v^{(n+m+1)} + \sum_{u \in V_{n+m+1,j}}
f_{vu}^{(n+m+1)} M_u^{(n+m, n)}\\
&=& s_v^{(n+m+1)} + (1 - s_v^{(n+m+1)}) \sum_{u \in V_{n+m+1,j}}
\frac{f_{vu}^{(n+m+1)}}{1 - s_v^{(n+m+1)}} M_u^{(n+m, n)}.
\end{eqnarray*}
If we denote
$$
r_{vu}^{(n+m+1)} = \frac{f_{vu}^{(n+m+1)}}{1 - s_v^{(n+m+1)}},
$$
then we see that
\begin{equation}\label{r_n}
\sum_{u \in V_{n+m+1,j}}r_{vu}^{(n+m+1)} = 1.
\end{equation}
Recall that $v$ is a vertex from $V_{n+m+2,j}$.
Using~(\ref{M-estimate}), we get
$$
\begin{aligned}
M_v^{(n+m+1, n)} & = s_v^{(n+m+1)} \left(1 -  \sum_{u \in V_{n+m+1,j}}
r_{vu}^{(n+m+1)}M_u^{(n+m, n)}\right) \\
& \ \ \  \ \ + \sum_{u \in V_{n+m+1,j}}
r_{vu}^{(n+m+1)}M_u^{(n+m, n)}\\
& \geq \frac{1}{2}s_v^{(n+m+1)} + \sum_{u \in
V_{n+m+1,j}}r_{vu}^{(n+m+1)}M_u^{(n+m, n)}\\
& \geq \frac{1}{2}\min_{v\in V_{n+m+2,j}} \sum_{u \notin V_{n+m+1,j}}
f_{vu}^{(n+m+1)}  + \sum_{u \in V_{n+m+1,j}}r_{vu}^{(n+m+1)}
M_u^{(n+m, n)}.
\end{aligned}
$$
Applying the same reasoning to $M_u^{(n+m, n)}$, $M_u^{(n+m-1, n)},
\ldots, M_u^{(n, n)}$ and using~(\ref{r_n}), we obtain

\begin{equation}\label{6}
\frac{\varepsilon}{C} \geq \frac{1}{2} \left(\min_{v\in V_{n+m+2,j}}
\sum_{u \notin V_{n+m+1,j}} f_{vu}^{(n+m+1)} + \ldots + \min_{v\in V_{n+1,j}} \sum_{u \notin V_{n,j}} f_{vu}^{(n)}
\right).
\end{equation}

Since $B$ is a diagram with regularly vanishing blocks, we can take $U_n =
V_{n,j}$, $U_n^c = V_n \setminus V_{n,j}$ and apply Definition
\ref{def vanishing blocks}.  It follows that
\begin{equation}\label{7}
\min_{v\in V_{n+1,j}} \sum_{u \notin V_{n,j}} f_{vu}^{(n)} \geq C_1
\max_{v\in V_{n+1,j}} \sum_{u \notin V_{n,j}} f_{vu}^{(n)}
\end{equation}
for sufficiently large $n$.

Applying (\ref{6}) and (\ref{7}), we obtain that for every $\varepsilon >
0$ there is some $N$ such that for  $n>N$, every $m$, and every $j = 1,
\ldots,l$

\begin{equation*}
\max_{v\in V_{n+m+2,j}} \sum_{u \notin V_{n+m+1,j}}
f_{vu}^{(n+m+1)} + \ldots + \max_{v\in V_{n+1,j}} \sum_{u \notin V_{n,j}} f_{vu}^{(n)} \leq \frac{2
\varepsilon}{C\cdot C_1}.
\end{equation*}

Hence, we finally obtain
$$
\sum_{n=1}^\infty \left(1 - \min_{v \in V_{n+1,j}} \sum_{u \in V_{n,j}}
f_{vu}^{(n)}\right) = \sum_{n=1}^\infty \max_{v \in V_{n+1,j}}\sum_{u
\notin V_{n,j}} f_{vu}^{(n)} < \infty.
$$
Thus, condition $(c)$ is proved.
\medskip

\textit{Proof of  $(d)$}. The proof is based on an application of the property
$$
\max_{u,u' \in V_{n,j}} || \ov y^{(n)}(u) - \ov y^{(n)}(u')|| \rightarrow 0
 \mbox { as } n \rightarrow \infty
$$
where $j = 1,\ldots,l$.

Recall that $d$ is the metric on $\mathbb R^k$  generated by the Euclidean
 norm $||\cdot||$, and the metric $ d^*(\ov x,\ov y) = \sum_{i=1}^k |x_i -
 y_i|$  ($\ov x, \ov y \in \mathbb R^k$)
is equivalent to $d$.  Then, for the
 standard  basis  $\{\ov e_i\}_{i = 1}^{k}$, we have
$$
d^*\left(\ov y^{(n)}(u),\ov y^{(n)}(u')\right) = \sum_{w \in V_1}\left|
g_{uw}^{(n,1)} - g_{u'w}^{(n,1)}\right|.
$$
Let $\{\varepsilon_n\}_{n=1}^\infty$ be a decreasing sequence of
positive numbers converging to zero.

\medskip
\textit{\textbf{Claim 2.}} There exist two sequences of natural numbers,
$\{n_i\}_{i = 0}^\infty$ and $\{m_i\}_{i = 1}^\infty$, such that
$$
1 = n_0 < n_1 < n_2 = n_1 + m_1 < n_3 = n_2 + m_2 < \ldots
$$
and
$$
\sum_{w \in V_{n_i}}\left|g_{uw}^{(n_i+m_i,n_i)} - g_{u'w}^{(n_i+m_i,n_i)}\right| <
 \varepsilon_i
$$
for $u,u' \in V_{n_i+m_i,j}$.
\medskip

\textit{Proof of Claim 2.}
We first choose $n_1>1$ such that
$$
\sum_{w\in V_{1}}\left|g_{uw}^{(n,1)} - g_{u'w}^{(n,1)}\right| < \varepsilon_1
$$
for all $n \geq n_1$ and all $u, u' \in V_{n,j}$. This is possible since both
$\ov y^{(n)}(u)$ and $\ov y^{(n)}(u')$ tend to $\ov y_j$ as $n \rightarrow
 \infty$  and $u,u' \in V_{n,j}$. We also have the relation
\begin{equation}\label{7a}
\ov y^{(n+n_1)}(u) = \sum_{w \in V_{n_1}} g_{uw}^{(n+n_1,n_1)} \ov
 y^{(n_1)}(w).
\end{equation}
There exists $n_2 > n_1$ such that for all $n \geq n_2$ and all $u,u' \in
 V_{n,j}$ the following inequality holds:
$$
d^*(\ov y^{(n+n_1)}(u),\ov y^{(n+n_1)}(u')) < \varepsilon_2.
$$
Since the metrics $d^*(E_1)$ and $d^*(E_2)$,  computed  with respect to
two different bases  $E_1$ and $E_2$ in
$\mathbb{R}^{k}$, are equivalent, we can work with the basis $\{\ov
 y_w^{(n_1)}\}_{w \in V_{n_1}}$. Hence, for $n_2$ sufficiently large, we
 get
$$
\sum_{w \in V_{n_1}}\left|g_{uw}^{(n_1 + m_1,n_1)} - g_{u'w}^{(n_1 +
m_1,n_1)}\right| < \varepsilon_2
$$
for $u,u' \in V_{n_2,j}$ and $m_1 = n_2 - n_1$.
Next, we find $n_3 > n_2$ such that for all $n \geq n_3$ and all $u,u' \in
V_{n_3},j$
$$
\sum_{w \in V_{n_2}}\left|g_{uw}^{(n_2 + m_2,n_2)} - g_{u'w}^{(n_2 +
m_2,n_2)}\right| < \varepsilon_3,
$$
where $m_2 = n_3 - n_2$. We continue the procedure and construct the
sequences
$\{n_i\}_{i = 0}^{\infty}$ and  $\{m_i\}_{i = 1}^{\infty}$
which satisfy the statement of Claim 2.

\medskip
To finish the proof of (d), we  telescope $B$ with respect to levels
$\{n_i\}_{i = 0}^{\infty}$, and we are done.
\medskip

\textit{Proof of  $(e1)$ and  $(e2)$}. We observe that formulas $(e1)$ and $(e2)$
are  consequences of
 condition (3) of Corollary~\ref{estimate2}. We note that condition $(e2)$
has  been already proved in Claim 1 above.
Here we prove $(e1)$.

By Lemma~\ref{remar1},  we have, for any $w \in V_{n,0}$,
$$
d^*(\ov y^{(n)}(w), \Delta_\infty^{(1)}) \rightarrow 0 \mbox{ as }
n \rightarrow \infty.
$$
Denote by $D$ the distance of $\ov y^{(n)}(w)$ to the $(l-1)$-dimensional
subspace  containing the simplex $\Delta_\infty^{(1)}$. Then
\begin{eqnarray*}
\mathrm{vol}_l S(\ov y_1,\ldots, \ov y_l, \ov y^{(n)}(w)) &=& \frac{1}{l} \, D \, \mathrm{vol}_{l-1} S(\ov y_1,\ldots, \ov y_l)\\
 &\leq& \frac{1}{l} \, d(\ov y^{(n)}(w), \Delta_\infty^{(1)}) \, \mathrm{vol}_{l-1} S(\ov y_1,\ldots, \ov y_l)\\
  &\rightarrow& 0
\end{eqnarray*}
as $n \rightarrow \infty$.
The above property and the conditions $\ov a_j^{(n)} \rightarrow \ov y_j$ as $n \rightarrow \infty$ imply $(e1)$.
 \end{proof}

\section{Bratteli diagrams of arbitrary rank and ergodic invariant
 measures} \label{sect 5}
The main result of this section is Theorem \ref{main_gen_case}. This
theorem
can be viewed as a converse statement to Theorem~\ref{main1}. We note
that Theorem \ref{main_gen_case} holds for for a wider set of Bratteli
 diagrams than Theorem \ref{main1}.

\subsection{A sufficient condition of unique ergodicity}

In this subsection, we focus on condition (\ref{2.3b}) of unique ergodicity of a Bratteli diagram and condition $(d)$ of Theorem~\ref{main1}, which corresponds to the unique ergodicity of a subdiagram. Our aim is to formulate more convenient sufficient conditions of unique ergodicity of a Bratteli diagram without using the telescoping of the matrices $F_n$. 
We assume that a Bratteli diagram $B$ satisfies conditions $(a) - (c)$ of Theorem~\ref{main1}, we do not assume that $B$ has the property of regularly vanishing blocks.  
Condition $(d)$ states that after telescoping of a Bratteli diagram $B$, which
satisfies  Theorem  \ref{main1}, one has
$$
\max_{v,v' \in V_{n+1,j}} \sum_{w \in V_{n}} \left|f_{vw}^{(n)} -
f_{v'w}^{(n)}\right| \rightarrow 0
$$
as $n \rightarrow \infty$ for $j = 1,\ldots,l$.

As was mentioned above, this  condition is a consequence of the fact that
\begin{equation}\label{alpha}
d^*(\ov y^{(n)}(u), \ov y^{(n)}(u')) \rightarrow 0
\end{equation}
as $n\rightarrow \infty$ whenever $u,u' \in V_{n,j}$ for some $j =
1,\ldots, l$.

It can be seen that relation (\ref{alpha}) is equivalent to the following:
\begin{equation}\label{beta}
\max_{u, u' \in V_{n,j}} \sum_{w \in V_1} \left|g_{uw}^{(n,1)} - g_{u'w}^{(n,
1)}\right| \rightarrow 0
\end{equation}
as $n\rightarrow \infty$.

Now we can  formulate and prove an assertion in terms of matrices $F_n$
 (without using telescoping) which implies~(\ref{beta}). We keep the
 same notation  that was used in the previous section.

 Let
$$
m_{n,j} = \min_{v\in V_{n+1,j}, w \in V_{n,j}} f_{vw}^{(n)}, \qquad M_{n,j}
= \max_{v\in V_{n+1,j}, w \in V_{n,j}} f_{vw}^{(n)},
$$
and
$$
m_w^{(n,j)} = \min_{v \in V_{n+1,j}}g_{vw}^{(n,1,j)}, \qquad M_w^{(n,j)} =
\max_{v \in V_{n+1,j}}g_{vw}^{(n,1,j)}, \qquad w \in V_{1,j},
$$
where $G_{(n,1,j)} = (g_{vw}^{(n,1,j)})_{v \in V_{n+1,j}, w \in V_{1,j}}$ is a
matrix obtained by multiplication of only those blocks of incidence matrices
$F_n, \ldots, F_1$ which correspond to the vertices from $\{V_{m,j}\}$ for
$m = 1, \ldots, n$. In other words,  we have
\begin{equation}\label{formula for h}
g_{vw}^{(n+1,1,j)} = \sum_{u \in V_{n+1,j}}f_{vu}^{(n+1)}g_{uw}^{(n,
1,j)}
\end{equation}
where $v \in V_{n+1,j}, w \in  V_{1,j}$.

\begin{thm}\label{1.6a}
Let $B$ be a Bratteli diagram of rank $k$ which satisfies conditions $(a)-(c)$
from Theorem~\ref{main1}. If 
$$
\sum_{n=1}^\infty m_{n,j} = \infty,
$$ 
then condition~(\ref{beta}) holds, and subdiagram $B_j$ corresponding to the vertices $\{V_{n,j}\}_{n = 1}^{\infty}$ is uniquely ergodic.
\end{thm}

\begin{proof}
First, we remark that if $|V_{n,j}| = 1$, then~(\ref{beta}) is true. So, we
assume, without loss of generality,  that $|V_{n,j}| \geq 2$. Further, if
$M_w^{(n,j)} = m_w^{(n,j)}$ for infinitely many $n$ and some $w \in V_n$
then, by (c), condition~(\ref{beta}) automatically holds. Thus, we assume
that $m_w^{(n,j)} < M_w^{(n,j)}$ for every $n \geq 1$ and every $w \in
V_n$. It follows that if $m_w^{(n,j)} = g_{vw}^{(n,1,j)}$ and $M_w^{(n,j)} =
g_{v'w}^{(n,1,j)}$ for some $v, v' \in V_{n+1,j}$, then $v \neq v'$.
For every $v \in V_{n+1,j}$, denote
$$
S_v^{(n,j)} = \sum_{u \in V_{n,j}} f_{vu}^{(n)}.
$$
It follows from condition (c) of Theorem~\ref{main1} that 
$S_v^{(n,j)} > 0$ 
where   $v \in V_{n+1,j}$ and $n$ is  large enough. In the following part of 
the proof we will implicitly assume  that $n$ is already chosen 
sufficiently  large.
Using (\ref{formula for h}), we obtain that,  for every $v \in V_{n+2,j}$ and
 $w \in V_{1,j}$,
 $$
 \begin{aligned}
 M_w^{(n,j)} - g_{vw}^{(n+1,1,j)} & = \sum_{u \in V_{n+1,j}}
\frac{f_{vu}^{(n+1)}}{S_{v}^{(n+1,j)}}M_w^{(n,j)} - \sum_{u \in
V_{n+1,j}}  f_{vu}^{(n+1)} g_{uw}^{(n,1,j)}\\
& = \sum_{u \in V_{n+1,j}}\frac{f_{vu}^{(n+1)}}{S_{v}^{(n+1,j)}}M_w^{(n,j)}
- {S_{v}^{(n+1,j)}} \sum_{u \in V_{n+1,j}} \frac{f_{vu}^{(n+1)}}{S_{v}^{(n
+1,j)}} g_{uw}^{(n,1,j)}\\
 & = \sum_{u \in V_{n+1,j}}\frac{f_{vu}^{(n+1)}}{S_{v}^{(n+1,j)}}
(M_w^{(n,j)} - g_{uw}^{(n,1,j)}) \\
& \ \ \ \ \ \ + (1 - S_{v}^{(n+1,j)})\sum_{u \in V_{n
+1,j}} \frac{f_{vu}^{(n+1)}}{S_{v}^{(n+1,j)}} g_{uw}^{(n,1,j)}\\
& \geq \frac{m_{n+1,j}}{S_{v}^{(n+1,j)}}(M_w^{(n,j)} - m_w^{(n,j)}) + (1 -
S_{v}^{(n+1,j)})\sum_{u \in V_{n+1,j}} \frac{f_{vu}^{(n+1)}}{S_{v}^{(n
+1,j)}} g_{uw}^{(n,1,j)}\\
& \geq m_{n+1,j}(M_w^{(n,j)} - m_w^{(n,j)}) + (1 - S_{v}^{(n+1,j)})
\sum_{u \in V_{n+1,j}} \frac{f_{vu}^{(n+1)}}{S_{v}^{(n+1,j)}}
g_{uw}^{(n,1,j)}.
 \end{aligned}
 $$

Similarly, we can prove that
$$
g_{vw}^{(n+1,1,j)} - m_w^{(n,j)} \geq m_{n+1,j}(M_w^{(n,j)} - m_w^{(n,j)})
- (1 - S_{v}^{(n+1,j)})\sum_{u \in V_{n+1,j}} \frac{f_{vu}^{(n+1)}}
{S_{v}^{(n+1,j)}} g_{uw}^{(n,1,j)}.
$$
It follows from the above inequalities that
\begin{multline*}
m_w^{(n,j)} + m_{n+1,j} (M_w^{(n,j)} - m_w^{(n,j)}) - (1 - S_{v}^{(n+1,j)})
\sum_{u \in V_{n+1,j}} \frac{f_{vu}^{(n+1)}}{S_{v}^{(n+1,j)}} g_{uw}^{(n,
1,j)} \leq\\
\leq  g_{vw}^{(n+1,1,j)} \leq \\
\leq M_w^{(n,j)} - m_{n+1,j}(M_w^{(n,j)} - m_w^{(n,j)}) - (1 - S_{v}^{(n+1,j)})
\sum_{u \in V_{n+1,j}} \frac{f_{vu}^{(n+1)}}{S_{v}^{(n+1,j)}} g_{uw}^{(n,
1,j)}.
\end{multline*}
Therefore,
\begin{multline*}
m_w^{(n,j)} + m_{n+1,j}(M_w^{(n,j)} - m_w^{(n,j)})  - (1 - S_{v}^{(n+1,j)})
\sum_{u \in V_{n+1,j}} \frac{f_{vu}^{(n+1)}}{S_{v}^{(n+1,j)}} g_{uw}^{(n,
1,j)} \leq\\
\leq m_w^{(n+1,j)} \leq M_w^{(n+1,j)} \leq\\
\leq M_w^{(n,j)} - m_{n+1,j}(M_w^{(n,j)} - m_w^{(n,j)})  - (1 - S_{v}^{(n+1,j)})
\sum_{u \in V_{n+1,j}} \frac{f_{vu}^{(n+1)}}{S_{v}^{(n+1,j)}} g_{uw}^{(n,
1,j)}.
\end{multline*}
Hence we have
\begin{equation}\label{gamma}
M_w^{(n+1,j)} - m_w^{(n+1,j)} \leq (M_w^{(n,j)} - m_w^{(n,j)})(1 - 2 m_{n
+1,j}).
\end{equation}
Applying~(\ref{gamma}) finitely many times, we finally obtain
$$
M_w^{(n+1,j)} - m_w^{(n+1,j)} \leq (M_w^{(1,j)} - m_w^{(1,j)})
\prod_{s = 2}^{n+1} (1 - 2 m_{s,j}).
$$
The condition $\sum_{n=1}^\infty m_{n,j} = \infty$ implies that
$\prod_{s = 1}^{\infty} (1 - m_{s,j}) = 0$. This means that
$$
\lim_{n \rightarrow \infty}(M_w^{(n,j)} - m_w^{(n,j)}) = 0
$$
for each $w \in V_{1,j}$. The last condition is equivalent to (\ref{beta}). The unique ergodicity of $B_j$ follows from Theorem~\ref{uniq_erg} and the proposition is proved.

\end{proof}

The following criterion of unique ergodicity is an immediate  corollary of
Theorem \ref{1.6a}. It is important to observe that it is true for an
arbitrary Bratteli diagram.

\begin{corol}
Let $B$ be a Bratteli diagram with stochastic incidence matrices $F_n$ and
 $$
 m_n = \min_{v \in V_{n+1}, w \in V_n} f_{vw}^{(n)}.
 $$
 If  
 $$
 \sum_{n = 1}^{\infty} m_n = \infty,
 $$ 
 then $B$ is uniquely ergodic.
\end{corol}

\subsection{The main theorem II}\label{main_thm_inverse}
In this subsection, we define a class of Bratteli diagrams that  generalizes,
in some sense,  the class of Bratteli diagrams of finite rank. Then we prove
 the converse of Theorem~\ref{main1}.

We recall that a (vertex) subdiagram $B'$ of a Bratteli diagram $B$ is defined
 by a sequence of proper subsets of vertices $W_n \subset V_n$ and by the
 corresponding sequence of incidence matrices $F'_n =
 (f_{vw}^{(n)})_{v \in W_{n+1}, w \in W_n}$. Denote by $X_{B'}$ the set
of all infinite paths of $B'$. Then $X_{B'} \subset X_B$ is a closed subset of
 $X_B$. Let $\wh X_{B'}$ be the subset of paths in $X_B$ which are
equivalent to  paths from $X_{B'}$. Let $\mu'$ be any tail invariant
 probability measure on $B'$. Then $\mu'$ can be uniquely extended to an
  invariant (finite or infinite) measure
$\wh{\mu'}$ on $\wh X_{B'}$ such that $\wh{\mu'}|_{X_{B'}} =
{\mu'}|_{X_{B'}}$ (for more details see
\cite{AdamskaBezuglyiKarpelKwiatkowski2016}).

We will need the following  result from
 \cite[Theorem 2.2]{AdamskaBezuglyiKarpelKwiatkowski2016}
 which holds for arbitrary Bratteli diagram.

\begin{prop}[\cite{AdamskaBezuglyiKarpelKwiatkowski2016}]\label{1.6b}
Let $B$ be a Bratteli diagram, and let $B'$ be its subdiagram defined by a
 sequence of vertices $W_n$.
If
$$
\sum_{n=1}^\infty \max_{v \in W_{n+1}} \left(\sum_{w \notin W_n} f_{vw}^{(n)}\right) < \infty,
$$
then any tail invariant probability measure $\mu'$ on $X_{B'}$ extends to
a finite invariant measure $\wh {\mu'}$ on $\wh X_{B'}$.
\end{prop}

In what follows, we will assume that (after telescoping) a Bratteli diagram
$B$ admits a  partition
$$
V_n = \bigcup_{i = 0}^{l_n}V_{n,i}, \ \ n = 1,2,\ldots,
$$
into disjoint subsets $V_{n,i}$ such that $V_{n,i} \neq \emptyset$, for
$i = 1,\ldots,l_{n}$, and $l_n \geq 1$. Moreover, let
$$
L_{n+1} = \{1,\ldots,l_{n+1}\} = \bigcup_{i = 1}^{l_n}L_{n+1}^{(i)},
$$
where $L_{n+1}^{(i)} \neq \emptyset$ and  $L_{n+1}^{(i)} \cap
 L_{n+1}^{(j)} = \emptyset$ for $i \neq j$, $i,j = 1,\ldots, l_n$. Hence,
  for every $j = 1, \ldots, l_{n+1}$, there exists a unique $i = i(j) \in
  \{1,\ldots, l_n\}$ such that $j \in L_{n+1}^{(i)}$. Denote
$$
V_{n+1}^{(i)} = \bigcup_{j \in L_{n+1}^{(i)}} V_{n+1,j}
$$
for $1 \leq i \leq l_n$.

We can interpret the sets $L_n^{(i)}$, defined above, in terms of
 subdiagrams. For this, select a sequence $\ov i = (i_1,i_2, \ldots)$ such
  that  $i_1 \in L_1$, $i_2 \in L_2^{(i_1)}$, $i_3 \in L_3^{(i_2)}, \ldots$
  and  define a subdiagram $B_{\ov i} = (\ov V, \ov E)$, where
  $$
  \ov V = \bigcup_{n=1}^{\infty}V_{n,i_n} \cup \{v_0\}.
$$

Now we formulate conditions $(c1)$, $(d1)$, $(e1)$ which are analogues
of conditions $(c)$, $(d)$, $(e)$ used in  Theorem~\ref{main1}:

\medskip
$(c1)$
$$
\sum_{n = 1}^{\infty} \left(\max_{i \in L_{n}} \max_{v \in V_{n+1}^{(i)}}
 \sum_{w \notin V_{n,i}} f_{vw}^{(n)}\right) < \infty;
$$

\medskip
$(d1)$
$$
\max_{j \in L_{n+1}}\max_{v,v' \in V_{n+1,j}} \sum_{w \in V_{n}}
\left|f_{vw}^{(n)} - f_{v'w}^{(n)}\right| \rightarrow 0 \mbox{ as } n \rightarrow \infty;
$$

\medskip
$(e1)$ for every $v \in V_{n+1,0}$, \\
$(e1.1)$
$$
\sum_{w \in V_n \setminus V_{n,0}} f_{vw}^{(n)} \rightarrow 1
\ \ \mbox{ as } n \rightarrow \infty;
$$

\noindent $(e1.2)$ there exists $C>0$ such that $F_{vi}^{(n)} \leq 1 - C$
for every $i = 1,\ldots,l$, where
$$
F_{vi}^{(n)} = \sum_{w \in V_{n,i}}f_{vw}^{(n)}.
$$

\medskip

 Let
 $$\mathcal{L} = \{\ov i
  = (i_1, i_2, \ldots) \ : \ i_1 \in L_1, i_{n+1} \in L_{n+1}^{(i_n)}, n = 1,2,
  \ldots\}.
  $$
  We call such a  sequence $\ov i \in   \mathcal{L}$  a \textit{chain}.
  A finite chain
  $\ov i(m,n)$ is a sequence $\{i_{n+m+1}, \ldots, i_n\}$ such that
  $i_{n+s} \in L_{n+s}^{(i_{n+s-1})}$, $s = 1,\ldots, m+1$.
We remark that a Bratteli diagram $B = (V,E)$ of finite rank has the form
 described in this subsection (see Corollary~\ref{inv_main_1} below).

\begin{thm}\label{main_gen_case}
Let $B = (V,E)$ be a Bratteli diagram satisfying the conditions $(c1),(d1),
(e1)$. Then:

(1) for each $\ov i \in \mathcal{L}$, any measure $\mu_{\ov i}$ defined on
$B_{\ov i}$ has a finite extension $\wh{\mu}_{\ov i}$ on $B$,

(2) each subdiagram $B_{\ov i}$, $\ov i \in \mathcal{L}$, is uniquely
ergodic,

(3) after normalization, the measures $\wh{\mu}_{\ov i}$, $\ov i \in
\mathcal{L}$, form the set of all probability ergodic invariant measures on
$B$.
\end{thm}

\begin{proof}
(1) To prove finiteness of the measure extension $\wh\mu_i$
from $B_{\ov i}$, we apply Proposition~\ref{1.6b} with $W_n = V_{n,i_n}$
 for every $n$. Then we  have
$$
\sum_{n=1}^{\infty} \max_{v \in V_{n+1, i_{n+1}}} \left( \sum_{w \notin
 V_{n,i_n}}f_{vw}^{(n)}\right) \leq \sum_{n = 1}^{\infty} \left(\max_{i \in
  L_{n}} \max_{v \in V_{n+1}^{(i)}}
 \sum_{w \notin V_{n,i}} f_{vw}^{(n)}\right) < \infty.
$$
Therefore, any invariant probability measure $\mu$ on $X_{B_{\ov i}}$
extends to a finite measure $\wh{\mu}$ on $\wh X_{B_{\ov i}}$.

\medskip
(2) We show that there exists a unique probability invariant measure
 concentrated on the set $\wh X_{B_{\ov i}}$. Let $\mu = \mu_{\ov i}$ be
 a  probability measure on $X_B$ which can be obtained as the extension of
 a  measure defined on the subdiagram $B_{\ov i}$. Let $\{\ov q^{(n)} =
 (q_w^{(n)})_{w \in V_n}\}$ be a sequence of vectors such that $
 \mu(X_w^{(n)}) = q_w^{(n)}$. The condition $\mu(\wh X_{B_{\ov i}}) =
 1$
 implies that
$$
Q_1^{(n)} = \sum_{w \in V_{n,i_n}}q_w^{(n)} \rightarrow 1
$$
as $n \rightarrow \infty$. Fix $n \geq 1$ and apply (\ref{2.3ab}). Denoting
$i = i_{n+m+1}$, we obtain
$$
\begin{aligned}
\ov q^{(n)} & = \sum_{v \in V_{n+m+1,i}} q_v^{(n+m+1)}
\ov g_v^{(n+m,n)} + \sum_{v \notin V_{n+m+1,i}} q_v^{(n+m+1)}\ov
g_v^{(n+m,n)} \\
& = Q_1^{n+m+1} \sum_{v \in V_{n+m+1,i}} \frac{q_v^{(n+m+1)}}
{Q_1^{(n+m+1)}} \ov g_v^{(n+m,n)}  \\
& \ \ \ \ \  + (1 - Q_1^{n+m+1}) \sum_{v \notin V_{n+m+1,i}}
\frac{q_v^{(n+m+1)}}{1 - Q_1^{(n+m+1)}} \ov g_v^{(n+m,n)}\\
& = Q_1^{(n+m+1)} \ov Y_1^{(n+m+1)} + (1 -Q_1^{(n+m+1)})
\ov Y_2^{(n+m+1)}
\end{aligned}
$$
where $\ov Y_1^{(n+m+1)}$ and $\ov Y_2^{(n+m+1)}$ denote the
corresponding sums in the third equality above.

Let $\Delta_{m,i}^{(n)}$ be the convex polytope in $\Delta^{(n)}$ spanned
 by the vectors $\ov g_v^{(n+m,n)}$, $v \in V_{n+m+1,i}$. Then
 $\ov Y_1^{(n+m+1)} \in \Delta_{m,i}^{(n)}$. Since $Q_1^{(n+m+1)}
 \rightarrow 1$ as $m \rightarrow \infty$, we have $\ov Y_1^{(n+m+1)}
 \rightarrow \ov q^{(n)}$ as $m \rightarrow \infty$, hence
$$
d^*(\ov q^{(n)}, \Delta_{m,i}^{(n)}) \rightarrow 0
$$
as $m \rightarrow \infty$.
If $\mu'$ is another probability measure such that $\mu(\wh X_{B_{\ov i}})
 = 1$ and $q_w^{(n)'} = \mu(X_w^{(n)})$, then $d^*(\ov q^{(n)'},
 \Delta^{(n)}_{m,i}) \rightarrow 0$ as $m \rightarrow \infty$. Using the
 same arguments as in the proof of Theorem \ref{uniq_erg}, we see that
  condition $(d1)$ implies that $\mathsf{diam}(\Delta_{m,i}^{(n)})
  \rightarrow 0$
 as $m \rightarrow \infty$. Thus, $\ov q^{(n)} = \ov {q}^{(n)'}$ for every
  $n  = 1,2,\ldots$ and this implies $\mu = \mu'$. This equality means that
   $B_{\ov i}$ is uniquely ergodic.

\medskip
(3) Before passing to the proof of the property $(3)$, we need to prove
some additional
properties of the matrices $F_n$'s and $G_{(n+m,n)}$'s. The following
lemma is an analogue of part (1) of Corollary~\ref{estimate2}.

\begin{lemma}\label{prop4.3}
Let $1 \leq i_{n+m+1} \leq l_{n+m+1}$, and let $\ov i(m,n)$ be the
finite chain determined by $i_{n+m+1}$. Then, for every $v \in
V_{n+m+1,i'}$, $i' = i_{n+m+1}$, the following relation holds
$$
\sum_{w \in V_{n, i}} g_{v_m w}^{(n+m,n)} \geq C_n \rightarrow 1,\; \ \
n\rightarrow \infty,
$$
where $i = i_n$ is the last element of the chain $\ov i(m,n)$.
\end{lemma}

\begin{proof}
We have
$$
\begin{aligned}
\sum_{w \in V_{n,i}}g_{v w}^{(n+m,n)} & =
\sum_{(u_n,\ldots,u_1) \in
V_{n+m}\times\ldots\times V_{n+1}} f_{vu_n}^{(n+m)}
f_{u_n u_{n-1}}^{(n+m-1)}\cdots f_{u_2 u_1}^{(n+1)}\sum_{w \in
 V_{n,i}}f_{u_1w}^{(n)}\\
& \geq \sum_{(u_n,\ldots,u_1) \in V_{n+m, i_{n+m}}\times\ldots\times
 V_{n+1, i_{n+1}}} f_{vu_n}^{(n+m)}
 \cdots
  f_{u_2 u_1}^{(n+1)}\sum_{w \in V_{n,i}}f_{u_1w}^{(n)} \\
& \geq \left(\min_{u_1 \in V_{n+1}^{(i_n)}}\sum_{w \in V_{n,i}}
f_{u_1w}^{(n)}\right)
\cdots  \left(\min_{u_{n+1} \in V_{n+m+1}^{(i_{n+m})}}\sum_{u_{n} \in V_{n+m, i_{n+m}}}f_{u_{n+1}u_n}^{(m+n)}\right) \\
&\geq \prod_{s = n}^{\infty}\left(\min_{u \in V_{s+1}^{(i)}, 1 \leq i \leq l_s}\sum_{w \in V_{s,i}}f_{uw}^{(s)}\right)\\
& = C_n \rightarrow 1 \qquad \mbox{as} \ n \to \infty.
\end{aligned}
$$
 Thus, Lemma~\ref{prop4.3} is proved.
\end{proof}

We continue the proof of $(3)$. Let $\mu$ be an ergodic invariant probability
measure on $B$. We find a chain $\ov i$ such that the corresponding
subdiagram $B_{\ov i}$ supports $\mu$. Denote $q_w^{(n)} =
\mu(X_w^{(n)})$ and consider the sequence of vectors $\{\ov q^{(n)} =
(q_w^{(n)})_{w \in V_n}\}$, where each vector $\ov q^{(n)} \in
\Delta_{\infty}^{(n)}$ is considered as an extreme point of the convex set $
\Delta_{\infty}^{(n)}$, $n = 1,2,\ldots$

According to Lemma~\ref{remar3}, we can find a sequence of vectors $\{\ov
y^{(n,m)}(v_m)\}$, $v_m = v_m(n) \in V_{n+m+1}$ of $\Delta_m^{(n)}$
such that $\ov y^{(n,m)}(v_m) \rightarrow \ov q^{(n)}$ as $m \rightarrow
\infty$. Recall that
$$
\ov y^{(n,m)}(v_m) = F_n^T \cdots F_{n+m}^T (\ov e^{(n+m+1)}(v_m)) =
\ov g_{v_m}^{(n+m,n)}.
$$

Assume first that $v_m \notin V_{n+m+1,0}$ for infinitely many $m$. We
telescope the diagram so that $v_m \notin V_{n+m+1,0}$ for all $m$. We
will show that in this case $\mu$ coincides with the measure $\mu_i$
supported by subdiagram $B_{\ov i}$, where $B_{\ov i}$ is determined by
vectors $\{\ov y^{(n,m)}(v_m)\}$. First, we notice that there exists
a unique $1
\leq i_{n+m+1} \leq l_{n+m+1}$ such that $v_m \in V_{n+m+1, i_{n+m+1}}
$. The number $i_{n+m+1}$ determines the finite
 chain $\{i_{n+m+1}, i_{n+m}, \ldots, i_n\}$.

Then, by Lemma~\ref{prop4.3}, we have
\begin{equation}\label{4.1}
\sum_{w \in V_{n, i_n}} g_{v_m w}^{(n+m,n)} \geq C_n \rightarrow 1
\end{equation}
as $n \rightarrow \infty$. We choose $m_n$ such that,
for $m \geq m_n$, we have
\begin{equation}\label{eq d-star}
d^{*}(\ov g_{v_m}^{(n+m,n)}, \ov q^{(n)}) \leq \varepsilon_n
\end{equation}
 where $\varepsilon_n  \rightarrow 0$. This inequality implies that
$$
d^{*}(\ov g_{v_m}^{(n+m,n)}, \ov g_{v_s}^{(n+s,n)}) \leq 2 \varepsilon_n
$$
for $m,s \geq m_n$. 
Therefore,
\begin{equation}\label{4.2}
\sum_{w \in V_n}\left|g_{v_m w}^{(n+m,n)} - g_{v_s w}^{(n+s,n)}\right| \leq 2
\varepsilon_n
\end{equation}
for $m,s \geq m_n$.

We show that all vertices $v_m$ belong to the same chain.
Indeed, let 
$$
\{i_{n+s+1},i_{n+s}', i_{n+s-1}', \ldots, i_n'\}
$$
be the chain
determined by $i_{n+s+1}$ as above, $v_s \in V_{n+s+1, i_{n+s+1}}$.
Then condition (c1) and relation (\ref{4.2}) imply that $i_n' = i_n$. The
 property (\ref{eq d-star})  implies that
 $$
 g_{v_m w}^{(n+m,n)} \rightarrow q^{(n)}_w
 $$
 as $m \rightarrow \infty$ for every $w \in V_n$. Further,
we have from~(\ref{4.1}) that
$$
\sum_{w \in V_{n, i_n}} q^{(n)}_w = \lim_{m \rightarrow \infty}\sum_{w \in
V_{n, i_n}}g_{v_m w}^{(n+m,n)} \geq C_n.
$$
Since $C_n \rightarrow 1$, we obtain that
\begin{equation}\label{4.3}
\mu\left(\bigcup_{w \in V_{n, i_n}} X_w^{(n)}\right) \rightarrow 1
\end{equation}
as $n\rightarrow \infty$.
In this way we have found numbers $1 \leq i_n \leq l_n$ such
that~(\ref{4.3}) holds.
Using the same reasoning as above, we prove that $i_{n+1} \in
L_{n+1}^{(i_n)}$ for every $n \geq 1$, so $\ov i = (i_1, i_2, \ldots)$ forms
an infinite chain.
Therefore, we determined the subdiagram $B_{\ov i}$ corresponding to the
 chain $\ov i$, which supports measure $\mu$.
This allows us to conclude that $\mu = \mu_{\ov i}$.

To finish the proof of property $(3)$, we consider the case when $\ov
g_{v_m}^{(n+m,n)} \rightarrow \ov q^{(n)}$ as $m \rightarrow \infty$ and
$v_m \in V_{n+m+1,0}$ for infinitely many $m$, or, equivalently, for all $m
\geq 1$. Using $(e1)$, we can find, for every $v \in V_{n+1,0}$,  a vector
\begin{equation}\label{4.4}
\ov b^{(n)}(v) = \ov b^{(n)} = \sum_{s \notin V_{n+1,0}} B_s^{(n)}\ov
f_s^{(n)},
\end{equation}
where $\sum_{s \notin V_{n+1,0}} B_s^{(n)} = 1$ and $B_s^{(n)} \geq 0$
for all $s,n$, and such that
\begin{equation}\label{4.5}
d^*(\ov f_v^{(n)}, \ov b^{(n)}) \leq \varepsilon_n,
\end{equation}
where $\varepsilon_n \to 0 $ as $n\to \infty$.
We will show that, for every $m \geq 1$ and for every $v \in V_{n+m+1,0}$,
there exists a vector
$$
\ov b^{(n+m,n)}(v) = \ov b^{(n+m,n)} \in \mathsf{Conv}(\ov
g_s^{(n+m,n)}, \ s \in V_{n+m+1} \setminus V_{n+m+1,0})
$$
such that
$$
d^*(\ov g_v^{(n+m,n)}, \ov b^{(n+m,n)}) \leq \varepsilon_n \rightarrow 0,
\qquad \  n\to \infty.
$$
It follows from~(\ref{4.4}) and (\ref{4.5}) that, taking $n+m$ instead of $n$,
 we can find a vector
 $$\ov b^{(n+m)} = \sum_{s \notin V_{n+m+1,0}}
B_s^{(n+m)}\ov f_s^{(n+m)}
$$
where  $\sum_{s \notin V_{n+m+1,0}} B_s^{(n+m)} = 1$ and
$B_s^{(n+m)} \geq 0$, which  satisfies the relation
$$
\sum_{u \in V_{n+m}} \left|f_{vu}^{(n+m)} - \sum_{s \notin V_{n+m+1,0}}
B_s^{(n+m)}f_{su}^{(n+m)}\right| \leq \varepsilon_{n+m}.
$$
Define
$$
\ov b^{(n+m,n)}(v):= \sum_{s \notin V_{n+m+1,0}} B_s^{(n+1)} \ov
g_s^{(n+m,n)},
$$
where
$$\ov g_v^{(n+m,n)} = \sum_{u \in V_{n+m}} f_u^{(n+m)} \ov
g_u^{(n+m-1,n)}.
$$
Then we obtain
$$
\ov b^{(n+m,n)}(v) = \sum_{u \in V_{n+m}}\left(\sum_{s \notin
V_{n+m+1,0}}B_s^{(n+1)} f_{su}^{(n+m)}\right)\ov g_{u}^{(n+m-1,n)}.
$$

Combining the proved relations, we can finally obtain the following result:

$$
\begin{aligned}
d^*(\ov g_v^{(n+m,n)}, \ov b^{(n+m,n)}(v)) & = \sum_{w \in V_n}
| g_{vw}^{(n+m,n)} - b_{w}^{(n+m,n)}(v) | \\
& = \left. \sum_{w \in V_n} \right|\sum_{u \in V_{n+m}} f_{vu}^{(n+m)}
g_{uw}^{(n +m-1,n)} \\
& \ \ \ \ \ - \left.\sum_{u \in V_{n+m}}\left(\sum_{s \notin V_{n+m+1,0}}
 B_s^{(n+1)} f_{su}^{(n+m)}\right)g_{uw}^{(n+m-1,n)} \right|\\
 & = \left.\sum_{w \in V_n} \right|\sum_{u \in V_{n+m}}
 g_{uw}^{(n+m-1,n)}\times \\
 & \ \ \ \ \ \times \left(f_{vu}^{(n+m)}
  - \left.\sum_{s \notin V_{n+m+1,0}}B_s^{(n+1)} f_{su}^{(n+m)}\right)
  \right|\\
 & \leq  \left.\sum_{u \in V_{n+m}} \right|f_{vu}^{(n+m)} -
 \left.\sum_{s \notin V_{n+m
 +1,0}} B_s^{(n+1)} f_{su}^{(n+m)}\right| \times \\
 &\ \ \ \ \ \times \sum_{w \in V_n} g_{uw}^{(n+m-1,n)} \\
 & =  d^*(\ov f_v^{(n+m)}, \sum_{s \notin V_{n+m+1,0}}B_s^{(n+1)}
 \ov f_{s}^{(n+m)}) \\
 &\leq \varepsilon_{n+m} \leq \sup_{m \geq 1}\varepsilon_{n+m} =
 \varepsilon_n' \rightarrow 0.
\end{aligned}
$$
Here we used the fact that $\sum_{w \in V_n}g_{uw}^{(n+m-1,n)} = 1$.

On the other hand, we have $\ov b^{(n+1)}(v) \in \mathsf{Conv}(\ov
g_s^{(n+m,n)}, s \notin V_{n+m+1,0})$. Thus, we conclude that
 $\ov b^{(n+m,n)}
(v_m) \rightarrow \ov q^{(n)}$ as $m \rightarrow \infty$. The vector $\ov
q^{(n)}$, considered as  an extreme point of $\Delta_{\infty}^{(n)}$, is the limit
 vector of the
polytope $\Delta_{m,0}^{(n)} = \mathsf{Conv}(\ov g_s^{(n+m,n)}, s \notin
V_{n+m+1,0})$. Repeating the same arguments as in Lemma~\ref{remar3},
we obtain that $\ov b^{(n+m,n)}(v_m)$ must be one of the extreme points of
polytope $\Delta_{m,0}^{(n)}$, i.e. $\ov b^{(n+m,n)}(v_m) =
\ov g_{v_m'}^{(n+m,n)}$ for some $v_m' \notin V_{n+m+1,0}$. Therefore,
$v = v_{\ov i}$ for a chain $\ov i$, and this proves (3).
\end{proof}

Now we can prove the converse of Theorem~\ref{main1}, which we 
 obtain as a corollary of Theorem~\ref{main_gen_case}. 
We remark that the following result does not require the stochastic incidence matrices of a Bratteli diagram to have the property of regularly vanishing blocks.

\begin{corol}\label{inv_main_1}
Let $B$ be a Bratteli diagram of finite rank $k \geq 2$ with nonsingular 
stochastic incidence matrices $(F_n)$. Suppose that  after telescoping $B$ 
satisfies conditions $(a) - (e2)$ of Theorem~\ref{main1}. Then $B$ has $l$ 
ergodic probability  invariant measures.
\end{corol}

\begin{proof}
Since $B$ has a finite rank, we can identify the sets $V_n = V$ for $n = 1,2,
\ldots$ and we have $l_n = l$, $L_n = \{1, \ldots, l\}$, $L_n^{(s)} = \{s\}$ 
for every $n = 1,2,\ldots$ and $s = 1,
\ldots, l$. The chains $\ov s$ reduce to the form $(s,s,\ldots)$. Notice that
conditions $(a)-(e2)$ imply conditions $(c1), (d1), (e1.2)$, while condition
$(e1.1)$ is stronger than condition $(e1)$ and does not follow from $(a)-
(e2)$. We define subdiagrams $B_{\ov s}$ as before and repeat the proof of
properties $(1)$ and $(2)$ in the same way as in Theorem~\ref{main_gen_case} (we do not use the property $(e1.1)$ in this part of
the proof). To complete the proof, it suffices to show that the vectors $\ov
y_1,\ldots, \ov y_l$ form the set of all extreme points of $\Delta_{\infty}^{(1)}$
(the notations were defined in Subsection~\ref{subsec4.1}). Assume that a
vector $\ov z$ is an additional extreme point of $\Delta_{\infty}^{(1)}$. Then we
can find a sequence of extreme points $\ov y^{(n)}(w)$ of $\Delta_n^{(1)}$, $w
\in V_n$, such that $\ov y^{(n)}(w) \rightarrow \ov z$ as $n \rightarrow
\infty$. Condition $(e1)$ guarantees that $\ov z$ is a vector of the plane
generated by $\ov y_1, \ldots, \ov y_l$. The dimension of the simplex
$S(\ov y_1, \ldots, \ov y_l)$ equals to the dimension of the above plane. If
$\ov z$ does not belong to $S(\ov y_1, \ldots, \ov y_l)$ then
$\Delta_{\infty}^{(1)}$ is not a simplex. Therefore, $\ov z  \in S(\ov y_1,
\ldots, \ov y_l)$ and this implies Property $(3)$ of Theorem~
\ref{main_gen_case}. Hence, the number of the probability ergodic invariant measures on $B$ equals $l$.
\end{proof}

We will need the following definitions to state the next proposition.
Let $B = (V,E)$ be a Bratteli diagram of finite rank $k$. Suppose that, for
$n\in \mathbb N$, we have a subset $W_n \subset V_n$ such that
 $|W_n| = k' = \mathrm{const}$, where
$1 \leq k' < k$. Let $B' = (\ov W, \ov E)$ be the subdiagram of rank $k'$
generated by $\{W_n\}_{n = 1}^{\infty}$. Then, after renumbering
vertices, $B'$ can be viewed as a ``vertical'' subdiagram of $B = (V,E)$.

An infinite $\sigma$-finite measure $\mu$ on $X_B$ is called \textit{regular
infinite} if there exists a clopen set such that $\mu$ takes a finite
(non-zero) value on this set.

Suppose that the  matrices $F_n$'s of a Bratteli diagram $B$ of rank $k$ are nonsingular and satisfy conditions $(a)-(e2)$ of Theorem~\ref{main1}. Then, by Corollary~\ref{inv_main_1}, 
$B$ has exactly $l$ ergodic invariant probability measures 
$\{\wh{\mu}_j\}_{j = 1}^l$. After renumbering each $V_n = V$, we can 
think of the corresponding subdiagrams $B_j = B_{\ov i_j}$ as of 
vertical subdiagrams. Each measure $\wh{\mu}_j$ is an extension of the 
unique ergodic measure ${\mu}_j$ on the subdiagram $B_j$.
In~\cite[Theorem 3.3]{BezuglyiKwiatkowskiMedynetsSolomyak2013} the
 following result describing the structure of all ergodic invariant measures on
 a finite rank diagram $B = (V,E)$ was proved.

\begin{prop}[\cite{BezuglyiKwiatkowskiMedynetsSolomyak2013}]
\label{1.6c}
(I) Each ergodic measure $\mu$ (finite or infinite) on $X_B$ is obtained as
an extension of finite ergodic measure from some vertical subdiagram
$B_{\mu} = (V_{\mu},E_{\mu})$;

(II) The number of finite or regular infinite ergodic invariant measures is not
greater than $k$;

(III) One can telescope the diagram $B$ in such a way that $V_{\mu} \cap
V_{\nu} = \emptyset$ for different ergodic measures $\mu$ and $\nu$;

(IV) Given a probability ergodic invariant measure $\mu$, there exists a
constant $\delta > 0$ such that for any $v \in V_{\mu}$ and any level $n$
one has $\mu(X_v^{(n)}) \geq \delta$;

(V) Each subdiagram $B_{\mu} = (V_{\mu},E_{\mu})$ is simple and uniquely
ergodic;

(VI) For all $v \notin V_\mu$, one has
$$
\lim_{n \rightarrow \infty} \mu(X_v^{(n)}) = 0.
$$

\end{prop}

In the following proposition, we will show how the results of Sections 4 and
5 are related to Proposition \ref{1.6c}.

\begin{prop}\label{connect_BKMS13} 
Let $\wh{\mu}_j$, $j = 1,\ldots,l$ be as above. Then the measures 
$\wh{\mu}_j$, $j = 1,\ldots,l$, coincide with the ergodic measures, 
described in Proposition~\ref{1.6c}.
\end{prop}

\begin{proof}
Indeed, it is clear that the measures $\wh {\mu}_j$, $j = 1,\ldots,l$, 
satisfy conditions (I) - 
(III) and the condition of unique ergodicity in (V) of Proposition \ref{1.6c}. 
We need to prove that the measures and the corresponding subdiagrams 
satisfy conditions (IV), (VI), and the condition of simplicity in (V). 
 Recall that any $ \mathcal{E}$-invariant probability measure $\mu$ on $B$
  is determined by a vector $\ov q^{(1)}  \in 
  \Delta_\infty^{(1)}$. This vector
 defines a sequence of vectors $\ov q^{(n)} = (q_w^{(n)}) =
 (\mu(X_w^{(n)}))$, $w \in V_n$, such that, for $n = 1,2,\ldots$ and $m =
 0, 1, \ldots$, we have
$$
q_w^{(n)} = \sum_{v \in V_{n+m+1}}g_{vw}^{(n+m,n)}q_v^{(n+m+1)}.
$$
The vectors $\ov q^{(n)} = (q_w^{(n)})_{w \in V_n}$ also satisfy
 (\ref{1.3a}).

Without loss of generality, we can assume that the measure $\wh {\mu}_j$
is the ergodic invariant measure on $B$ defined by $\ov q^{(1)} = \ov y_j$ for some
$j = 1, \ldots, l$. Then we have $\ov y_j = \lim_{n \rightarrow \infty}\ov
y_w^{(n)}$ whenever $w \in V_{n,j}$. To find the vectors $\ov q^{(n)}$
corresponding to measure $\mu_j$, fix $n \geq 1$, we represent $\ov y_j$
as follows:
$$
\ov y_j = \lim_{m \rightarrow \infty} \ov y_v^{(n+m+1)} = \lim_{m
\rightarrow \infty} \sum_{w \in V_n}g_{vw}^{(n+m,n)}\ov y_w^{(n)} =
\sum_{w \in V_n} \ov y_w^{(n)} \lim_{m \rightarrow \infty} g_{vw}^{(n
+m,n)}.
$$
It follows from (\ref{1.3a}) that
$$
q_w^{(n)} = \lim_{m \rightarrow \infty} g_{vw}^{(n+m,n)}
$$
for $v \in V_{n+m+1,j}$.

To assure that the conditions (IV) and (V) are satisfied, we must
replace the sets $V_{n,j}, j = 1, \ldots, l$ by some subsets $V'_{n,j}$. 
Indeed, to prove that condition (IV) holds, we should show
that $\lim _{m \rightarrow \infty} g_{vw}^{(n+m,n)} \geq \delta$ for some
$\delta > 0$ and any $j = 1,\ldots,l$, and any $v \in V_{n+m,j}$  and 
$w \in V_{n,j}$. By condition (d) of Theorem~\ref{main1}, matrices $F_n$'s satisfy the condition
$$
\max_{v \in V_{n+1,j}}f_{vw}^{(n)} - \min_{v \in V_{n+1,j}}f_{vw}^{(n)}
\rightarrow 0
$$
for every $w \in V_{n,j}$.
Assume that for some $w \in V_{n,j}$, $n = 1,2,\ldots$ we have $\max
f_{vw}^{(n)} \rightarrow 0$ as $n \rightarrow \infty$. Set $V_{n,j}' =
V_{n,j} \setminus \{w\}$ and $V_{n,0}' = V_{n,0} \cup \{w\}$. Since $
\mu_j$ satisfies condition (c) of Theorem~\ref{main1}, is clear that the set
$V_{n,j}'$ is non-empty. Moreover, for $v \in V'_{n+m,j}$, the following
holds:
\begin{eqnarray*}
\sum_{w \notin V_{n,j}'}g_{vw}^{(n+m,n)} &=& \sum_{w \notin V_{n,j}'}
\sum_{u \in V_{n+1}} g_{vw}^{(n+m,n+1)} f_{uw}^{(n)}\\
&=& \sum_{u \in V_{n+1,j}} \sum_{w \notin V_{n,j}'} g_{vw}^{(n+m,n+1)}
f_{uw}^{(n)} + \sum_{u \notin V_{n+1,j}} g_{vw}^{(n+m,n+1)} \sum_{w
\notin V_{n,j}'} f_{uw}^{(n)}\\
&\leq& \max_{u \in V_{n+1,j}}\sum_{w \notin V_{n,j}'} f_{uw}^{(n)} +
\max_{v \in V_{n+m+1,j}}\sum_{u \notin V_{n+1,j}} g_{vu}^{(n+m,n+1)}\\
&\rightarrow& 0 \mbox{ as } n \rightarrow \infty
\end{eqnarray*}
uniformly with respect to $m$.

Now one can repeat the same reasoning as in the proof of condition (c) of 
Theorem~\ref{main1} and get
$$
\sum_{n=1}^\infty \left(1 - \min_{v \in V_{n+1,j}'} \sum_{w \in V_{n,j}'}
f_{vw}^{(n)}\right) < \infty.
$$

We construct the sets $V_{n,j}$, $j = 1, \ldots, l$ and $V_{n,0}$ as in the
proof of Theorem~\ref{main1}. If there are infinitely many levels $n_k$ such
that $\max_{v \in V_{n+1,j}}f_{vw}^{(n_k)} \geq \delta > 0$ for $k = 1,2,
\ldots$ and $w \in V_{n_k,j}$, then we telescope the diagram with respect
to $n_k$ and condition (IV) of Proposition~\ref{1.6c} is proved. Otherwise,
if $\max_{v \in V_{n+1,j}}f_{vw}^{(n)}\rightarrow 0$ as $n \rightarrow
\infty$ for some $w \in V_{n,j}$, then we replace $V_{n,j}$ with $V_{n,j}' =
V_{n,j} \setminus \{w\}$. Repeating this procedure finitely many times, we
find  non-empty sets $V_{n,j}$ satisfying condition (IV) of Proposition~\ref{1.6c}. The described reduction of the sets $V_{n,j}, j = 1, \ldots, l$ to the sets $V'_{n,j}, j = 1, \ldots, l$ also implies that the subdiagrams $B'_j$ are simple.

Now we prove condition (VI). If $w \notin V_{n,j}$ then $\max_{v \in V_{n+m,j}}g_{vw}^{(n+m,n)}
\rightarrow 0$ as $m \rightarrow \infty$ for $n = 1,2,\ldots$ and
$$
\lim_{n \rightarrow \infty} q_w^{(n)} = \lim_{n \rightarrow \infty}
\mu_j(X_w^{(n)}) = 0.
$$
Thus, condition (VI) holds.
\end{proof}

\begin{remar}
Theorem~\ref{main1} does not guarantee that the subdiagrams $B_i$
 corresponding to the vertices from $V_{n,i}$ are simple. But one can reduce
  these subdiagrams to the smallest possible ones, as was described above in
the proof of Proposition~\ref{connect_BKMS13}, and obtain simple uniquely 
ergodic subdiagrams. After reduction, the subdiagrams $B_i$ are the same 
subdiagrams that were considered in~\cite{BezuglyiKwiatkowskiMedynetsSolomyak2013}.  For instance, in 
Example~\ref{non-simple-stat-ex}, one can take $V_{n,1} = V_n$ and 
$V_{n,0} = \emptyset$ for all $n$. After reduction, we obtain that 
$V'_{n,1}$ consists only of the first vertex on each level $n$, and $V'_{n,0}$ 
consists of the second one.
\end{remar}

\section{Examples}\label{sect 6}

\subsection{Stationary Bratteli diagrams}
A Bratteli diagram of finite rank $B = (V,E)$ is called \textit{stationary},
if $\tl F_n = \tl F$ for $n = 1,2,\ldots$ The paper~\cite{BezuglyiKwiatkowskiMedynetsSolomyak2010} contains an explicit
description of all ergodic invariant probability measures on $B$. Using the
main results of~\cite{BezuglyiKwiatkowskiMedynetsSolomyak2010}, we can
indicate the sets $V_{n,j}$ and the ergodic measures $\wh{\mu}_j$, $j = 1,
\ldots,l$.

 For the reader's convenience, we recall the necessary definitions
and results from~\cite{BezuglyiKwiatkowskiMedynetsSolomyak2010}. In this
subsection, by $\ov x$ we denote a  vector, either column or row one, it will
be either mentioned explicitly, or understood from the context.

The incidence matrix $\tl F = (\tl f_{vw})_{v,w \in V}$ defines a directed
graph $G(\tl F)$ in a following way: the set of the vertices of $G(\tl F)$ is
equal to $V$ and there is a directed edge from a vertex $v$ to a vertex $w$
if and only if $\tl f_{vw} > 0$. The vertices $v$ and $w$ are equivalent (we
write $v \sim w$) if either $v = w$ or there is a path in $G(\tl F)$ from $v$
to $w$ and also a path from $w$ to $v$. Let $\mathcal{E}_1,\ldots,
\mathcal{E}_m$ denote all equivalence classes in $G(\tl F)$. We will also
identify $\mathcal{E}_{\alpha}$ with the corresponding subsets of $V$. We
write $\mathcal{E}_{\alpha} \succeq \mathcal{E}_{\beta}$ if either
$\mathcal{E}_{\alpha} = \mathcal{E}_{\beta}$ or there is a path in
$G(\tl F)$ from a vertex of $\mathcal{E}_{\alpha}$ to a vertex of
$\mathcal{E}_{\beta}$. We write $\mathcal{E}_{\alpha} \succ
\mathcal{E}_{\beta}$ if $\mathcal{E}_{\alpha} \succeq
\mathcal{E}_{\beta}$ and $\mathcal{E}_{\alpha} \neq
\mathcal{E}_{\beta}$. Every class $\mathcal{E}_{\alpha}$, $\alpha = 1,
\ldots,m$, defines an irreducible submatrix $\tl F_{\alpha}$ of $\tl F$
obtained by restricting $\tl F$ to the set of vertices from
$\mathcal{E}_{\alpha}$. Let $\rho_{\alpha}$ be the spectral radius of
$\tl F_{\alpha}$, i.e.
$$
\rho_{\alpha} = \max\{|\lambda|: \lambda \in
\mathsf{spec}(\tl F_{\alpha})\},
$$
where by $\mathsf{spec}(\tl F_{\alpha})$ we mean the set of all complex
numbers $\lambda$ such that there exists a non-zero vector $\ov x =
(x_v)_{v \in \mathcal{E}_{\alpha}}$ satisfying $\ov x \tl{F}_{\alpha} =
\lambda \overline x$.

A class $\mathcal{E}_{\alpha}$ is called \textit{distinguished} if
\begin{equation}\label{eq disting class}
\rho_{\alpha} > \rho_{\beta} \ \ \mathrm{whenever} \ \ \
\mathcal{E}_{\alpha} \succ \mathcal{E}_{\beta}
\end{equation}
(in~\cite{BezuglyiKwiatkowskiMedynetsSolomyak2010} the notion of being
distinguished is defined in an opposite way because it is based on the matrix
transpose to the incidence matrix).

The real number $\lambda$ is called a \textit{distinguished eigenvalue} if
 there exists a non-negative left-eigenvector $\ov x = (x_v)_{v \in V}$ such
 that $\ov x \tl F = \lambda \ov x$. It is known (Frobenius theorem) that $
 \lambda$ is a distinguished eigenvalue if and only if $\lambda =
 \rho_{\alpha}$ for some distinguished class $\mathcal{E}_{\alpha}$.
 Moreover, there is a unique (up to scaling) non-negative eigenvector
 $\ov  x(\alpha) = (x_v)_{v \in V}$, $\ov x(\alpha)\tl F = \rho_{\alpha}
 \ov  x(\alpha)$ such that $x_v > 0$ if and only if there is a path from a
  vertex  of $\mathcal{E}_\alpha$ to the vertex $v$. The distinguished class
   $\alpha$ defines a measure $\mu_{\alpha}$ on $B = (V,E)$ as follows:
$$
\mu_{\alpha}(X_v^{(n)}) = \frac{x_v}{\rho_{\alpha}^{n-1}} h_v^{(n)}, \ \
\ v \in V_n = V.
$$
The main result of~\cite{BezuglyiKwiatkowskiMedynetsSolomyak2010} says
 that the set $\{\mu_{\alpha}\}$, where $\alpha$ runs over all distinguished
 vertex classes, generates the simplex of  all  $\mathcal E$-invariant
 probability  measures on the stationary Bratteli diagram $B = (V,E)$.

In the next proposition, we relate the distinguished classes to the
subsets $V_{n,j}$ of $V$ defined in Section \ref{sect 4}.

\begin{prop} Let $B = (V,E)$ be a stationary Bratteli diagram and $V_{n,j}$, $j = 1,\ldots,l$ be subsets of vertices defined in Section~\ref{sect 4}. Then the distinguished classes $\alpha$ (as subsets of $V$)
coincide with the sets $V_{n,j}$, $j = 1,\ldots,l$.
\end{prop}

\begin{proof}
To show this, we represent the matrix $\tl F$ in the Frobenius normal form
(similarly to the way it was done in~
\cite{BezuglyiKwiatkowskiMedynetsSolomyak2010}):
$$
F =\left(
  \begin{array}{ccccccc}
    \tl F_1 & 0 & \cdots & 0 & 0 & \cdots & 0 \\
    0 & \tl F_2 & \cdots & 0 & 0 & \cdots & 0 \\
    \vdots & \vdots & \ddots & \vdots & \vdots & \cdots& \vdots \\
    0 & 0 & \cdots & \tl F_s & 0 & \cdots & 0 \\
    Y_{s+1,1} & Y_{s+1,2} & \cdots & Y_{s+1,s} & \tl F_{s+1} & \cdots & 0
    \\
    \vdots & \vdots & \cdots & \vdots & \vdots & \ddots & \vdots \\
    Y_{m,1} & Y_{m,2} & \cdots & Y_{m,s} & Y_{m,s+1} & \cdots & \tl F_m
    \\
  \end{array}
\right)
$$
where all $\{\tl F_i\}_{i = 1}^m$ are irreducible square matrices, and for any
$j = s+1,\ldots,m$, at least one of the matrices $Y_{j,u}$ is non-zero. All
classes $\{\mathcal{E}_{\alpha}\}_{\alpha = 1}^s$ $(s \geq 1)$, are
distinguished (there is no $\beta$ such that $\alpha > \beta$). For every
$\alpha \geq s+1$ such that $\mathcal{E}_{\alpha}$ is a distinguished class
and for every $1 \leq \beta < \alpha $ we have either $\mathcal{E}_{\beta}
\prec \mathcal{E}_{\alpha}$ and $\rho_{\beta} < \rho_{\alpha}$, or there
is no relation between $\mathcal{E}_{\alpha}$ and $\mathcal{E}_{\beta}$.

We note that, for a stationary Bratteli diagram, the entries of the
stochastic matrix $F_n$ are
\begin{equation}\label{eq entries}
f_{vw}^{(n)} = \tl f_{vw}\frac{h_w^{(n)}}{h_v^{(n+1)}},\ \ \ v \in
V_{n+1},  w \in V_n.
\end{equation}
We use the notation introduced in Sections 4 and 5 and denote by
$\tl G^{(n)} = (\tl g^{(n)}_{vw})_{v \in V_{n+1}, w \in V_1}$ the $n$-th
power $\tl F^n$ of the incidence matrix $\tl F$. By Perron-Frobenius
 theorem the following relation holds:
\begin{equation}\label{(a)}
a_{vw} := \lim_{n \rightarrow \infty}\frac{\tl g_{vw}^{(n)}}
{\rho_{\alpha}^{(n)}} > 0
\end{equation}
whenever $v \in \mathcal{E}_{\alpha}$, $w \in \mathcal{E}_{\beta}$, and
$\mathcal{E}_{\alpha} \succeq \mathcal{E}_{\beta}$
(see~\cite[Section 4]{BezuglyiKwiatkowskiMedynetsSolomyak2010}).

Now, consider the distinguished class $\mathcal{E}_{\alpha}$,
and let $v \in \mathcal{E}_{\alpha}$. We have
$$
h_v^{(n+1)} = \sum_{w \in V_1}\tl g_{vw}^{(n)} h_w^{(1)} =
\sum_{\mathcal{E}_{\beta} \preceq \mathcal{E}_{\alpha}}
\left(\sum_{w \in \mathcal{E}_{\beta}}\tl g_{vw}^{(n)} h_w^{(1)}\right).
$$
It follows  from (\ref{(a)}) that
\begin{equation}\label{(b)}
\lim_{n \rightarrow \infty}\frac{h_{v}^{(n+1)}}{\rho_{\alpha}^{(n)}} =
\sum_{\mathcal{E}_{\beta} \preceq \mathcal{E}_{\alpha}} \left(\sum_{w
\in \mathcal{E}_{\beta}}a_{vw} h_w^{(1)}\right)
\end{equation}
is a constant depending on $\alpha$.
Further, using~(\ref{(b)}), we compute
\begin{eqnarray*}
\sum_{w \notin \mathcal{E_{\alpha}}} f_{vw}^{(n)} &=&  \sum_{w \notin
\mathcal{E_{\alpha}}}\tl f_{vw}\frac{h_w^{(n)}}{h_v^{(n+1)}} \\
 &=&  \sum_{\mathcal{E}_{\beta} \prec \mathcal{E}_{\alpha}}
 \left(\sum_{w \in \mathcal{E}_{\beta}}\tl f_{vw}\frac{h_w^{(n)}}{h_v^{(n
 +1)}} \right)\\
 &=&   \sum_{\mathcal{E}_{\beta} \prec \mathcal{E}_{\alpha}}
 \left(\sum_{w \in \mathcal{E}_{\beta}} \tl f_{vw} \frac{h_w^{(n)}}
 {\rho_{\beta}^{n}} \frac{\rho_{\alpha}^{n}}{h_v^{(n+1)}}
 \left(\frac{\rho_{\beta}}{\rho_{\alpha}}\right)^{n} \right)\\
 &\leq& C \left(\frac{\max_{\mathcal{E}_{\beta} \prec
 \mathcal{E}_{\alpha}}\rho_{\beta}}{\rho_{\alpha}}\right)^{n},
 \end{eqnarray*}
where $C$ is a positive constant. Thus, we conclude using
(\ref{eq disting class}) that
$$
\sum_{n = 1}^{\infty}\left(\max_{v \in \mathcal{E}_{\alpha}}\sum_{w
\notin \mathcal{E}_{\alpha}}f_{vw}^{(n)}\right) < \infty.
$$

In other words, we have shown that the distinguished classes (sets of
vertices) $\mathcal{E}_{\alpha}$ satisfy condition $(c)$ of Theorem~
\ref{main1}. Show that these classes satisfy also condition $(d)$ of
Theorem \ref{main1}.
Indeed, it follows from (\ref{(b)}) that, for $v,w \in \mathcal{E}_{\alpha}$,
$$
\lim_{n \rightarrow \infty} f_{vw}^{(n)} = \lim_{n \rightarrow \infty}\tl
f_{vw}\frac{h_w^{(n)}}{h_v^{(n+1)}}
= \tl f_{vw} \frac{1}{\rho_{\alpha}}
$$
We can assume, without loss of generality,  that $\tl f_{vw} \geq 1$ for all
$v,w \in \mathcal{E}_{\alpha}$ since $\tl F_{\alpha}$ is an irreducible
matrix. Then we apply Theorem~\ref{1.6a} to get condition $(d)$.

Next, we show that the non-distinguished classes $\mathcal{E}_{\alpha}$ do
not satisfy condition $(c)$. We can find a class $\mathcal{E}_{\beta_0}
\prec \mathcal{E}_{\alpha}$ such that $\rho_{\beta_0} \geq
\rho_{\alpha}$. By~(\ref{(b)}), for any $v \in \mathcal{E}_{\alpha}$ we
have
$$
\sum_{w \notin \mathcal{E}_{\alpha}} f_{vw}^{(n)} \geq \sum_{w \in
\mathcal{E}_{\beta_0}}  f_{vw}^{(n)} = \sum_{w \in
\mathcal{E}_{\beta_0}}  \tl f_{vw}\frac{h_w^{(n)}}{h_{v}^{(n+1)}} \geq
C' \left(\frac{\rho_{\beta_0}}{\rho_{\alpha}}\right)^n \geq C'
$$
for some positive constant $C'$.
Thus,
$$
\sum_{n = 1}^{\infty}\left(\max_{v \in \mathcal{E}_{\alpha}}\sum_{w
\notin \mathcal{E}_{\alpha}}f_{vw}^{(n)}\right) = \infty
$$
and condition $(c)$ is not satisfied.

To finish the proof that the sets $V_{n,j}$, $j = 1,\ldots,l$, coincide with the
distinguished classes $\mathcal{E}_{\alpha}$, it remains to show that each
set $V_{n,j}$ is contained in some equivalence class
$\mathcal{E}_{\alpha}$. Let $v,w \in V_{n,j}$. Condition $(c)$ implies that
$f_{vw}^{(n)} > 0$ for sufficiently large $n$ (the sets $V_{n,j}$ are
minimal sets satisfying Theorem~\ref{main1}). By (\ref{eq entries}),
we see that $f_{vw}^{(n)} > 0$ implies that $\tl f_{vw}^{(n)} > 0$
 for all $v,w \in V_{n,j}$, i.e., $v \sim w$ and
$V_{n,j} \subset \mathcal{E}_{\alpha}$ for some $\alpha = 1,\ldots,m$.
\end{proof}

\subsection{Pascal-Bratteli diagrams}

Theorem~\ref{main_gen_case} defines a class of Bratteli diagrams $B =
(V,E)$ such that the set $\mathcal{M}_1(B)$ of all invariant probability
measures coincides with the set $\mathcal{L}$ of all infinite chains. Each
ergodic probability invariant measure $\wh{\mu}_{\ov i}$ is an extension of
a unique invariant measure $\mu_{\ov i}$ from the subdiagram $B_{\ov i}$,
and the sets $X_{B_{\ov i}}$ are pairwise disjoint.

Consider the
\textit{Pascal-Bratteli} diagram $B_p = (V,E)$, see e.g.
\cite{MelaPetersen2005} or \cite{Vershik_2011, Vershik_2014,
FrickPetersenShields2017} for more information.
Our goal is to show that  the set $\mathcal{M}_1(B_p)$ has different
structure  than that of the set $\mathcal{L}$. Recall that, for the
Pascal-Bratteli diagram, we have $V_n = \{0,1,\ldots, n\}$ for $n = 0,1,
\ldots$, and the entries $\tl f_{ki}^{(n)}$ of the incidence matrix
$\tl F_n$  are of the form
$$
\tl f_{ki}^{(n)} =
\left\{
\begin{aligned}
1 &, \mbox{ if } i = k \mbox{ for } 0 \leq k < n+1,\\
1 &, \mbox{ if } i = k - 1 \mbox{ for } 0 < k \leq n+1,\\
0 &, \mbox{ otherwise}.\\
\end{aligned}
\right.
$$
where $k = 0, \ldots, n+1$,  $i = 0,\ldots,n$.
Moreover,
$$
h_i^{(n)} = {n \choose i},
$$
for $i = 0,\ldots,n$. Then we find the entries of the stochastic matrix
$F_n$:
\begin{equation}\label{Pascal_stochastic}
f_{ki}^{(n)} =
\left\{
\begin{aligned}
\frac{k}{n+1} &, \mbox{ if } i = k -1 \mbox{ and } 0 < k \leq n+1,\\
1 - \frac{k}{n+1} &, \mbox{ if } i = k \mbox{ and } 0 \leq k < n+1,\\
0 &, \mbox{ otherwise}.\\
\end{aligned}
\right.
\end{equation}
It is known~(see e.g. \cite{MelaPetersen2005}) that each ergodic invariant
 probability measure has the form $\mu_p$, $0 < p < 1$, where
$$
\mu_p\left(X_i^{(n)}\right) = {n \choose i}p^i (1-p)^{n-i}, \quad i = 0,
\ldots,n.
$$

\begin{prop} For the Pascal-Bratteli diagram,
the set $\mathcal{L}$ of all infinite chains $\overline{i}$ is empty.
\end{prop}

\begin{proof}
Assume that there exists a sequence of partitions $\{V_{n,0}, V_{n,i}, i = 1,
\ldots,l_n\}$ of $V_n$ satisfying conditions $(c1),(d1),(e1)$ defined in
Subsection~\ref{main_thm_inverse}. Recall that we denote by
$\ov f_j^{(n)}$ the vector  $(f_{ji}^{(n)})$, $i \in V_n$.
It is easy to see that~(\ref{Pascal_stochastic}) implies
$$
d^{*}(\ov f_j^{(n)}, \ov f_{j'}^{(n)}) = \sum_{i = 0}^{n} |f_{ji}^{(n)} -
f_{j'i}^{(n)}| \geq \frac{1}{2}
$$
whenever $j \neq j'$, $j,j' \in V_{n+1}$.
Thus,  it follows from condition $(d1)$ that $V_{n,i}$ is a single point in
the set  $V_n$ whenever $i = 1, \ldots, l_n$.

Suppose that $1 \leq s \leq n$ is an element of $V_{n+1}$ such that $s \in
L_{n+1}$. By condition $(c1)$, there exists a unique element $r = r(s) \in
\{1,\ldots,l_n\}$ such that
\begin{equation}\label{6.4}
\sum_{n = 1}^{\infty}\sum_{u \notin V_{n,r}} f_{su}^{(n)} < \infty.
\end{equation}

Conditions (\ref{Pascal_stochastic}) and (\ref{6.4}) imply that $s-1, s \in
V_{n,r}$. Hence, $V_{n,r}$ contains at least two elements of $V_n$ for
 some $r \in L_n$, and we get a contradiction.

Therefore, we have $V_{n,0} \supset \{1,\ldots, n-1\}$.
Then, by $(e1.1)$, we have
$$
\sum_{i = 1}^{n - 1} f_{j,i}^{(n)} \rightarrow 0
$$
for every $j = 1,\ldots, n-1$. But this fact
 contradicts~(\ref{Pascal_stochastic}).
Thus, we have shown that there is no sequence of partitions $\{V_{n,0},
V_{n,i}, i = 1,\ldots,l_n\}$ of $V_n$ satisfying conditions $(c1),(d1),(e1)$.
\end{proof}

\subsection{A class of Bratteli diagrams with countably many ergodic
invariant measures}\label{subsection_cntbl_erg_meas}
 In this subsection, we present a class of
 Bratteli diagrams with countably infinite set of ergodic invariant measures.

To construct such diagrams, we  let $V_n = \{0, 1, \ldots, n\}$ for
$n = 0, 1, \ldots$, and let
$\{a_n\}_{n = 0}^{\infty}$ be a sequence of natural numbers such that
\begin{equation}\label{property1}
\sum_{n = 0}^{\infty} \frac{n}{a_n + n} < \infty.
\end{equation}
To define the edge set $r^{-1}(w)$ for every vertex $w$, we
use the following procedure.
 For $w\in V_{n+1}$ such that $w\neq n
+1$, the set $r^{-1}(w)$ consists of $a_n$ (vertical) edges connecting
$w \in V_{n+1}$ with the vertex $w \in V_n$ and by one edge connecting
 $w \in V_{n+1}$ with every vertex $u \in V_n$, $u \neq w$.
For $w=n+1$,  let $r^{-1}(w)$ contain $a_n$ edges connecting $w$
 with the vertex $n$ on level $V_n$ and by one edge connecting $w$ with
  all other vertices $u = 0,1, \ldots, n-1$ of $V_n$. Then
 $$
 |r^{-1}(w)| = a_n + n
 $$
for every $w \in V_{n+1}$ and every $n = 0,1, \ldots$.

  We observe that the Bratteli diagram defined above admits
 an order generating the Bratteli-Vershik homeomorphism, see
  \cite{HermanPutnamSkau1992}, \cite{GiordanoPutnamSkau1995}, or
\cite{BezuglyiKwiatkowskiYassawi2014}, \cite{BezuglyiKarpel2016} for
futher details. 
 In particular, we can use a so called \textit{consecutive} ordering. To 
 introduce a consecutive ordering, we define a linear 
 order $\leq$ on every set $r^{-1}(w)$, $w \in V_{n+1}$, $n  \geq 0$ in 
 such a way that the edges from $r^{-1}(w)$ are enumerated from left to 
 right as they appear in the diagram. Then, for every $e_1, e_2, e_3 \in r^{-1}(w)$ such that $e_1 \leq e_2 \leq e_3$ and $s(e_1) = s(e_3) = u$, 
 we have $s(e_2) = u$; and this order is consecutive by definition (see e.g.  \cite{Durand2010}). 
In particular, we will obtain that the minimal edge is  always  some edge 
between $w$ and vertex $0 \in V_{n}$ and the maximal edge is some edge 
between $w$  and the vertex $n \in V_n$.
  Then, it is easy to see that $X_B$ has the unique
 minimal infinite path passing through the vertices $0 \in V_n$, $n \geq 0$
 and the unique maximal infinite path passing through the vertices $n \in V_n
 $, $n \geq 0$. Thus, a Vershik map $\varphi_B \colon X_B \rightarrow
 X_B$ exists and it is minimal. Figure 1 below shows an example of such a 
 Bratteli  diagram. It is known that all minimal Bratteli-Vershik systems with a 
 consecutive ordering have entropy zero (see e.g. \cite{Durand2010}) hence 
 the system that we describe in this subsection has zero entropy.
 
\begin{figure}[ht]
\unitlength = 0.4cm
\begin{center}
\begin{graph}(28,14)
\graphnodesize{0.4}
\roundnode{V0}(1,13.5)
\roundnode{V11}(1,9)
\roundnode{V12}(10,9)

\bow{V11}{V0}{0.18}
\bow{V11}{V0}{-0.21}
\bow{V12}{V0}{0.07}
\bow{V12}{V0}{-0.06}
\freetext(1.1,11){$\cdots$}
\freetext(5.6,11){$\cdots$}

\roundnode{V21}(1,4.5)
\roundnode{V22}(10,4.5)
\roundnode{V23}(19,4.5)

\bow{V21}{V11}{0.18}
\bow{V21}{V11}{-0.21}
\edge{V22}{V11}
\edge{V21}{V12}
\bow{V22}{V12}{0.18}
\bow{V22}{V12}{-0.21}
\edge{V23}{V11}
\bow{V23}{V12}{0.06}
\bow{V23}{V12}{-0.07}
\freetext(1.1,6){$\cdots$}
\freetext(10.1,6){$\cdots$}
\freetext(16,6){$\cdots$}

\roundnode{V31}(1,0)
\roundnode{V32}(10,0)
\roundnode{V33}(19,0)
\roundnode{V34}(28,0)

\bow{V31}{V21}{0.18}
\bow{V31}{V21}{-0.21}
\freetext(1.1,1.5){$\cdots$}
\edge{V31}{V22}
\edge{V31}{V23}

\edge{V32}{V21}
\bow{V32}{V22}{0.18}
\freetext(10.1,1.5){$\cdots$}
\bow{V32}{V22}{-0.21}
\edge{V32}{V23}

\edge{V33}{V21}
\bow{V33}{V23}{0.18}
\freetext(19.1,1.5){$\cdots$}
\bow{V33}{V23}{-0.21}
\edge{V33}{V22}

\edge{V34}{V21}
\bow{V34}{V23}{0.06}
\freetext(24.9,1.5){$\cdots$}
\bow{V34}{V23}{-0.06}
\edge{V34}{V22}

\freetext(1,-1){$\vdots$}
\freetext(10,-1){$\vdots$}
\freetext(19,-1){$\vdots$}
\freetext(28,-1){$\vdots$}

\freetext(10,-2.5){Figure 1}

\end{graph}
\end{center}

\vspace{0.8 cm}
\end{figure}

Denote by $B_i = (W^{(i)}, E^{(i)})$, $i = 0,1,\ldots, \infty$, the
 subdiagrams of $B$ determined by the following sequences of vertices
 (taken consecutively from $V_0$, $V_1$, $\ldots$): for $B_0$,
 $W^{(0)} = (0,0,0, \ldots)$; for $B_i$,   $W^{(i)} = (0,1,\ldots,i-1,i,i,i
 \ldots)$ for $i = 1,2,\ldots$,  and for $B_{\infty}$,  $W^{(\infty)} =(0,1,2,
  \ldots)$. Then each $B_i$ is an odometer and $E^{(i)}$ is the set of all
  edges from $B$ that belong to $B_i$.

Let $\mu_i$ be the unique invariant (hence ergodic) probability measure on
 the odometer $B_i$. Then, by~(\ref{property1}), each measure $\mu_i$
  can
 be extended to a finite invariant measure $\wh{\mu}_i$ on the diagram $B$
 and it is supported by the set $\wh{X}_{B_i}$. The problem about
 finiteness of measure extension was discussed in detail in the papers
 \cite{BezuglyiKwiatkowskiMedynetsSolomyak2013},
 \cite{AdamskaBezuglyiKarpelKwiatkowski2016}. We use the same symbol
 $ \wh{\mu}_i$ to denote the normalized (probability)
measure obtained from the extension of $\mu_i$ for $i = 0,1,\ldots,\infty$.

\begin{prop}\label{countbl_erg_meas}
The measures $\wh{\mu}_i$, $i = 0, 1, \ldots, \infty$, form a set of all
ergodic probability invariant measures on the Bratteli diagram $B = (V,E)$
defined above.
\end{prop}

\begin{proof}
In the same way as in Theorem~\ref{main_gen_case}, one can show that
each $\wh{\mu}_i$ is the unique invariant measure on the set
$\wh{X}_{B_i}$ therefore $\wh{\mu}_i$, $i = 0,1,\ldots,\infty$, is a
probability ergodic measure on $B$. Thus, the proof will be complete
if we show that any ergodic invariant probability measure $\mu$ on $B$
 coincides
with one of the measures $\wh{\mu}_i$, $i = 0,1,\ldots,\infty$.
The incidence matrices $\tl F_n$ of $B$ have the following form
$$
\tl F_n =
\begin{pmatrix}
a_n & 1 & 1 & \ldots & 1 & 1\\
1 & a_n & 1 & \ldots & 1 & 1\\
\vdots & \vdots & \vdots & \ddots & \vdots & \vdots \\
1 & 1 & 1 & \ldots & 1 & a_n\\
1 & 1 & 1 & \ldots & 1 & a_n\\
\end{pmatrix}
$$
for $n = 1,2,\ldots$
Therefore we see that
$$
\sum_{w = 0}^n \tl f_{vw}^{(n)} = a_n + n, \qquad v = 0,1,\ldots, n + 1,
$$
and $B$ is an $ERS$ Bratteli diagram with $r_n = a_n + n$ for $n \geq 0$.
Hence,
$$
h_v^{(n+1)} = a_0 (a_1 + 1) \cdots (a_n + n), \ \ v = 0,1,\ldots, n+1,
n\geq 0,
$$
 and
$$
\frac{h_w^{(n)}}{h_v^{(n + 1)}} = \frac{1}{a_n + n}
$$
where $w \in V_n$, $v \in V_{n+1}$ and $n = 1,2,\ldots$.
Furthermore, the entries of the stochastic matrices $F_n, n \geq 0,$ are
$$
f_{vw}^{(n)} = \frac{\tl f_{vw}^{(n)}}{a_n + n} =
\left\{
\begin{aligned}
\frac{a_n}{a_n + n} &, \mbox{ if } v = w, w \in \{0,1,\ldots, n\},
 \mbox{ or } v = n + 1, w = n,\\
\frac{1}{a_n + n} &, \mbox{ otherwise}.\\
\end{aligned}
\right.
$$

Consider the matrices $G_{(n+m,n)} = F_{n+m}\cdots F_n$ for $n \geq 1$
and $m \geq 1$. Using~(\ref{property1}) and repeating the same arguments
as in Lemma~\ref{prop4.3}, we get
\begin{equation}\label{property2}
g_{vw}^{(n+m,n)} \geq \prod_{s = n}^{\infty} \frac{a_s}{a_s + s} =
C_n \to 1 \mbox{ as } n\rightarrow \infty,
\end{equation}
in the case when $v = w$, $w \in \{0,1,\ldots, n\}$,
 or in the case when $v = n + 1, \ldots, n + m + 1$, $w = n$ and
 $m = 1,2, \ldots$

We recall that the convex polytope $\Delta_m^{(n)}$ is spanned by the
probability vectors $\ov y_v^{(n+m,n)} =
(g_{vw}^{(n+m,n)})_{w = 0}^{n}$, $v = 0,1,\ldots, n+m+1$.
Inequalities~(\ref{property2}) imply that
$$
d^{*}(\ov y_v^{(n+m,n)}, \ov e_v^{(n+1)}) \leq 2 (1 - C_n), \ \ \
v = 0,1,\ldots, n
$$
and
$$
d^{*}(\ov y_v^{(n+m,n)}, \ov e_n^{(n+1)}) \leq 2 (1 - C_n), \ \ \
 v = n + 1,\ldots, n + m + 1
$$
where $m = 1,2,\ldots$
Therefore the set $\Delta_{\infty}^{(n)}$ has exactly $(n + 1)$ extreme points
$\ov q_v^{(n)}$, where
$$
\ov  q_v^{(n)} = \lim_{m \rightarrow \infty} \ov y_v^{(n+m,n)}.
$$
Moreover,
$$
d^{*}(\ov q_v^{(n)}, \ov e_v^{(n+1)}) \leq 2 (1 - C_n).
$$
In fact, $\Delta_{\infty}^{(n)}$ is a simplex in $\mathbb{R}^{n + 1}$ for
$n$ large enough.
Further, we have
$$
\left\{
\begin{aligned}\label{property3}
\ov  q_v^{(n)} &= F_n^T(\ov  q_v^{(n+1)}) \mbox{ for } v = 0,1,\ldots, n,
\\
\ov  q_n^{(n)} &= F_n^T(\ov  q_{n+1}^{(n+1)}).
\end{aligned}
\right.
$$
Observe that
\begin{equation}\label{property4}
d^*(\ov  q_v^{(n)},\ov  q_w^{(n)}) \geq (2C_n - 1) \to 1
\mbox{ as } n \rightarrow \infty
\end{equation}
whenever $v \neq w$, $v, w \in \{0,1,\ldots, n\}$.

Let $\mu$ be an ergodic invariant probability measure on $B = (V,E)$,
 and let
$$
\mu(X_w^{(n)}) = q_w^{(n)}, \ \ \ w = 0,1,\ldots, n.
$$
Denote $\ov q^{(n)} = (q_w^{(n)})_{w = 0}^{n}$.
Then $\ov q^{(n)}$ is an extreme point of $\Delta_{\infty}^{(n)}$, i.e. $\ov q^{(n)}
= \ov q^{(n)}_{w_n}$ for some $0 \leq w_n \leq n$ and each $n = 1,2,
\ldots$ Moreover, we have
$$
\ov q^{(n)}_{w_n} = F_n^T(\ov q^{(n)}_{w_{n+1}}).
$$
Then (except for some  initial part) the
 sequence $(w_0, w_1, w_2, \ldots)$ is one of the following sequences:
$\ov w_0 = (0,0,0,\ldots)$, $\ov w_1 = (0,1,1,\ldots)$, $\ov w_2 =
(0,1,2,2,\ldots)$, $\ldots$, $\ov w_{\infty} = (0,1,2,3, \ldots)$. Assume
that $w_n = 0$ for every $n = 0,1,\ldots$ Then
$$
\mu\left(X_0^{(n)}\right) = q_0^{(n)} \geq C_n \nearrow 1 \mbox{ as } n
\rightarrow \infty
$$ and
$$
\mu\left(\bigcup_{w \neq 0} X_w^{(n)}\right) \leq (1 - C_n) \searrow 0
\mbox{ as } n \rightarrow \infty.
$$
The above inequalities imply that $\mu(\wh X_{B_0}) = 1$ hence $\mu =
\wh \mu_0$. In the same way we show that $\mu = \wh \mu_i$, $i = 1,2,
\ldots, \infty$ if $(w_0, w_1, w_2, \ldots) = \ov w^{(i)}$ for $i = 1,2,
\ldots, \infty$.
\end{proof}

\section{Interpretation of the main theorems in terms of symbolic dynamics}\label{Section7}

Theorems~\ref{main1} and \ref{main_gen_case} can be interpreted in terms
of symbolic dynamics. In order to define a dynamical system
(which is known by the name of ``Vershik map'') on a
  Bratteli diagram, one needs to use a partial order on the set of edges.
The reader can find  details about the definition of dynamical systems
 on Bratteli diagrams  in \cite{HermanPutnamSkau1992},
\cite{GiordanoPutnamSkau1995}, or in \cite{Durand2010},
 \cite{BezuglyiKarpel2016}. More advanced study of ordered  Bratteli
 diagrams and the existence of corresponding dynamics are discussed in
 \cite{BezuglyiKwiatkowskiYassawi2014} and \cite{BezuglyiYassawi2017}.

For simplicity, throughout this section we will consider only simple properly
ordered Bratteli diagrams of finite rank which correspond to Cantor minimal
systems of topological rank $K > 1$. Recall that a Cantor minimal system has topological rank $K$ if it admits a Bratteli-Vershik representation with $|V_i| = K$ for every $i \geq 1$ and $K$ is the smallest such integer. If a Cantor minimal system has topological rank $K$ then it has at most $K$ probability ergodic invariant measures (see e.g. \cite{Durand2010}). If a system has $K$ probability ergodic invariant measures then the rank of the corresponding Bratteli diagram is at least $K$ (see \cite{BezuglyiKwiatkowskiMedynetsSolomyak2013}).

In \cite{DownarowiczMaass2008}, T. Downarowicz and A. Maass
proved that every Cantor minimal system of topological finite rank $K > 1$ is
expansive. This result was generalized in
\cite{BezuglyiKwiatkowskiMedynets2009} to aperiodic Cantor dynamical
systems of topological finite rank. Due to Hedlund~\cite{Hedlund1969},
every expansive Cantor dynamical system is conjugate to a subshift. Our
goal is to describe explicitly, how to code a Vershik map on a simple properly
ordered Bratteli diagram to obtain a conjugacy with a minimal subshift.

\subsection{From Bratteli diagrams to subshifts}
Let $\varphi_B \colon X_B \rightarrow X_B$ be a Vershik map on a simple
properly ordered Bratteli diagram $B$. Set $s_n = \sum_{w \in V_n}
h_{w}^{(n)}$ and $S_n = \{0, \ldots, s_n - 1\}$. Let $\{\ov e_0, \ldots,
 \ov e_{s_n - 1}\}$ be a set of all
finite paths between $v_0$ and vertices from $V_n$. Denote by $X^{(n)}$
the partition of $X_B$ into the corresponding cylinder sets $\{X^{(n)}(\ov
e_s)\}_{s = 0}^{s_n - 1}$.
 Then
\begin{equation}\label{nfactorsubshift}
(\psi_n(x))_i = s \Leftrightarrow \varphi_B^i(x) \in X^{(n)}(\ov e_s),
 \quad s\in S_n,\ i \in \mathbb{Z}.
\end{equation}
 is a factor map $\psi_n \colon X_B \rightarrow S_n^{\mathbb{Z}}$.
 Denote $Y_n = \psi_n(X_B)$,
and let $\sigma$ be the shift map. Then $(Y_n, \sigma)$ is a subshift of the
full shift $(S_n^{\mathbb{Z}}, \sigma)$. We note also that since the
sequence of  partitions
$X^{(n)}$ is nested, the system $(Y_n, \sigma)$ is a factor of the system
$(Y_{n+1}, \sigma)$ for every $n$. Moreover, $(X_B, \varphi_B)$ is
conjugate  to
the inverse limit of the systems $(Y_n, \sigma)$ for any ordered Bratteli
diagram that admits a continuous Vershik map. If $X^{(n)}$ is a generating
partition, then $\psi_n$ is a conjugacy  between $(X_B, \varphi_B)$ and
$(Y_n, \sigma)$.

Let $(X_B, \varphi_B)$ have topological finite rank $K > 1$, and let
$\delta$ be the expansivity constant for $(X_B, \varphi_B)$. Then, by
\cite{Hedlund1969}, every partition of $X_B$ into  clopen sets with
diameter less than $\delta$ is a generating partition. Hence, there exists a
natural number $n_0$ such that for all $n \geq n_0$ the systems $(X_B,
\varphi_B)$ and $(Y_n, \sigma)$ are conjugate.

\begin{remar}
In general, the rank of the diagram may be bigger than the rank of the
corresponding dynamical system. In \cite{FrickPetersenShields2017}, the
following criterion was given for an ordered Bratteli diagram to be isomorphic
to an odometer. It is said that the level $V_n$ of an ordered Bratteli
diagram is\textit{ uniformly ordered} if there exists a word $\omega$ over
the alphabet $V_{n-1}$ such that the coding of each vertex from $V_n$ in
terms of  $V_{n-1}$ is a concatenation of $\omega$.
\end{remar}

\begin{thm}[\cite{FrickPetersenShields2017}]
A simple properly ordered Bratteli diagram is topologically conjugate to an
odometer if and only if it has a telescoping for which there are infinitely many
uniformly ordered levels.
\end{thm}

Let $B = (V,E)$ be a simple properly ordered Bratteli diagram of topological
rank $K > 1$ such that the Vershik map $\varphi_B$ exists. We can assume
that $|V_n| = K$ for every $n \geq 1$. Choose $n_0$ such that $(X_B,
\varphi_B)$ is topologically conjugate to $(Y_n, \sigma)$ for all $n \geq
n_0$. To describe the symbolic shift $(Y_{n_0}, \sigma)$, we will need a
sequence $\{\mathcal{A}_n\}_{n \geq n_0}$ of families of blocks over
$S_{n_0}$. Each family $\mathcal{A}_n$ consists of $K$ blocks
$A_w^{(n)}$, $w \in V_n$. We define $\mathcal{A}_n$ inductively. First, to
define blocks $A_w^{(n_0)}$, $w \in V_{n_0}$, we write the set
$X_w^{(n_0)}$ in the form $\{\ov e_{i_1} < \ov e_{i_2} < \ldots < \ov
e_{i_w}\}$, where $\ov e_{i_1}, \ov e_{i_2}, \ldots, \ov e_{i_w}$ are all
paths between $v_0$ and $w$ which are compared with respect to the
lexicographical order. Set $A_w^{(n_0)} = (i_1, \ldots, i_w)$. Note that
$i_1, \ldots, i_w$ are symbols from $S_{n_0}$. Assume that $A_w^{(n)}$
are already defined for some $n \geq n_0$ and for all $w \in V_n$. Take $v
\in V_{n+1}$ and consider $r^{-1}(v) = \{e_1 < e_2 < \ldots < e_{|r^{-1}
(v)|}\}$. Let $w_i = s(e_i)$, $i = 1, \ldots, |r^{-1}(v)|$ be the vertices of
$V_n$ determined by the set $r^{-1}(v)$. Then we define a block $A_v^{(n
+1)}$ as the concatenation of the blocks $A_{w_1}^{(n)}, \ldots, A_{w_{|
r^{-1}(v)|}}^{(n)}$, i.e.

\begin{equation}\label{blockAvn}
A_v^{(n+1)} = A_{w_1}^{(n)} \ldots A_{w_{|r^{-1}(v)|}}^{(n)}.
\end{equation}

In this way we construct the sequence $\{\mathcal{A}_n\}_{n \geq
n_0}$. Now we define the language $L(\mathcal{A}_n)$ as the set of all
words which appear as factors of $A_w^{(n)}$ for $w \in V_n$, $n \geq
n_0$. Then we set
$$
Y_{n_0} = \{y \in S_{n_0}^{\mathbb{Z}} : y[-n,n] \in L(\{\mathcal{A}_n\})
\mbox{ for any } n \geq 1\}.
$$

\begin{remar}
We assume that $\tl f_{vw}^{(n)} \geq 1$ for each $v \in V_{n+1}$, $w \in
V_n$, and $n \geq 1$. This obviously  implies that $(X_B, \varphi_B)$ is
 minimal.
However, this does not guarantee that $B = (V,E)$ is proper, i.e. it has a
unique maximal path and a unique minimal path (of course, there exists a
 properly ordered
Bratteli diagram $B'$, possibly of different rank than $B$,  such that the Vershik maps $(X_B, \varphi_B)$
and $(X_{B'}, \varphi_{B'})$ are topologically conjugate, 
see \cite{HermanPutnamSkau1992} ).
\end{remar}

It follows from Proposition 3.2 in~\cite{BezuglyiKwiatkowskiYassawi2014},
that $B = (V,E)$ has the same number (say $k \leq K$) of minimal and
maximal paths. Denote them by $e^{(\min,i)}$ and $e^{(\max,i)}$
correspondingly, for $i = 1,\ldots, k$. There exists a permutation $\rho$ of
the set $\{1,\ldots, k\}$
such that $\rho(i) = j$ if and only if $\varphi_B(e^{(\max,i)}) = e^{(\min,j)}
$. Minimality of $(X_B, \varphi_B)$ implies that
$$
X_B = \overline{\{\varphi_B^s(e^{(\min,i)}), s \in \mathbb{Z}\}}
$$
and
$$
X_B = \overline{\{\varphi_B^s(e^{(\max,i)}), s \in \mathbb{Z}\}}
$$
for every $i = 1,\ldots,k$.
Thus the trajectories (with respect to the shift $\sigma$) of $y_{n_0}^{(i)}
= \psi_{n_0}(e^{(\min,i)})$ and $z_{n_0}^{(i)} = \psi_{n_0}(e^{(\max,i)})
$ are dense in $Y_{n_0}$ for $i = 1,\ldots,k$. To see how the sequences
$y_{n_0}^{(i)}$ and $z_{n_0}^{(i)}$ look, we write
$$
e^{(\min,i)} = (e^{(\min,i)}_1, e^{(\min,i)}_2, \ldots),
$$
$$
e^{(\max,i)} = (e^{(\max,i)}_1, e^{(\max,i)}_2, \ldots),
$$
and let $r(e^{(\min,i)}_s) = w_s^{(\min,i)}$, $r(e^{(\max,i)}_s) =
w_s^{(\max,i)}$ for $s = 1,2,\ldots$ Then $w_s^{(\min,i)}$ and
$w_s^{(\max,i)}$ are the consecutive vertices of the paths $e^{(\min,i)}$
and
$e^{(\max,i)}$, $i = 1,\ldots,k$ respectively. It follows from the definition of
$\psi_{n_0}$ (see~(\ref{nfactorsubshift})) that
$$
y_{n_0}^{(j)}\left[-h_s^{(\max,i)}, h_s^{(\min,j)} - 1\right] = A^{(s)}
_{w_s^{(\max,i)}}A^{(s)}_{w_s^{(\min,j)}}
$$
and
$$
z_{n_0}^{(i)} = \sigma^{(-1)}(y_{n_0}^{(j)}),
$$
where $j = \rho(i)$ and $h_s^{(\min,j)} = |A^{(s)}_{w_s^{(\min,j)}}|$,
$h_s^{(\max,j)} = |A^{(s)}_{w_s^{(\max,j)}}|$, $s = n_0, n_0 + 1, \ldots$

\begin{remar}
We can reformulate Theorems~\ref{main1} and \ref{main_gen_case} using
 the
families of blocks $A_w^{(n)}$, $w \in V_n$, $n \geq 1$. It follows from the
above considerations that the elements of the incidence matrices $\tl F_n$
are the numbers of occurrences of the blocks $A_w^{(n)}$ inside the blocks
$A_v^{(n+1)}$, $w \in V_n$, $v \in V_{n+1}$. More precisely, we have
$$
\tl f_{vw}^{(n)} = \mathrm{card}\left \{1 \leq i \leq |r^{-1}(v)|: A_{w_i}
^{(n)} = A_w^{(n)}\right\},
$$
where $A_{w_i}^{(n)}$ are the blocks appearing in the
concatenation~(\ref{blockAvn}). In this case we have $h_w^{(n)} = |
A_w^{(n)}|$, $w \in V_n$. Then Theorem~\ref{main1} reveals the structure
of the set of all invariant measures of $(Y_{n_0}, \sigma)$, while Theorem~\ref{main_gen_case} allows us to construct symbolic systems with
 uncountably many probability ergodic invariant measures. For more details about the connection between Bratteli diagrams and symbolic dynamical systems see e.g.~\cite{DurandHostSkau1999, BezuglyiKwiatkowskiMedynets2009, Durand2010, DownarowiczMaass2008, BezuglyiKarpel2014, DownarowiczKarpel2018, DownarowiczKarpel}.
\end{remar}

\subsection{Toeplitz flows}
The details about Toeplitz dynamical systems can be found in
\cite{Williams1984, DownarowiczKwiatkowskiLacroix1995,
BulatekKwiatkowski1990, Downarowicz1990}.
A \textit{Toeplitz dynamical system} is defined by a Toeplitz sequence
$\omega$
over a set of symbols $S$, $|S| \geq 2$. The sequence $\omega$ can be
obtained as a limit of periodic sequences $\omega_n$ over the alphabet $S
\cup \emptyset$, where $\emptyset$ denotes the empty symbol. Take a
sequence of natural numbers $\{\lambda_0, \lambda_1, \ldots\}$ such that
$\lambda_n \geq 2$ for all $n \geq 0$. Set $p_n = \lambda_0 \lambda_1
\cdots \lambda_n$, $n \geq 0$. By induction we define a sequence of blocks
$\{A_n, n \geq 0\}$ over the alphabet $S \cup \emptyset$ with $|A_n| =
p_n$. Each block $A_n$ should have some positions filled by the elements of
$S$ (these positions are called the filling positions) and some positions
occupied by the empty symbol $\emptyset$.
We take any block $A_0$ with $|A_0| = p_0 \geq 3$ such that $A_0[0]$
and $A_0[p_0-1]$ belong to $S$, and there is at least one symbol $
\emptyset$ among the symbols on the positions $\{1, \ldots, p_0 - 1\}$.
Assume that $A_n$ is defined. To define $A_{n+1}$, we first consider a
concatenation of $\lambda_{n+1}$ copies of $A_n$ and then fill some but
not all of the empty positions in this concatenation. Denote by $l_n$ and
$k_n$ the first and the last position in $A_n$, where the empty symbol $
\emptyset$ occurs. Then $l_n \leq k_n$ and $A_n[i] \in S$ whenever $0
\leq i \leq l_n - 1$ and $k_n +1 \leq i \leq p_n-1$. We require that $l_n$
and $(p_n - k_n)$ tend to infinity as $n$ tends to infinity.

For every $n$, we define a sequence $\omega_n$ as an infinite
concatenation of $A_n$ such that the block $A_n$ starts at a zero position.
Thus, each $\omega_n$ is a periodic two-sided sequence over $S \cup
\emptyset$ with the period $p_n$. Notice that the block $\omega_n[-k_n +
1, l_n - 1]$ is filled with the symbols from $S$. Now define a sequence $
\omega$ over the symbols $S$ in such a way that
\begin{equation}\label{def_omega}
\omega[-k_n + 1, l_n - 1] = \omega_n[-k_n + 1, l_n - 1]
\end{equation}
for $n = 0,1, \ldots$ Let us remark that $\omega$ is well defined since $
\omega_{n+1}[-k_n + 1, l_n - 1] = \omega_{n}[-k_n + 1, l_n - 1]$ and $l_n,
k_n \nearrow \infty$ as $n \rightarrow \infty$.

A sequence $\omega$ defined by~(\ref{def_omega}) is called a
 \textit{Toeplitz sequence }over $S$ whenever $\omega$ is not periodic.
 By a \textit{Toeplitz flow} we
mean a topological dynamical system $(\ov{\mathcal{O}(\omega)},
\sigma)$, where
$$
\ov{\mathcal{O}(\omega)} = \ov{\{\sigma^i(\omega), i \in \mathbb{Z}\}}
\subset S^{\mathbb{Z}}.
$$

 It is known how one can characterize Bratteli-Vershik systems associated to
  Toeplitz flows. The following theorem was proved  in  
  \cite{GjerdeJohansen2000}.
  
\begin{thm}
The family of expansive Bratteli-Vershik systems associated to simple Bratteli
diagrams with the ERS property coincides with the family of Toeplitz flows
up to conjugacy.
\end{thm}

Indeed, in order to construct a Bratteli diagram for a Toeplitz flow 
$(\ov{\mathcal{O}
(\omega)}, \sigma)$, we need a sequence $\{\mathcal{A}_n\}$ of families
of blocks over $S$ as follows:
$$
\mathcal{A}_n = \{\omega[mp_n, (m+1)p_n - 1], m \in \mathbb{Z}\}.
$$
Observe that each $\mathcal{A}_n$ is a finite family of blocks $\{A_1^{(n)},
\ldots, A_{s_n}^{(n)}\}$ and $|A_i^{(n)}| = p_n$ for every $i = 1,\ldots,
s_n$. We will call the blocks $A_i^{(n)}$ the $n$-symbols. Then every $(n
+1)$-symbol $A_j^{(n+1)}$ is a concatenation of $\lambda_{n+1}$ $n$-
symbols, i.e.
\begin{equation}\label{n-symbol}
A_j^{(n+1)} = A_{j_1}^{(n)}\ldots A_{j_{\lambda}}^{(n)},
\end{equation}
 where $\lambda = \lambda_n$.

The families $\mathcal{A}_n$ of $n$-symbols allow us to construct a Bratteli
diagram $B_{\omega} = (V_{\omega}, E_{\omega})$ as follows. Set $V_n =
\{1,\ldots,s_n\}$ and $n \geq 1$ and $V_0 = \{v_0\}$. For any $j \in V_{n
+1}$, we set $r^{-1}(j) = \{j_1 < j_2 < \ldots < j_{\lambda}\}$, where the
vertices $j_1, \ldots, j_{\lambda} \in V_n$ come from~(\ref{n-symbol}) and
``$<$'' means an order in $r^{-1}(j)$. Then the incidence matrices $\tl F_n$
are the matrices of appearances of $n$-symbols inside the $(n+1)$-symbols,
i.e.
$$
\tl f_{ji}^{(n)} = \mathrm{card}\{1 \leq r \leq \lambda_n: A_{j_r}^{(n)} =
A_{j_i}^{(n)}\},
$$
where $A^{(n)}_{j_1}, \ldots, A^{(n)}_{j_{\lambda}}$ come from
(\ref{n-symbol}). For the Toeplitz-Bratteli diagram $B_{\omega} = (V_{\omega},
E_{\omega})$ we have
$$
h_i^{(n)} = p_n
$$
for $i = 1,2,\ldots, s_n$. Then
$$
f_{ji}^{(n)} = \frac{h_i^{(n)}}{h_{j+1}^{(n)}} \tl f_{ji}^{(n)} = \frac{1}
{\lambda_{n+1}}\tl f_{ji}^{(n)} = \mathrm{fr}^{(n)}(A_i^{(n)}, A_j^{(n
+1)}).
$$
Now again we can formulate Theorem~\ref{main1} using the frequency
matrices $F_n$ and describe the set of all invariant ergodic measures of a
Toeplitz flow $(\ov{\mathcal{O}(\omega)}, \sigma)$.  Using Theorem~
\ref{main_gen_case} we can construct Toeplitz flows with uncountably many
ergodic invariant measures.

\medskip
\textbf{Acknowledgements.}
The research of the second author is supported by the NCN (National Science
Center, Poland) Grant 2013/08/A/ST1/00275. The authors would like to thank Tomasz Downarowicz, Palle E.T. Jorgensen, Alejandro Maass and Mariusz Lema\'nczyk for helpful discussions. The authors are grateful to
the Nicolas Copernicus University and the University of Iowa for the
hospitality and support. The authors thank the referee for detailed comments on the paper, which helped to make the exposition much better. 

\bibliographystyle{alpha}
\bibliography{referencesBKK}

\def\ocirc#1{\ifmmode\setbox0=\hbox{$#1$}\dimen0=\ht0 \advance\dimen0
  by1pt\rlap{\hbox to\wd0{\hss\raise\dimen0
  \hbox{\hskip.2em$\scriptscriptstyle\circ$}\hss}}#1\else {\accent"17 #1}\fi}
\begin{thebibliography}{ABKK17}

\bibitem[ABKK17]{AdamskaBezuglyiKarpelKwiatkowski2016}
M.~Adamska, S.~Bezuglyi, O.~Karpel, and J.~Kwiatkowski.
\newblock Subdiagrams and invariant measures on bratteli diagrams.
\newblock {\em Ergodic Theory Dynam. Systems}, 37(8):2417--2452, 2017.

\bibitem[BDK06]{BezuglyiDooleyKwiatkowski2006}
S.~Bezuglyi, A.~H. Dooley, and J.~Kwiatkowski.
\newblock Topologies on the group of {B}orel automorphisms of a standard
  {B}orel space.
\newblock {\em Topol. Methods Nonlinear Anal.}, 27(2):333--385, 2006.

\bibitem[BH14]{BezuglyiHandelman2014}
S.~Bezuglyi and D.~Handelman.
\newblock Measures on {C}antor sets: the good, the ugly, the bad.
\newblock {\em Trans. Amer. Math. Soc.}, 366(12):6247--6311, 2014.

\bibitem[BJ15]{BezuglyiJorgensen2015}
S.~Bezuglyi and Palle E.~T. Jorgensen.
\newblock Representations of {C}untz-{K}rieger relations, dynamics on
  {B}ratteli diagrams, and path-space measures.
\newblock In {\em Trends in harmonic analysis and its applications}, volume 650
  of {\em Contemp. Math.}, pages 57--88. Amer. Math. Soc., Providence, RI,
  2015.

\bibitem[BK90]{BulatekKwiatkowski1990}
W.~Bulatek and J.~Kwiatkowski.
\newblock The topological centralizers of {T}oeplitz flows and their
  $\mathbb{Z}_2$-extensions.
\newblock {\em Publ. Mat.}, 34:45--65, 1990.

\bibitem[BK11]{BezuglyiKarpel2011}
S.~Bezuglyi and O.~Karpel.
\newblock Homeomorphic measures on stationary {B}ratteli diagrams.
\newblock {\em J. Funct. Anal.}, 261(12):3519--3548, 2011.

\bibitem[BK14]{BezuglyiKarpel2014}
S.~Bezuglyi and O.~Karpel.
\newblock Orbit equivalent substitution dynamical systems and complexity.
\newblock {\em Proc. Amer. Math. Soc.}, 142(12):4155--4169, 2014.

\bibitem[BK16]{BezuglyiKarpel2016}
S.~Bezuglyi and O.~Karpel.
\newblock Bratteli diagrams: structure, measures, dynamics.
\newblock In {\em Dynamics and numbers}, volume 669 of {\em Contemp. Math.},
  pages 1--36. Amer. Math. Soc., Providence, RI, 2016.

\bibitem[BKM09]{BezuglyiKwiatkowskiMedynets2009}
S.~Bezuglyi, J.~Kwiatkowski, and K.~Medynets.
\newblock Aperiodic substitution systems and their {B}ratteli diagrams.
\newblock {\em Ergodic Theory Dynam. Systems}, 29(1):37--72, 2009.

\bibitem[BKMS10]{BezuglyiKwiatkowskiMedynetsSolomyak2010}
S.~Bezuglyi, J.~Kwiatkowski, K.~Medynets, and B.~Solomyak.
\newblock Invariant measures on stationary {B}ratteli diagrams.
\newblock {\em Ergodic Theory Dynam. Systems}, 30(4):973--1007, 2010.

\bibitem[BKMS13]{BezuglyiKwiatkowskiMedynetsSolomyak2013}
S.~Bezuglyi, J.~Kwiatkowski, K.~Medynets, and B.~Solomyak.
\newblock Finite rank {B}ratteli diagrams: structure of invariant measures.
\newblock {\em Trans. Amer. Math. Soc.}, 365(5):2637--2679, 2013.

\bibitem[BKY14]{BezuglyiKwiatkowskiYassawi2014}
S.~Bezuglyi, J.~Kwiatkowski, and R.~Yassawi.
\newblock Perfect orderings on finite rank {B}ratteli diagrams.
\newblock {\em Canad. J. Math.}, 66(1):57--101, 2014.

\bibitem[Bra72]{Bratteli1972}
O.~Bratteli.
\newblock Inductive limits of finite dimensional {$C^{\ast} $}-algebras.
\newblock {\em Trans. Amer. Math. Soc.}, 171:195--234, 1972.

\bibitem[BY17]{BezuglyiYassawi2017}
Sergey Bezuglyi and Reem Yassawi.
\newblock Orders that yield homeomorphisms on {B}ratteli diagrams.
\newblock {\em Dyn. Syst.}, 32(2):249--282, 2017.

\bibitem[Cor06]{Cortez2006}
Maria~Isabel Cortez.
\newblock Realization of a {C}hoquet simplex as the set of invariant
  probability measures of a tiling system.
\newblock {\em Ergodic Theory Dynam. Systems}, 26(5):1417--1441, 2006.

\bibitem[DHS99]{DurandHostSkau1999}
F.~Durand, B.~Host, and C.~Skau.
\newblock Substitutional dynamical systems, {B}ratteli diagrams and dimension
  groups.
\newblock {\em Ergodic Theory Dynam. Systems}, 19(4):953--993, 1999.

\bibitem[DK]{DownarowiczKarpel}
T.~Downarowicz and O.~Karpel.
\newblock Decisive {B}ratteli-{V}ershik models.
\newblock {\em Studia Math.}, doi: 10.4064/sm170519-5-2.

\bibitem[DK18]{DownarowiczKarpel2018}
T.~Downarowicz and O.~Karpel.
\newblock Dynamics in dimension zero: a survey.
\newblock {\em Discrete Contin. Dyn. Syst.}, 38(3):1033--1062, 2018.

\bibitem[DKL95]{DownarowiczKwiatkowskiLacroix1995}
T.~Downarowicz, J.~Kwiatkowski, and Y.~Lacroix.
\newblock A criterion for {T}oeplitz flows to be topologically isomorphic and
  applications.
\newblock {\em Colloq. Math.}, 68:219--228, 1995.

\bibitem[DM08]{DownarowiczMaass2008}
T.~Downarowicz and A.~Maass.
\newblock Finite rank {B}ratteli-{V}ershik diagrams are expansive.
\newblock {\em Ergodic Theory Dynam. Systems}, 28:739--747, 2008.

\bibitem[Dow90]{Downarowicz1990}
T.~Downarowicz.
\newblock How a function on a zero-dimensional group {$\Delta_a$} defines a
  {T}oeplitz flow.
\newblock {\em Bull. Polish Acad. Sci.}, 38:219--222, 1990.

\bibitem[Dow91]{Downarowicz1991}
T.~Downarowicz.
\newblock The {C}hoquet simplex of invariant measures for minimal flows.
\newblock {\em Israel J. Math.}, 74(2-3):241--256, 1991.

\bibitem[Dow06]{Downarowicz2006}
T.~Downarowicz.
\newblock Minimal models for noninvertible and not uniquely ergodic systems.
\newblock {\em Israel J. Math.}, 156:93--110, 2006.

\bibitem[Dow08]{Downarowicz2008}
T.~Downarowicz.
\newblock Faces of simplexes of invariant measures.
\newblock {\em Israel J. Math.}, 165:189--210, 2008.

\bibitem[Dur10]{Durand2010}
Fabien Durand.
\newblock Combinatorics on {B}ratteli diagrams and dynamical systems.
\newblock In {\em Combinatorics, automata and number theory}, volume 135 of
  {\em Encyclopedia Math. Appl.}, pages 324--372. Cambridge Univ. Press,
  Cambridge, 2010.

\bibitem[FFT09]{FerencziFisherTalet2009}
Sebastien Ferenczi, Albert~M. Fisher, and Marina Talet.
\newblock Minimality and unique ergodicity for adic transformations.
\newblock {\em J. Anal. Math.}, 109:1--31, 2009.

\bibitem[FH17]{FrantzikinakisHost2017}
N.~{Frantzikinakis} and B.~{Host}.
\newblock {The logarithmic Sarnak conjecture for ergodic weights}.
\newblock {\em ArXiv e-prints}, August 2017.

\bibitem[FPS17]{FrickPetersenShields2017}
Sarah Frick, Karl Petersen, and Sandy Shields.
\newblock Dynamical properties of some adic systems with arbitrary orderings.
\newblock {\em Ergodic Theory Dynam. Systems}, 37(7):2131--2162, 2017.

\bibitem[GJ00]{GjerdeJohansen2000}
Richard Gjerde and {\O}rjan Johansen.
\newblock Bratteli-{V}ershik models for {C}antor minimal systems: applications
  to {T}oeplitz flows.
\newblock {\em Ergodic Theory Dynam. Systems}, 20(6):1687--1710, 2000.

\bibitem[Gla03]{Glasner2003}
Eli Glasner.
\newblock {\em Ergodic theory via joinings}, volume 101 of {\em Mathematical
  Surveys and Monographs}.
\newblock American Mathematical Society, Providence, RI, 2003.

\bibitem[GMPS10]{GiordanoMatuiPutnamSkau2010}
Thierry Giordano, Hiroki Matui, Ian~F. Putnam, and Christian~F. Skau.
\newblock Orbit equivalence for {C}antor minimal {$\Bbb Z^d$}-systems.
\newblock {\em Invent. Math.}, 179(1):119--158, 2010.

\bibitem[GPS95]{GiordanoPutnamSkau1995}
Thierry Giordano, Ian~F. Putnam, and Christian~F. Skau.
\newblock Topological orbit equivalence and {$C^*$}-crossed products.
\newblock {\em J. Reine Angew. Math.}, 469:51--111, 1995.

\bibitem[Hed69]{Hedlund1969}
Gustav~A. Hedlund.
\newblock Endomorphisms and automorphisms of the shift dynamical systems.
\newblock {\em Math. Syst. Theory}, 3:320--375, 1969.

\bibitem[HKY11]{HamachiKeaneYuasa2011}
Toshihiro Hamachi, Michael~S. Keane, and Hisatoshi Yuasa.
\newblock Universally measure-preserving homeomorphisms of {C}antor minimal
  systems.
\newblock {\em J. Anal. Math.}, 113:1--51, 2011.

\bibitem[HPS92]{HermanPutnamSkau1992}
Richard~H. Herman, Ian~F. Putnam, and Christian~F. Skau.
\newblock Ordered {B}ratteli diagrams, dimension groups and topological
  dynamics.
\newblock {\em Internat. J. Math.}, 3(6):827--864, 1992.

\bibitem[KW04]{KwiatkowskiWata2004}
Jan Kwiatkowski and Marcin Wata.
\newblock Dimension and infinitesimal groups of {C}antor minimal systems.
\newblock {\em Topol. Methods Nonlinear Anal.}, 23(1):161--202, 2004.

\bibitem[Med07]{Medynets2007}
Konstantin Medynets.
\newblock On approximation of homeomorphisms of a {C}antor set.
\newblock {\em Fund. Math.}, 194(1):1--13, 2007.

\bibitem[MP05]{MelaPetersen2005}
Xavier M\'ela and Karl Petersen.
\newblock Dynamical properties of the {P}ascal adic transformation.
\newblock {\em Ergodic Theory Dynam. Systems}, 25(1):227--256, 2005.

\bibitem[Phe01]{Phelps2001}
Robert~R. Phelps.
\newblock {\em Lectures on {C}hoquet's theorem}, volume 1757 of {\em Lecture
  Notes in Mathematics}.
\newblock Springer-Verlag, Berlin, second edition, 2001.

\bibitem[Pul71]{Pullman1971}
N.~J. Pullman.
\newblock A geometric approach to the theory of nonnegative matrices.
\newblock {\em Linear Algebra and Appl.}, 4:297--312, 1971.

\bibitem[Put10]{Putnam2010}
Ian~F. Putnam.
\newblock Orbit equivalence of {C}antor minimal systems: a survey and a new
  proof.
\newblock {\em Expo. Math.}, 28(2):101--131, 2010.

\bibitem[Ska00]{Skau2000}
Christian Skau.
\newblock Ordered {$K$}-theory and minimal symbolic dynamical systems.
\newblock {\em Colloq. Math.}, 84/85(part 1):203--227, 2000.
\newblock Dedicated to the memory of Anzelm Iwanik.

\bibitem[Ver81]{Vershik1981}
A.~M. Vershik.
\newblock Uniform algebraic approximation of shift and multiplication
  operators.
\newblock {\em Dokl. Akad. Nauk SSSR}, 259(3):526--529, 1981.

\bibitem[Ver82]{Vershik1982}
A.~M. Vershik.
\newblock A theorem on {M}arkov periodic approximation in ergodic theory.
\newblock {\em Zap. Nauchn. Sem. Leningrad. Otdel. Mat. Inst. Steklov. (LOMI)},
  115:72--82, 306, 1982.
\newblock Boundary value problems of mathematical physics and related questions
  in the theory of functions, 14.

\bibitem[Ver11]{Vershik_2011}
A.~M. Vershik.
\newblock The {P}ascal automorphism has a continuous spectrum.
\newblock {\em Funktsional. Anal. i Prilozhen.}, 45(3):16--33, 2011.

\bibitem[Ver14]{Vershik_2014}
A.~M. Vershik.
\newblock Intrinsic metric on graded graphs, standardness, and invariant
  measures.
\newblock {\em Zap. Nauchn. Sem. S.-Peterburg. Otdel. Mat. Inst. Steklov.
  (POMI)}, 421(Teoriya Predstavleni\u\i , Dinamicheskie Sistemy, Kombinatornye
  Metody. XXIII):58--67, 2014.

\bibitem[Wil84]{Williams1984}
S.~Williams.
\newblock Toeplitz minimal flows which are not uniquely ergodic.
\newblock {\em Z. Wahrscheinlichkeitstheorie verw Gebiete}, 67:95--107, 1984.

\end{thebibliography}

\end{document}